\documentclass[twoside]{amsart}
\usepackage{url}

\usepackage[unicode,bookmarks]{hyperref}
\usepackage[usenames,dvipsnames]{xcolor}
\hypersetup{colorlinks=true,citecolor=NavyBlue,linkcolor=BrickRed,urlcolor=Orange}

\usepackage{endnotes,hyperendnotes}
\let\footnote=\endnote
\renewcommand{\theendnote}{[\arabic{endnote}]}

\newif\ifdraft
\draftfalse
\usepackage{amsmath,amsfonts,amsthm,mathrsfs}
\usepackage{amssymb}

\usepackage[alphabetic,initials]{amsrefs}

\usepackage{enumitem}
\newenvironment{renumerate}{%
	\begin{enumerate}[label=(\roman{*}), ref=(\roman{*})]
}{%
	\end{enumerate}%
}

\newenvironment{aenumerate}{%
	\begin{enumerate}[label=(\alph{*}), ref=(\alph{*})]
}{%
	\end{enumerate}%
}

\usepackage{chngcntr}

\ifdraft
\usepackage[notcite,notref,color]{showkeys}

\definecolor{labelkey}{gray}{0.5}
\fi

\usepackage{tikz}

\usepackage{tikz-cd}
\tikzset{commutative diagrams/arrow style=math font}

\usetikzlibrary{matrix,arrows}
\newlength{\myarrowsize} 

\pgfarrowsdeclare{cmto}{cmto}{
	\pgfsetdash{}{0pt} 
	\pgfsetbeveljoin 
	\pgfsetroundcap 
	\setlength{\myarrowsize}{0.6pt}
	\addtolength{\myarrowsize}{.5\pgflinewidth}
	\pgfarrowsleftextend{-4\myarrowsize-.5\pgflinewidth} 
	\pgfarrowsrightextend{.8\pgflinewidth}
}{
	\setlength{\myarrowsize}{0.6pt} 
  	\addtolength{\myarrowsize}{.5\pgflinewidth}  
	\pgfsetlinewidth{0.5\pgflinewidth}
	\pgfsetroundjoin
	\pgfpathmoveto{\pgfpoint{1.5\pgflinewidth}{0}}
	\pgfpatharc{-109}{-170}{4\myarrowsize}
	\pgfpatharc{10}{189}{0.58\pgflinewidth and 0.2\pgflinewidth}
	\pgfpatharc{-170}{-115}{4\myarrowsize+\pgflinewidth}
	\pgfpathclose
	\pgfusepathqfillstroke
	\pgfpathmoveto{\pgfpoint{1.5\pgflinewidth}{0}}
	\pgfpatharc{109}{170}{4\myarrowsize}
	\pgfpatharc{-10}{-189}{0.58\pgflinewidth and 0.2\pgflinewidth}
	\pgfpatharc{170}{115}{4\myarrowsize+\pgflinewidth}
	\pgfpathclose
	\pgfusepathqfillstroke
	\pgfsetlinewidth{2\pgflinewidth}
}

\pgfarrowsdeclare{cmonto}{cmonto}{
	\pgfsetdash{}{0pt} 
	\pgfsetbeveljoin 
	\pgfsetroundcap 
	\setlength{\myarrowsize}{0.6pt}
	\addtolength{\myarrowsize}{.5\pgflinewidth}
	\pgfarrowsleftextend{-4\myarrowsize-.5\pgflinewidth} 
	\pgfarrowsrightextend{.8\pgflinewidth}
}{
	\setlength{\myarrowsize}{0.6pt} 
  	\addtolength{\myarrowsize}{.5\pgflinewidth}  
	\pgfsetlinewidth{0.5\pgflinewidth}
	\pgfsetroundjoin
	\pgfpathmoveto{\pgfpoint{1.5\pgflinewidth}{0}}
	\pgfpatharc{-109}{-170}{4\myarrowsize}
	\pgfpatharc{10}{189}{0.58\pgflinewidth and 0.2\pgflinewidth}
	\pgfpatharc{-170}{-115}{4\myarrowsize+\pgflinewidth}
	\pgfpathclose
	\pgfusepathqfillstroke
	\pgfpathmoveto{\pgfpoint{1.5\pgflinewidth}{0}}
	\pgfpatharc{109}{170}{4\myarrowsize}
	\pgfpatharc{-10}{-189}{0.58\pgflinewidth and 0.2\pgflinewidth}
	\pgfpatharc{170}{115}{4\myarrowsize+\pgflinewidth}
	\pgfpathclose
	\pgfusepathqfillstroke
	\pgfpathmoveto{\pgfpoint{1.5\pgflinewidth-0.3em}{0}}
	\pgfpatharc{-109}{-170}{4\myarrowsize}
	\pgfpatharc{10}{189}{0.58\pgflinewidth and 0.2\pgflinewidth}
	\pgfpatharc{-170}{-115}{4\myarrowsize+\pgflinewidth}
	\pgfpathclose
	\pgfusepathqfillstroke
	\pgfpathmoveto{\pgfpoint{1.5\pgflinewidth-0.3em}{0}}
	\pgfpatharc{109}{170}{4\myarrowsize}
	\pgfpatharc{-10}{-189}{0.58\pgflinewidth and 0.2\pgflinewidth}
	\pgfpatharc{170}{115}{4\myarrowsize+\pgflinewidth}
	\pgfpathclose
	\pgfusepathqfillstroke
	\pgfsetlinewidth{2\pgflinewidth}
}

\pgfarrowsdeclare{cmhook}{cmhook}{
	\pgfsetdash{}{0pt} 
	\pgfsetbeveljoin 
	\pgfsetroundcap 
	\setlength{\myarrowsize}{0.6pt}
	\addtolength{\myarrowsize}{.5\pgflinewidth}
	\pgfarrowsleftextend{-4\myarrowsize-.5\pgflinewidth} 
	\pgfarrowsrightextend{.8\pgflinewidth}
}{
	\setlength{\myarrowsize}{0.6pt} 
  	\addtolength{\myarrowsize}{.5\pgflinewidth}  
 	\pgfsetdash{}{0pt}
	\pgfsetroundcap
	\pgfpathmoveto{\pgfqpoint{0pt}{-4.667\pgflinewidth}}
	\pgfpathcurveto
    {\pgfqpoint{4\pgflinewidth}{-4.667\pgflinewidth}}
    {\pgfqpoint{4\pgflinewidth}{0pt}}
    {\pgfpointorigin}
	\pgfusepathqstroke
}

\newenvironment{diagram*}[2]{%
\[%
\begin{tikzpicture}[>=cmto,baseline=(current bounding box.center),%
	to/.style={->,font=\scriptsize,cap=round},%
	into/.style={cmhook->,font=\scriptsize,cap=round},%
	onto/.style={-cmonto,font=\scriptsize,cap=round},%
	math/.style={matrix of math nodes, row sep=#2, column sep=#1,%
		text height=1.5ex, text depth=0.25ex}]%
}{%
\end{tikzpicture}%
\]%
\ignorespacesafterend%
}

\newcommand{\MHM}{\operatorname{MHM}}

\newcommand{\pt}{\mathit{pt}}
\newcommand{\Dmod}{\mathscr{D}}
\newcommand{\Mmod}{\mathcal{M}}
\newcommand{\Nmod}{\mathcal{N}}

\newcommand{\shT}{\mathscr{T}}

\newcommand{\derR}{\mathbf{R}}
\newcommand{\derL}{\mathbf{L}}

\newcommand{\Ltensor}{\overset{\derL}{\tensor}}
\newcommand{\decal}[1]{\lbrack #1 \rbrack}

\newcommand{\shH}{\mathcal{H}}

\newcommand{\abs}[1]{\lvert #1 \rvert}

\newcommand{\eps}{\varepsilon}

\newcommand{\tensor}{\otimes}
\newcommand{\del}{\partial}

\renewcommand{\dbar}{\bar{\partial}}

\newcommand{\shHom}{\mathcal{H}\hspace{-1pt}\mathit{om}}

\newcommand{\ZZ}{\mathbb{Z}}
\newcommand{\QQ}{\mathbb{Q}}
\newcommand{\RR}{\mathbb{R}}
\newcommand{\CC}{\mathbb{C}}
\newcommand{\HH}{\mathbb{H}}

\newcommand{\menge}[2]{\bigl\{ \thinspace #1 \thinspace\thinspace \big\vert%
\thinspace\thinspace #2 \thinspace \bigr\}}
\newcommand{\Menge}[2]{\Bigl\{ \thinspace #1 \thinspace\thinspace \Big\vert%
\thinspace\thinspace #2 \thinspace \Bigr\}}
\newcommand{\mengge}[2]{\biggl\{ \thinspace #1 \thinspace\thinspace \bigg\vert%
\thinspace\thinspace #2 \thinspace \biggr\}}

\DeclareMathOperator{\im}{im}
\DeclareMathOperator{\rk}{rk}

\DeclareMathOperator{\Spec}{Spec}

\DeclareMathOperator{\id}{id}

\DeclareMathOperator{\Supp}{Supp}
\DeclareMathOperator{\codim}{codim}

\DeclareMathOperator{\rat}{rat}

\DeclareMathOperator{\Sym}{Sym}
\DeclareMathOperator{\gr}{gr}
\DeclareMathOperator{\DR}{DR}

\DeclareMathOperator{\Hom}{Hom}
\DeclareMathOperator{\Ext}{Ext}
\DeclareMathOperator{\GL}{GL}

\DeclareMathOperator{\Ch}{Ch}

\DeclareMathOperator{\Pic}{Pic}

\newcommand{\define}[1]{\emph{#1}}

\newcommand{\lie}[2]{\lbrack #1, #2 \rbrack}

\newcommand{\SU}[1]{\mathrm{SU}(#1)}

\newcommand{\shf}[1]{\mathscr{#1}}
\newcommand{\OX}{\shf{O}_X}
\newcommand{\OmX}{\Omega_X}

\newcommand{\argbl}{-}

\def\overbar#1#2#3{{%
	\setbox0=\hbox{\displaystyle{#1}}%
	\dimen0=\wd0
	\advance\dimen0 by -#2 
	\vbox {\nointerlineskip \moveright #3 \vbox{\hrule height 0.3pt width \dimen0}%
		\nointerlineskip \vskip 1.5pt \box0}%
}}

\newcommand{\into}{\hookrightarrow}

\newcommand{\pil}{\pi_{\ast}}
\newcommand{\jl}{j_{\ast}}

\newcommand{\ju}{j^{\ast}}
\newcommand{\fu}{f^{\ast}}
\newcommand{\fl}{f_{\ast}}

\newcommand{\iu}{i^{\ast}}

\newcommand{\il}{i_{\ast}}

\newcommand{\pl}{p_{\ast}}

\newcommand{\piu}{\pi^{\ast}}

\newcommand{\fp}{f_{+}}

\newcommand{\shF}{\shf{F}}

\newcommand{\shE}{\shf{E}}

\newcommand{\shO}{\shf{O}}

\makeatletter
\let\@@seccntformat\@seccntformat
\renewcommand*{\@seccntformat}[1]{%
  \expandafter\ifx\csname @seccntformat@#1\endcsname\relax
    \expandafter\@@seccntformat
  \else
    \expandafter
      \csname @seccntformat@#1\expandafter\endcsname
  \fi
    {#1}%
}
\newcommand*{\@seccntformat@subsection}[1]{%
  \textbf{\csname the#1\endcsname.}
}
\makeatother

\makeatletter
\let\@paragraph\paragraph
\renewcommand*{\paragraph}[1]{%
	\vspace{0.3\baselineskip}%
	\@paragraph{\textit{#1}}%
}
\makeatother

\counterwithin{equation}{subsection}
\counterwithout{subsection}{section}
\counterwithin{figure}{subsection}

\newtheorem{theorem}[equation]{Theorem}
\newtheorem*{theorem*}{Theorem}
\newtheorem{lemma}[equation]{Lemma}
\newtheorem*{lemma*}{Lemma}
\newtheorem{corollary}[equation]{Corollary}
\newtheorem{proposition}[equation]{Proposition}
\newtheorem*{proposition*}{Proposition}
\newtheorem{conjecture}[equation]{Conjecture}

\theoremstyle{definition}
\newtheorem{definition}[equation]{Definition}
\newtheorem*{definition*}{Definition}
\theoremstyle{remark}

\newtheorem*{question}{Question}

\newtheorem{example}[equation]{Example}
\newtheorem*{example*}{Example}
\newtheorem*{problem*}{Problem}
\newtheorem*{note}{Note}

\theoremstyle{plain}

\newcommand{\theoremref}[1]{\hyperref[#1]{Theorem~\ref*{#1}}}
\newcommand{\lemmaref}[1]{\hyperref[#1]{Lemma~\ref*{#1}}}
\newcommand{\definitionref}[1]{\hyperref[#1]{Definition~\ref*{#1}}}
\newcommand{\propositionref}[1]{\hyperref[#1]{Proposition~\ref*{#1}}}
\newcommand{\conjectureref}[1]{\hyperref[#1]{Conjecture~\ref*{#1}}}
\newcommand{\corollaryref}[1]{\hyperref[#1]{Corollary~\ref*{#1}}}
\newcommand{\exampleref}[1]{\hyperref[#1]{Example~\ref*{#1}}}

\makeatletter
\let\old@caption\caption
\renewcommand*{\caption}[1]{%
	\setcounter{figure}{\value{equation}}%
	\stepcounter{equation}%
	\old@caption{#1}\relax%
}
\makeatother

\newcounter{intro}

\newtheorem{intro-conjecture}[intro]{Conjecture}
\newtheorem{intro-corollary}[intro]{Corollary}
\newtheorem{intro-theorem}[intro]{Theorem}

\newcommand{\OA}{\mathscr{O}_A}
\newcommand{\OmA}{\Omega_A}

\newcommand{\Ah}{\hat{A}}

\newcommand{\parref}[1]{\hyperref[#1]{\S\ref*{#1}}}

\makeatletter
\newcommand*\if@single[3]{%
  \setbox0\hbox{${\mathaccent"0362{#1}}^H$}%
  \setbox2\hbox{${\mathaccent"0362{\kern0pt#1}}^H$}%
  \ifdim\ht0=\ht2 #3\else #2\fi
  }
\newcommand*\rel@kern[1]{\kern#1\dimexpr\macc@kerna}
\newcommand*\widebar[1]{\@ifnextchar^{{\wide@bar{#1}{0}}}{\wide@bar{#1}{1}}}
\newcommand*\wide@bar[2]{\if@single{#1}{\wide@bar@{#1}{#2}{1}}{\wide@bar@{#1}{#2}{2}}}
\newcommand*\wide@bar@[3]{%
  \begingroup
  \def\mathaccent##1##2{%
    \if#32 \let\macc@nucleus\first@char \fi
    \setbox\z@\hbox{$\macc@style{\macc@nucleus}_{}$}%
    \setbox\tw@\hbox{$\macc@style{\macc@nucleus}{}_{}$}%
    \dimen@\wd\tw@
    \advance\dimen@-\wd\z@
    \divide\dimen@ 3
    \@tempdima\wd\tw@
    \advance\@tempdima-\scriptspace
    \divide\@tempdima 10
    \advance\dimen@-\@tempdima
    \ifdim\dimen@>\z@ \dimen@0pt\fi
    \rel@kern{0.6}\kern-\dimen@
    \if#31
      \overline{\rel@kern{-0.6}\kern\dimen@\macc@nucleus\rel@kern{0.4}\kern\dimen@}%
      \advance\dimen@0.4\dimexpr\macc@kerna
      \let\final@kern#2%
      \ifdim\dimen@<\z@ \let\final@kern1\fi
      \if\final@kern1 \kern-\dimen@\fi
    \else
      \overline{\rel@kern{-0.6}\kern\dimen@#1}%
    \fi
  }%
  \macc@depth\@ne
  \let\math@bgroup\@empty \let\math@egroup\macc@set@skewchar
  \mathsurround\z@ \frozen@everymath{\mathgroup\macc@group\relax}%
  \macc@set@skewchar\relax
  \let\mathaccentV\macc@nested@a
  \if#31
    \macc@nested@a\relax111{#1}%
  \else
    \def\gobble@till@marker##1\endmarker{}%
    \futurelet\first@char\gobble@till@marker#1\endmarker
    \ifcat\noexpand\first@char A\else
      \def\first@char{}%
    \fi
    \macc@nested@a\relax111{\first@char}%
  \fi
  \endgroup
}
\makeatother

\newcommand{\subsecref}[1]{\hyperref[#1]{Section~\ref*{#1}}}

\newcommand{\OAh}{\shf{O}_{\Ah}}
\newcommand{\OAsh}{\shf{O}_{\Ash}}
\newcommand{\OBsh}{\shf{O}_{\Bsh}}

\newcommand{\Ash}{A^{\natural}}

\newcommand{\Bsh}{B^{\natural}}
\newcommand{\fsh}{f^{\natural}}
\newcommand{\fsi}{f^{+}}
\newcommand{\isi}{i^{+}}

\newcommand{\Dbcoh}{\mathrm{D}_{\mathit{coh}}^{\mathit{b}}}

\newcommand{\Db}{\mathrm{D}^{\mathit{b}}}
\newcommand{\Dtcoh}[1]{\mathrm{D}_{\mathit{coh}}^{#1}}

\newcommand{\Coh}{\operatorname{Coh}}

\newcommand{\pDtcoh}[2]{ {^{#1}} \Dtcoh{#2}}
\newcommand{\pCoh}[2]{ {^{#1}} \mathrm{Coh}{#2}}
\newcommand{\piDtc}[1]{ {^\pi} \mathrm{D}_c^{#1}}

\newcommand{\mCoh}{ {^m} \mathrm{Coh}}
\newcommand{\mDtcoh}[1]{ {^m} \Dtcoh{#1}}
\newcommand{\mhCoh}{ {^{\mh}} \mathrm{Coh}}
\newcommand{\mhDtcoh}[1]{ {^{\mh}} \Dtcoh{#1}}
\newcommand{\mh}{\hat{m}}

\newcommand{\Dbh}{\Db_{\mathit{h}}}
\newcommand{\Dbc}{\Db_{\mathit{c}}}
\newcommand{\Dbrh}{\Db_{\mathit{rh}}}
\newcommand{\Modh}{\operatorname{Mod}_{\mathit{h}}}

\newcommand{\Dth}[1]{\mathrm{D}_h^{#1}}

\renewcommand{\Hom}{\operatorname{Hom}}
\renewcommand{\Ext}{\operatorname{Ext}}
\renewcommand{\shHom}{\mathcal{H}\mathit{om}}

\DeclareMathOperator{\IC}{IC}

\newcommand{\omX}{\omega_X}

\newcommand{\CH}{\operatorname{\mathit{CH}}}

\newcommand{\shA}{\mathcal{A}}

\newcommand{\mm}{\mathfrak{m}}

\renewcommand{\dbar}{\bar{\partial}}

\DeclareMathOperator{\FMt}{\widetilde{FM}}
\newcommand{\shL}{\mathcal{L}}
\newcommand{\shLC}{\shL_{\CLambda}}

\newcommand{\shLk}{\shL_{\kLambda}}
\newcommand{\shLCc}{\shL_{\CcLambda}}
\newcommand{\mmrho}{\mm_{\rho}}
\renewcommand{\shH}{\mathcal{H}}
\newcommand{\derH}{\mathbf{H}}

\newcommand{\shI}{\mathcal{I}}

\renewcommand{\Dmod}{\mathscr{D}}
\newcommand{\Rmod}{\mathscr{R}}

\newcommand{\Pg}{\widetilde{P}}
\newcommand{\nablag}{\widetilde{\nabla}}

\DeclareMathOperator{\Aut}{Aut}
\newcommand{\Asig}{A^{\sigma}}
\newcommand{\Asigsh}{(A^{\sigma})^{\natural}}
\newcommand{\Lsig}{L^{\sigma}}
\newcommand{\nablasig}{\nabla^{\sigma}}
\newcommand{\csig}{c_{\sigma}}

\newcommand{\Msig}{M_{\sigma}}
\newcommand{\Mmodsig}{\Mmod_{\sigma}}

\newcommand{\Mmodt}{\widetilde{M}}
\newcommand{\ReesM}{\widetilde{R_F \Mmod}}

\newcommand{\CLambda}{\CC \lbrack \Lambda \rbrack}
\newcommand{\kLambda}{k \lbrack \Lambda \rbrack}
\newcommand{\CcLambda}{\CC \{ \Lambda \}}

\newcommand{\cMmod}{\Mmod^{\bullet}}

\newcommand{\shP}{\mathcal{P}}

\renewcommand{\derR}{\mathbf{R}}
\renewcommand{\derL}{\mathbf{L}}
\renewcommand{\Ltensor}{\stackrel{\derL}{\tensor}}
\renewcommand{\argbl}{-}

\renewcommand{\Hom}{\operatorname{Hom}}
\renewcommand{\Ext}{\operatorname{Ext}}
\renewcommand{\shHom}{\mathcal{H}\mathit{om}}
\newcommand{\DA}{\mathbf{D}_A}

\DeclareMathOperator{\Char}{Char}
\newcommand{\CCrho}{\CC_{\rho}}
\newcommand{\Psh}{P^{\natural}}
\newcommand{\nablash}{\nabla^{\natural}}
\DeclareMathOperator{\FM}{FM}
\newcommand{\CCst}{\CC^{\ast}}
\renewcommand{\shH}{\mathcal{H}}
\renewcommand{\derH}{\mathbf{H}}
\newcommand{\opp}{\mathit{opp}}

\newcommand{\QQb}{\bar{\QQ}}

\DeclareMathOperator{\Hol}{Hol}

\begin{document}

\title{Holonomic D-modules on abelian varieties}
\author{Christian Schnell}
\address{Department of Mathematics, Stony Brook University,
Stony Brook, NY 11794, USA}
\email{cschnell@math.sunysb.edu}



\begin{abstract}
We study the Fourier-Mukai transform for holonomic $\Dmod$-modules on 
complex abelian varieties. Among other things, we show that the cohomo-logy support loci
of a holonomic $\Dmod$-module are finite unions of linear subvarieties, which go through
points of finite order for objects of geometric origin; that the
standard t-structure on the derived category of holonomic complexes corresponds,
under the Fourier-Mukai transform, to a certain perverse coherent t-structure in the
sense of Kashiwara and Arinkin-Bezrukavnikov; and that Fourier-Mukai transforms
of simple holonomic $\Dmod$-modules are intersection complexes in this t-structure.
This supports the conjecture that Fourier-Mukai transforms of holonomic
$\Dmod$-modules are ``hyperk\"ahler perverse sheaves''.
\end{abstract}
\maketitle

\section{Introduction}

\subsection{Agenda}

In this paper, we begin a systematic study of holonomic $\Dmod$-modules on complex
abelian varieties; recall that a $\Dmod$-module is said to be \define{holonomic} if
its characteristic variety is a Lagrangian subset of the cotangent bundle.  Regular
holonomic $\Dmod$-modules, which correspond to perverse sheaves under the
Riemann-Hilbert correspondence, are familiar objects in complex algebraic geometry.
Due to recent breakthroughs by Kedlaya, Mochizuki, and Sabbah (summarized in
\cite{Sabbah}), we now have an almost equally good understanding of irregular holonomic
$\Dmod$-modules, and many important results from the regular case (such as the
Decomposition Theorem or the Hard Lefschetz Theorem) have been extended to the
irregular case. A careful study of an important special case, namely that of complex
abelian varieties, may therefore be of some interest.

The original motivation for this project comes from a sequence of papers by Green and
Lazarsfeld \cites{GL1,GL2}, Arapura \cite{Arapura}, and Simpson \cite{Simpson}. In
their work on the \define{generic vanishing theorem}, these authors analyzed the loci
\[
	S_m^{p,q}(X) = \menge{L \in \Pic^0(X)}{\dim H^q \bigl( X, \OmX^p \tensor L \bigr) \geq m}
		\subseteq \Pic^0(X),
\]
for $X$ a projective (or compact K\"ahler) complex manifold. Among other things, they
showed that each irreducible component of $S_m^{p,q}(X)$ is a translate of a subtorus
by a point of finite order; and they obtained bounds on the codimension in the most
interesting cases ($p=0$ and $p = \dim X$). These bounds imply for example that when the
Albanese mapping of $X$ is generically finite over its image, all higher cohomology
groups of $\omX \tensor L$ vanish for a generic line bundle $L \in \Pic^0(X)$. 

Hacon pointed out that the codimension bounds can be interpreted as properties of
certain coherent sheaves on abelian varieties, and then reproved them using the
Fourier-Mukai transform \cite{Hacon}. His method applies particularly well to those
coherent sheaves that occur in the Hodge filtration of a mixed Hodge module; based on
this observation, Popa and I generalized all of the results connected with the
generic vanishing theorem (in the projective case) to Hodge modules of geometric
origin on abelian varieties \cite{PS}. As a by-product, we also obtained results
about certain regular holonomic $\Dmod$-modules on abelian varieties, namely those
that can be realized as direct images of structure sheaves of smooth projective
varieties with nontrivial first Betti number. A pretty application of the
$\Dmod$-module theory to the geometry of varieties of general type can be found in
\cite{PS-Kodaira}.

As we shall see below, all of the results about $\Dmod$-modules in \cite{PS} remain
true for \emph{arbitrary} holonomic $\Dmod$-modules on abelian varieties. In most
cases, the proof in the general case turns out to be simpler; we shall also
discover that certain statements -- such as the codimension bounds -- only reveal their true
meaning in this broader context. 

The principal results about holonomic $\Dmod$-modules are summarized in
\subsecref{subsec:intro-structure} to \subsecref{subsec:intro-simple}; for the
convenience of the reader, we also translate everything into the language of perverse
sheaves in \subsecref{subsec:intro-perverse}.  Our main technical tool will be the
Fourier-Mukai transform for algebraic $\Dmod$-modules, introduced by Laumon
\cite{Laumon} and Rothstein \cite{Rothstein}, and our results suggest a conjecture
about the structure of Fourier-Mukai transforms of (regular or irregular) holonomic
$\Dmod$-modules. This conjecture, together with some evidence for it, is described in
\subsecref{subsec:intro-conjecture}.

\subsection{The structure theorem}
\label{subsec:intro-structure}

Let $A$ be a complex abelian variety, and let $\Dmod_A$ be the sheaf of linear
differential operators of finite order. The simplest examples of left
$\Dmod_A$-modules are line bundles $L$ with integrable connection $\nabla
\colon L \to \Omega_A^1 \tensor L$.  Because $A$ is an abelian variety, the moduli
space $\Ash$ of such pairs $(L, \nabla)$ is a quasi-projective algebraic variety of
dimension $2 \dim A$.  The basic idea in the study of $\Dmod_A$-modules is to exploit
the fact that $\Ash$ is so big. 

One approach is to consider, for a left $\Dmod_A$-module $\Mmod$, the cohomology
groups (in the sense of $\Dmod$-modules) of the various twists $\Mmod \tensor_{\OA}
(L, \nabla)$; we use this symbol to denote the natural $\Dmod_A$-module structure on
the tensor product $\Mmod \tensor_{\OA} L$. That information is contained in
the \define{cohomology support loci} of $\Mmod$, which are the sets
\begin{equation} \label{eq:CSL-h}
	S_m^k(A, \Mmod) = \Menge{(L, \nabla) \in \Ash}{\dim \derH^k \Bigl( A, 
		\DR_A \bigl( \Mmod \tensor_{\OA} (L, \nabla) \bigr) \Bigr) \geq m}.
\end{equation}
The definition works more generally for complexes of $\Dmod$-modules; we are
especially interested in the case of a \define{holonomic complex} $\Mmod \in
\Dbh(\Dmod_A)$, that is to say, a cohomologically bounded complex of
$\Dmod_A$-modules with holonomic cohomology sheaves. Our first result is the
following structure theorem.

\begin{theorem} \label{thm:h-linear}
Let $\Mmod \in \Dbh(\Dmod_A)$ be a holonomic complex.
\begin{aenumerate}
\item Each $S_m^k(A, \Mmod)$ is a finite union of linear subvarieties of $\Ash$.
\item If $\Mmod$ is a semisimple regular holonomic $\Dmod_A$-module of geometric
origin, in the sense of \cite{BBD}*{6.2.4}, then these linear subvarieties are
arithmetic.  
\end{aenumerate}
\end{theorem}

Here we are using the new term \define{(arithmetic) linear subvarieties} for what Simpson
called \define{(torsion) translates of triple tori} in \cite{Simpson}*{p.~365};
the definition is as follows.

\begin{definition} \label{def:linear}
A \define{linear subvariety} of $\Ash$ is any subset of the form 
\begin{equation} \label{eq:lin-Ash}
	(L, \nabla) \tensor \im \bigl( \fsh \colon \Bsh \to \Ash \bigr),
\end{equation}
for a surjective morphism of abelian varieties $f \colon A \to B$ with connected
fibers, and a line bundle with integrable connection $(L, \nabla) \in \Ash$. We say
that a linear subvariety is \define{arithmetic} if $(L, \nabla)$ can be taken to be torsion
point.%
\footnote{
We use the word \define{linear} because the linear subvarieties of $\Ash$ are
precisely those whose preimage in the universal cover are linear subspaces.
The reason for the term \define{arithmetic} is as follows. Let $\Char(A)$ be the
space of characters of the fundamental group of $A$; it is also a complex
algebraic variety, biholomorphic to $\Ash$, but with a different algebraic
structure. When $A$ is defined over a number field, the torsion points are precisely
those points on the algebraic varieties $\Ash$ and $\Char(A)$ that are simultaneously
defined over a number field in both \cite{Simpson}*{Proposition~3.4}.
}
\end{definition}

To prove \theoremref{thm:h-linear}, we use the Riemann-Hilbert
correspondence. If $\Mmod$ is a holonomic $\Dmod_A$-module, then according to a
fundamental theorem by Kashiwara \cite{HTT}*{Theorem~4.6.6}, its de Rham complex
\[
	\DR_A(\Mmod) = \Bigl\lbrack \Mmod \to \Omega_A^1 \tensor \Mmod \to 
		\dotsb \to \Omega_A^{\dim A} \tensor \Mmod \Bigr\rbrack \decal{\dim A},
\]
placed in degrees $-\dim A, \dotsc, 0$, is a perverse sheaf on $A$. More generally,
$\DR_A(\Mmod)$ is a constructible complex for any $\Mmod \in \Dbh(\Dmod_A)$
\cite{HTT}*{Theorem~4.6.3}, and the Riemann-Hilbert correspondence
\cite{HTT}*{Theorem~7.2.1} asserts that the functor
\[
	\DR_A \colon \Dbrh(\Dmod_A) \to \Dbc(\CC_A)
\]
from regular holonomic complexes to constructible complexes is an equivalence
of categories. 

Now let $\Char(A)$ be the space of characters of the fundamental group of $A$; any
character $\rho \colon \pi_1(A, 0) \to \CCst$ determines a local system $\CCrho$ of
rank one on $A$. We define the cohomology support loci of a constructible
complex $K \in \Dbc(\CC_A)$ as
\[
	S_m^k(A, K) = \Menge{\rho \in \Char(A)}{\dim \derH^k \bigl( A, K \tensor_{\CC}
		\CCrho \bigr) \geq m}.
\]
The well-known correspondence between vector bundles with integrable connection and
representations of the fundamental group gives a biholomorphic mapping
\begin{equation} \label{eq:Phi}
	\Phi \colon \Ash \to \Char(A), \quad (L, \nabla) \mapsto \Hol(L, \nabla),
\end{equation}
and it is very easy to show -- see \theoremref{thm:relationship} below -- that the
cohomology support loci for $\Mmod$ and $\DR_A(\Mmod)$ are related by the formula 
\begin{equation} \label{eq:relation}
	\Phi \bigl( A, S_m^k(A, \Mmod) \bigr) = S_m^k \bigl( A, \DR_A(\Mmod) \bigr).
\end{equation}

The proof of \theoremref{thm:h-linear} is based on the fact that $\Char(A)$ and
$\Ash$, while isomorphic as complex manifolds, are not isomorphic as complex
algebraic varieties. According to a nontrivial theorem by Simpson, a closed algebraic
subset $Z \subseteq \Ash$ is a finite union of linear subvarieties if and only if its
image $\Phi(Z) \subseteq \Char(A)$ remains algebraic \cite{Simpson}*{Theorem~3.1}. We
show (in \theoremref{thm:alg-c} and \propositionref{prop:alg-h}) that cohomology
support loci are algebraic subsets of $\Ash$ and $\Char(A)$; this is enough to prove
the first half of \theoremref{thm:h-linear}. To prove the second half, we show (in
\subsecref{subsec:geometric}) that the cohomology support loci of an object of
geometric origin are stable under the action of $\Aut(\CC/\QQ)$; we can then apply
another result by Simpson, namely that every ``absolute closed'' subset of $\Ash$ is
a finite union of arithmetic linear subvarieties. Another proof, completely different
but also based on the Fourier-Mukai transform, is explained in \cite{Schnell}.

\subsection{The Fourier-Mukai transform}
\label{subsec:intro-FM}

A second way to present the information about cohomology of twists of $\Mmod$ is
through the \define{Fourier-Mukai transform} for algebraic $\Dmod_A$-modules, 
introduced and studied by Laumon \cite{Laumon} and Rothstein \cite{Rothstein}. It is
an exact functor
\begin{equation} \label{eq:FM-intro}
	\FM_A \colon \Dbcoh(\Dmod_A) \to \Dbcoh(\OAsh),
\end{equation}
defined as the integral transform with kernel $(\Psh, \nablash)$, the tautological line
bundle with relative integrable connection on $A \times \Ash$.
As shown by Laumon and Rothstein, $\FM_A$ is an equivalence between the bounded
derived category
of coherent algebraic $\Dmod_A$-modules, and that of coherent algebraic sheaves on
$\Ash$. In essence, this means that an algebraic $\Dmod$-module on an
abelian variety can be recovered from the cohomology of its twists by line bundles
with integrable connection. 

The support of the complex of coherent sheaves $\FM_A(\Mmod)$ is related to the
cohomology support loci of $\Mmod$: by the base change theorem, one has
\[
	\Supp \FM_A(\Mmod) = \bigcup_{k \in \ZZ} S_1^k(A, \Mmod).
\]
In particular, the support is a finite union of linear subvarieties.
But the Fourier-Mukai transform of a holonomic complex actually satisfies a much
stronger version of \theoremref{thm:h-linear}. We shall say that a subset of $\Ash$
is \define{definable in terms of $\FM_A(\Mmod)$} if can be obtained by applying various
sheaf-theoretic operations -- such as $\derR \shHom(\argbl, \OAsh)$, truncation, or
restriction to a linear subvariety -- to $\FM_A(\Mmod)$, and then taking the support of
the resulting complex of coherent sheaves. 

\begin{theorem} \label{thm:FM-linear}
Let $\Mmod \in \Dbh(\Dmod_A)$ be a holonomic complex on an abelian variety. If a
subset of $\Ash$ is definable in terms of $\FM_A(\Mmod)$, then it is a finite
union of linear subvarieties. These linear subvarieties are arithmetic whenever
$\Mmod$ is a semisimple regular holonomic $\Dmod_A$-module of geometric origin.
\end{theorem}

The proof of \theoremref{thm:FM-linear} is based on an analogue of the Fourier-Mukai
transform for constructible complexes $K \in \Dbc(\CC_A)$ (explained in
\subsecref{subsec:constructible}). The main point is that the group ring $R = \CC
\lbrack \pi_1(A, 0) \rbrack$ is a representation of the fundamental group, and
therefore determines a local system of $R$-modules $\shL_R$ on the abelian variety.
Because $K$ is constructible and $p \colon A \to \pt$ is proper, the direct image
$\derR \pl(K \tensor_{\CC} \shL_R)$ therefore belongs to $\Dbcoh(R)$ and gives rise
to a complex of coherent algebraic sheaves on the affine algebraic variety $\Char(A)
= \Spec R$. When $K = \DR_A(\Mmod)$, we show that the resulting complex of 
coherent analytic sheaves, pulled back along $\Phi \colon \Ash \to \Char(A)$, becomes
canonically isomorphic to $\FM_A(\Mmod)$. Both assertions in
\theoremref{thm:FM-linear} then follow as before from Simpson's theorems.

\subsection{Codimension bounds and perverse coherent sheaves}
\label{subsec:intro-codimension}

Inequalities for the codimension of cohomology support loci first appeared in the
work of Green and Lazarsfeld on the generic vanishing theorem \cite{GL1}. For
example, when $X$ is a projective complex manifold whose Albanese mapping
is generically finite over its image, Green and Lazarsfeld proved that
\[
	\codim_{\Pic^0(X)} \menge{L \in \Pic^0(X)}{H^k(X, \omX \tensor L) \neq 0} \geq k
\]
for every $k \geq 0$. More recently, Popa \cite{Popa} noticed that such codimension
bounds can be expressed in terms of a certain nonstandard t-structure on the
derived category, introduced by Kashiwara \cite{Kashiwara} and Arinkin and
Bezrukavnikov \cite{AB} in their work on ``perverse coherent sheaves''. 

In our context, codimension bounds and t-structures are related on a much deeper
level. The first result is that the position of a holonomic complex with
respect to the standard t-structure on the category $\Dbh(\Dmod_A)$ is detected by the
codimension of its cohomology support loci. 

\begin{theorem}  \label{thm:h-t}
Let $\Mmod \in \Dbh(\Dmod_A)$ be a holonomic complex. Then one has
\begin{align*}
\Mmod \in \Dth{\leq 0}(\Dmod_A) \quad &\Longleftrightarrow \quad
		\text{$\codim S_1^k(A, \Mmod) \geq 2 k$ for every $k \in \ZZ$,} \\
\Mmod \in \Dth{\geq 0}(\Dmod_A) \quad &\Longleftrightarrow \quad
		\text{$\codim S_1^k(A, \Mmod) \geq -2k$ for every $k \in \ZZ$.}
\end{align*}
In particular, $\Mmod$ is a single holonomic $\Dmod_A$-module if and only if its
cohomology support loci satisfy $\codim S_1^k(A, \Mmod) \geq \abs{2k}$ for every $k
\in \ZZ$.
\end{theorem}

The natural setting for this result is the theory of \define{perverse coherent
sheaves}, developed by Kashiwara and by Arinkin and Bezrukavnikov. As a matter of
fact, there is a perverse t-structure on $\Dbcoh(\OAsh)$ with the property that
\[
	\mDtcoh{\leq 0}(\OAsh) = \menge{F \in \Dbcoh(\OAsh)}%
		{\text{$\codim \Supp \shH^k F \geq 2k$ for every $k \in \ZZ$}};
\]
it corresponds to the supporting function $m = \left\lfloor \tfrac{1}{2} \codim
\right\rfloor$ on the topological space of the scheme $\Ash$, in Kashiwara's
terminology. Its heart $\mCoh(\OAsh)$ is the abelian category of \define{$m$-perverse
coherent sheaves} (see \subsecref{subsec:perverse}).

Now \theoremref{thm:h-t} is a consequence of the following better result, which says
that the Fourier-Mukai transform interchanges the standard t-structure
on $\Dbh(\Dmod_A)$ and the $m$-perverse t-structure on $\Dbcoh(\OAsh)$.%
\footnote{%
This is very surprising at first glance, because one expects the Fourier-Mukai
transform to interchange ``local'' and ``global'' data. But both t-structures
in this result are in fact defined by local conditions.%
}

\begin{theorem} \label{thm:h-t-FM}
Let $\Mmod \in \Dbh(\Dmod_A)$ be a holonomic complex on $A$. The one has
\begin{align*}
\Mmod \in \Dth{\leq k}(\Dmod_A) \quad &\Longleftrightarrow \quad
		\text{$\FM_A(\Mmod) \in \mDtcoh{\leq k}(\OAsh)$,} \\
\Mmod \in \Dth{\geq k}(\Dmod_A) \quad &\Longleftrightarrow \quad
		\text{$\FM_A(\Mmod) \in \mDtcoh{\geq k}(\OAsh)$.}
\end{align*}
In particular, $\Mmod$ is a single holonomic $\Dmod_A$-module if and only if
its Fourier-Mukai transform $\FM_A(\Mmod)$ is an $m$-perverse coherent sheaf on $\Ash$.
\end{theorem}

The proofs of both theorems can be found in \subsecref{subsec:codimension}. The first
step in the proof is to show that when $\Mmod$ is a holonomic $\Dmod_A$-module, the
cohomology sheaves $\shH^i \FM_A(\Mmod)$ are torsion sheaves for $i > 0$. Here the
crucial point is that the characteristic variety $\Ch(\Mmod)$ inside $T^{\ast} A = A
\times H^0(A, \OmA^1)$ has the same dimension as $A$ itself; this
makes the second projection
\[
	\Ch(\Mmod) \to H^0(A, \OmA^1)
\]
finite over a general point of $H^0(A, \OmA^1)$. To deduce results about
$\FM_A(\Mmod)$, we use an extension of the Fourier-Mukai transform to
$\Rmod_A$-modules, where $\Rmod_A = R_F \Dmod_A$ is the Rees algebra. Choose a good
filtration $F_{\bullet} \Mmod$, and consider the coherent sheaf $\gr^F \!\! \Mmod$ on
$T^{\ast} A$ determined by the graded $\Sym \shT_A$-module $\gr_{\bullet}^F \!
\Mmod$; its support is precisely $\Ch(\Mmod)$. The extended Fourier-Mukai transform
of the Rees module $R_F \Mmod$ then interpolates between $\FM_A(\Mmod)$ and the
complex
\[
	\derR (p_{23})_{\ast} \Bigl( p_{12}^{\ast} P \tensor 
		p_{13}^{\ast}(\gr^F \!\! \Mmod) \Bigr),
\]
and because the higher cohomology sheaves of the latter are torsion, we obtain the
result for $\FM_A(\Mmod)$. This ``generic vanishing theorem'' implies also that the
cohomology support loci $S_1^k(A, \Mmod)$ are proper subvarieties for $k \neq 0$;
in the regular case, this result is due to Kr\"amer and Weissauer
\cite{KW}*{Theorem~2}.

Once the generic vanishing theorem has been established, \theoremref{thm:FM-linear} implies
that $\shH^i \FM_A(\Mmod)$ is supported in a finite union of linear subvarieties of
lower dimension; because of the functoriality of the Fourier-Mukai transform,
\theoremref{thm:h-t-FM} can then be deduced very easily by induction on the
dimension. 

From there, the basic properties of the $m$-perverse t-structure quickly lead to the
following result about the Fourier-Mukai transform.

\begin{corollary} \label{cor:FM-summary}
Let $\Mmod$ be a holonomic $\Dmod_A$-module. The only potentially nonzero
cohomology sheaves of the Fourier-Mukai transform $\FM_A(\Mmod)$ are
\[
	\text{$\shH^0 \FM_A(\Mmod)$, $\shH^1 \FM_A(\Mmod)$, \dots, $\shH^{\dim A} \FM_A(\Mmod)$}.
\]
Their supports satisfy $\codim \Supp \shH^i \FM_A(\Mmod) \geq 2i$, and if $r \geq 0$
is the least integer for which $\shH^r \FM_A(\Mmod) \neq 0$, then $\codim \Supp
\shH^r \FM_A(\Mmod) = 2r$.
\end{corollary}

\subsection{Results about simple holonomic D-modules}
\label{subsec:intro-simple}

According to \theoremref{thm:FM-linear}, the Fourier-Mukai transform of a holonomic
$\Dmod_A$-module is supported in a finite union of linear subvarieties. For
\emph{simple} holonomic $\Dmod_A$-modules, one can say more: the support of the
Fourier-Mukai transform is always irreducible, and if it is not equal to $\Ash$, then
the $\Dmod_A$-module is -- up to tensoring by a line bundle with integrable
connection -- pulled back from an abelian variety of lower dimension.

\begin{theorem} \label{thm:h-simple}
Let $\Mmod$ be a simple holonomic $\Dmod_A$-module. Then 
\[
	\Supp \FM_A(\Mmod) 
		= (L, \nabla) \tensor \im \bigl( \fsh \colon \Bsh \to \Ash \bigr)
\]
is a linear subvariety of $\Ash$ (in the sense of \definitionref{def:linear}), and we have
\[
	\Mmod \tensor_{\OA} (L, \nabla) \simeq \fu \Nmod 
\]
for a simple holonomic $\Dmod_B$-module $\Nmod$ with $\Supp \FM_B(\Nmod) = \Bsh$.
\end{theorem}

The idea of the proof is that for some $r \geq 0$, the support of $\shH^r
\FM_A(\Mmod)$ has to contain a linear subvariety $(L, \nabla) \tensor \im \fsh$ of
codimension $2r$. Because of the functoriality of the Fourier-Mukai transform,
restricting $\FM_A(\Mmod)$ to this subvariety corresponds to taking the direct image
$\fp \bigl( \Mmod \tensor (L, \nabla) \bigr)$. We then use adjointness and the
fact that $\Mmod$ is simple to conclude that $\Mmod \tensor (L, \nabla)$ is
pulled back from $B$.

One application of \theoremref{thm:h-simple} is to classify simple holonomic
$\Dmod_A$-modules with Euler characteristic zero. Recall that the \define{Euler
characteristic} of a coherent algebraic $\Dmod_A$-module $\Mmod$ is the integer
\[
	\chi(A, \Mmod) = \sum_{k \in \ZZ} (-1)^k \dim \derH^k \bigl( A, \DR_A(\Mmod) \bigr).
\]
When $\Mmod$ is holonomic, we have $\chi(A, \Mmod) \geq 0$ as a consequence of
\theoremref{thm:h-t-FM} and the deformation invariance of the Euler characteristic.
In the regular case, the following result has been proved in a different way by Weissauer
\cite{Weissauer}*{Theorem~2}.
	
\begin{corollary} \label{cor:h-simple}
Let $\Mmod$ be a simple holonomic $\Dmod_A$-module. If $\chi(A, \Mmod) = 0$, then
there exists an abelian variety $B$, a surjective morphism $f \colon A
\to B$ with connected fibers, and a simple holonomic $\Dmod_B$-module $\Nmod$ with
$\chi(B, \Nmod) > 0$, such that
\[	
	\Mmod \tensor_{\OA} (L, \nabla) \simeq \fu \Nmod
\]
for a suitable point $(L, \nabla) \in \Ash$. 
\end{corollary}

Now suppose that $\Mmod$ is a simple holonomic $\Dmod$-module with $\shH^0
\FM_A(\Mmod) \neq 0$. In that case, the proof of \theoremref{thm:h-simple} actually
gives the stronger inequalities
\[
	\codim \Supp \shH^i \FM_A(\Mmod) \geq 2i+2 \quad
		\text{for every $i \geq 1$.}
\]
We deduce from this that $\shH^0 \FM_A(\Mmod)$ is a reflexive sheaf, 
locally free on the complement of a finite union of linear subvarieties of
codimension $\geq 4$. This fact allows us to reconstruct (in \corollaryref{cor:IC})
the entire complex $\FM_A(\Mmod)$ from the locally free sheaf $\ju \shH^0
\FM_A(\Mmod)$ by applying the functor
\[
	\tau_{\leq \ell(A)-1} \circ \derR \shHom(\argbl, \shO) \circ \dotsb \circ
		\tau_{\leq 2} \circ \derR \shHom(\argbl, \shO) \circ \tau_{\leq 1} \circ 
		\derR \shHom(\argbl, \shO) \circ \jl.
\]
Here $\ell(A)$ is the smallest odd integer $\geq \dim A$, and $j$ is the inclusion
of the open set where $\shH^0 \FM_A(\Mmod)$ is locally free. 

This formula looks a bit like Deligne's formula for the intersection complex of a
local system \cite{BBD}*{Proposition~2.1.11}. We investigate this analogy in
\subsecref{subsec:IC}, where we show that the same formula can be used to define an
\define{intersection complex}
\[
	\IC_X(\shE) \in \mCoh(\OX),
\]
where $j \colon U \into X$ is an open subset of a smooth complex algebraic variety
$X$ with $\codim(X \setminus U) \geq 2$, and $\shE$ is a locally free coherent sheaf on
$U$. This complex has some of the same properties as its cousin in \cite{BBD}. In that sense, 
\[
	\FM_A(\Mmod) \simeq \IC_{\Ash} \bigl( \ju \shH^0 \FM_A(\Mmod) \bigr)
\]
is indeed the intersection complex of a locally free sheaf. When $\shH^0 \FM_A(\Mmod)
= 0$, \theoremref{thm:h-simple} shows that $\FM_A(\Mmod)$ is still the intersection
complex of a locally free sheaf, but now on a linear subvariety of $\Ash$ of lower
dimension.

\subsection{A conjecture}
\label{subsec:intro-conjecture}

By now, it will have become clear that Fourier-Mukai transforms of holonomic
$\Dmod_A$-modules are very special complexes of coherent sheaves on the moduli space
$\Ash$. Because the Fourier-Mukai transform
\[
	\FM_A \colon \Dbcoh(\Dmod_A) \to \Dbcoh(\OAsh)
\] 
is an equivalence of categories, this suggests the following general question.

\begin{question}
Let $\Dbh(\Dmod_A)$ denote the full subcategory of $\Dbcoh(\Dmod_A)$, consisting of
complexes with holonomic cohomology sheaves. What is the image of $\Dbh(\Dmod_A)$
under the Fourier-Mukai transform? In particular, which complexes of coherent sheaves
on $\Ash$ are Fourier-Mukai transforms of holonomic $\Dmod_A$-modules?
\end{question}

In this section, I would like to propose a conjectural answer to this question.
Roughly speaking, the answer seems to be the following:
\begin{align*}
 	\FM_A \bigl( \Dbh(\Dmod_A) \bigr) 
		&= \text{derived category of \define{hyperk\"ahler constructible complexes},} \\
	\FM_A \bigl( \Modh(\Dmod_A) \bigr) 
		&= \text{abelian category of \define{hyperk\"ahler perverse sheaves}.}
\end{align*}
Recall that the space of line bundles with connection is a hyperk\"ahler manifold: as
complex manifolds, one has $\Ash \simeq H^1(A, \CC) / H^1(A, \ZZ)$, and any
polarization of the Hodge structure on $H^1(A, \CC)$ gives rise to a flat hyperk\"ahler
metric on $\Ash$. Here is some evidence for this point of view:

\begin{enumerate}
\item Finite unions of linear subvarieties of $\Ash$ are precisely those algebraic
subvarieties that are also hyperk\"ahler subvarieties.
\item Given a holonomic complex $\Mmod \in \Dbh(\Dmod_A)$, there is a finite
stratification of $\Ash$ by hyperk\"ahler subvarieties such that the restriction of
$\FM_A(\Mmod)$ to each stratum has locally free cohomology sheaves.
\item We prove in \subsecref{subsec:truncation} that a complex of coherent sheaves
lies in the subcategory $\FM_A \bigl( \Dbh(\Dmod_A) \bigr)$ if and only if all
of its cohomology sheaves do. This gives some justification for using the term
``constructible complex''.
\item If we use \define{quaternionic} dimension, \theoremref{thm:h-t-FM} becomes
\[
	\dim_{\HH} \Supp \shH^i \FM_A(\Mmod) \leq \dim_{\HH} \Ash - i = \dim A - i
\]
for a holonomic $\Dmod_A$-module $\Mmod$; this says that the complex $\FM_A(\Mmod)
\decal{\dim A}$ is perverse for the usual middle perversity \cite{BBD}*{Chapter~2}
over $\HH$.
\item For a simple holonomic $\Dmod$-module $\Mmod$, the Fourier-Mukai transform
$\FM_A(\Mmod)$ is the intersection complex of a locally free sheaf.
\end{enumerate}

Unfortunately, nobody has yet defined a category of hyperk\"ahler perverse sheaves,
even in the case of compact hyperk\"ahler manifolds; and our situation presents the
additional difficulty that $\Ash$ is not compact. Nevertheless, I believe that, based
on the work of Mochizuki on twistor $\Dmod$-modules \cite{Mochizuki-wild}, it is
possible to make an educated guess, at least in the case of simple
holonomic $\Dmod$-modules.%

\begin{conjecture}
Let $\shF$ be a reflexive coherent algebraic sheaf on $\Ash$. Then there exists a
simple holonomic $\Dmod_A$-module $\Mmod$ with the property that $\shF \simeq \shH^0
\FM_A(\Mmod)$ if and only if the following conditions are satisfied:
\begin{aenumerate}
\item $\shF$ is locally free on the complement of a finite union of linear
subvarieties of codimension at least four.
\item \label{en:conj-b}
The resulting locally free sheaf admits a Hermitian metric $h$ whose
curvature tensor $\Theta_h$ is $\SU{2}$-invariant and locally square-integrable on $\Ash$.
\item \label{en:conj-c}
The pointwise norm of $\Theta_h$, taken with respect to $h$, is in $O \bigl( d^{-(1+\eps)}
\bigr)$, where $d$ is the distance to the origin in $\Ash$.
\end{aenumerate}
Moreover, $\Mmod$ is regular if and only if the pointwise norm of $\Theta_h$ is
in $O(d^{-2})$.
\end{conjecture}

There is a certain amount of redundancy in the conditions. In fact, we could start from a
holomorphic vector bundle $\shE$ on the complement of a finite union of linear
subvarieties of codimension $\geq 2$, and assume that it admits a Hermitian metric
$h$ for which \ref{en:conj-b} and \ref{en:conj-c} are true. Then $h$ is admissible in
the sense of Bando and Siu \cite{BS}, and therefore extends uniquely to a reflexive
coherent analytic sheaf on $\Ash$; by virtue of \ref{en:conj-c}, the extension is
acceptable in the sense of \cite{Mochizuki-wild}*{Chapter~21}, and therefore
algebraic.%
\footnote{%
This is just to give an idea; in reality, the extension is typically not locally
free, and so  it is more difficult to prove that is is algebraic.%
}
In particular, $\shE$ itself is algebraic, and the discussion at the end of
\subsecref{subsec:intro-simple} shows that the simple holonomic $\Dmod_A$-module must
be
\[
	\FM_A^{-1} \Bigl( \IC_{\Ash}(\shE) \Bigr),
\]
the inverse Fourier-Mukai transform of the intersection complex of $\shE$.

The paper \cite{Mochizuki-N} establishes an equivalent result in the case of elliptic
curves. The reason for believing that regularity should correspond to quadratic decay
in the curvature is the work of Jardim \cite{Jardim}. In general, the existence of
the metric, and the $\SU{2}$-invariance of its curvature, should be consequences of
the fact that every simple holonomic $\Dmod$-module lifts to a polarized wild
pure twistor $\Dmod$-module. The remaining points will probably require methods from
analysis. 

Note that the conjecture is consistent with the result (in \corollaryref{cor:Chern})
that all Chern classes of $\FM_A(\Mmod)$ are zero in cohomology. Another interesting
question is whether the existence of the metric in \ref{en:conj-b} is equivalent to
an algebraic condition such as stability. If that was the case, then I would guess
that the semistable objects are what corresponds to Fourier-Mukai transforms of
holonomic $\Dmod_A$-modules.

\subsection{Results about perverse sheaves}
\label{subsec:intro-perverse}

For the convenience of those readers who are more familiar with constructible
complexes and perverse sheaves, we shall now translate the main results into that
language.  In the sequel, a \define{constructible complex} on the abelian variety $A$
means a complex $K$ of sheaves of $\CC$-vector spaces, whose cohomology sheaves
$\shH^i K$ are constructible with respect to an algebraic stratification of
$A$, and vanish for $i$ outside some bounded interval. We denote by $\Dbc(\CC_A)$ the
bounded derived category of constructible complexes. It is a basic fact
\cite{HTT}*{Section~4.5} that the hypercohomology groups $\derH^i(A, K)$ are
finite-dimensional complex vector spaces for any $K \in \Dbc(\CC_A)$.

Now let $\Char(A)$ be the space of characters of the fundamental group; it is
also the moduli space for local systems of rank one. For any character $\rho \colon
\pi_1(A,0) \to \CCst$, we denote the corresponding local system on $A$ by the symbol
$\CCrho$. It is easy to see that $K \tensor_{\CC} \CCrho$ is again constructible for
any $K \in \Dbc(\CC_A)$; we may therefore define the \define{cohomology support loci}
of $K \in \Dbc(\CC_A)$ to be the subsets
\begin{equation} \label{eq:CSL-c}
	S_m^k(A, K) = \Menge{\rho \in \Char(A)}{\dim \derH^k \bigl( A, K \tensor_{\CC}
		\CCrho \bigr) \geq m},
\end{equation}
for any pair of integers $k,m \in \ZZ$. Since the space of characters is
very large -- its dimension is equal to $2 \dim A$ -- these loci
should contain a lot of information about the original constructible complex $K$, and
indeed they do.

Our first result is a structure theorem for cohomology support loci.

\begin{definition}
A \define{linear subvariety} of $\Char(A)$ is any subset of the form 
\[
	\rho \cdot \im \bigl( \Char(f) \colon \Char(B) \to \Char(A) \bigr),
\]
for a surjective morphism of abelian varieties $f \colon A \to B$ with connected
fibers, and a character $\rho \in \Char(A)$. We say that a linear subvariety is
\define{arithmetic} if $\rho$ can be taken to be torsion point of $\Char(A)$.
\end{definition}

\begin{theorem} \label{thm:c-linear}
Let $K \in \Dbc(\CC_A)$ be a constructible complex.
\begin{aenumerate}
\item Each $S_m^k(A, K)$ is a finite union of linear subvarieties of $\Char(A)$.
\label{en:c-linear-a}
\item If $K$ is a semisimple perverse sheaf of geometric origin \cite{BBD}*{6.2.4},
then these linear subvarieties are arithmetic.
\label{en:c-linear-b}
\end{aenumerate}
\end{theorem}

\begin{proof}
For \ref{en:c-linear-a}, we use the Riemann-Hilbert correspondence to find
a regular holonomic complex $\Mmod \in \Dbrh(\Dmod_A)$ with $\DR_A(\Mmod) \simeq
K$. Since $S_m^k(A, K) = \Phi \bigl( S_m^k(A, \Mmod) \bigr)$ by
\theoremref{thm:relationship}, the assertion follows from \theoremref{thm:h-linear}.
The statement in \ref{en:c-linear-b} may be deduced from \theoremref{thm:geom-origin} by a
similar argument.
\end{proof}

The next result has to do with the codimension of the cohomology support loci.
Recall that the category $\Dbc(\CC_A)$ has a nonstandard t-structure
\[
	\Bigl( \piDtc{\leq 0}(\CC_A), \piDtc{\geq 0}(\CC_A) \Bigr),
\]
called the \define{perverse t-structure}, whose heart is the abelian category of
perverse sheaves \cite{BBD}. We show that the position of a constructible
complex with respect to this t-structure can be read off from its cohomology
support loci.

\begin{theorem} \label{thm:c-t}
Let $K \in \Dbc(\CC_A)$ be a constructible complex. Then one has
\begin{align*}
	K \in \piDtc{\leq 0}(\CC_A) \quad &\Longleftrightarrow \quad
		\text{$\codim S_1^k(A, K) \geq 2 k$ for every $k \in \ZZ$,} \\
	K \in \piDtc{\geq 0}(\CC_A) \quad &\Longleftrightarrow \quad
		\text{$\codim S_1^k(A, K) \geq -2k$ for every $k \in \ZZ$.}
\end{align*}
Thus $K$ is a perverse sheaf if and only if $\codim S_1^k(A, K) \geq
\abs{2k}$ for every $k \in \ZZ$.
\end{theorem}

\begin{proof}
Let $\Mmod \in \Dbrh(\Dmod_A)$ be a regular holonomic complex such that $K \simeq
\DR_A(\Mmod)$. Since $S_m^k(A, K) = \Phi \bigl( S_m^k(A, \Mmod) \bigr)$, the first
assertion is a consequence of \theoremref{thm:t-structure}. Now let $\DA \colon
\Dbc(\CC_A) \to \Dbc(\CC_A)$ be the Verdier duality functor; then
\[
	S_m^k(A, K) = \langle -1_{\Char(A)} \rangle \, S_m^{-k}(A, \DA K)
\]
by Verdier duality. Since $K \in \piDtc{\geq 0}(\CC_A)$ if and only if $\DA K \in
\piDtc{\leq 0}(\CC_A)$, the second assertion follows. The final assertion is clear
from the definition of perverse sheaves as the heart of the perverse t-structure on
$\Dbc(\CC_A)$.
\end{proof}

A consequence is the following ``generic vanishing theorem'' for perverse sheaves; a
similar -- but less precise -- statement has been proved some time ago by Kr\"amer
and Weissauer \cite{KW}*{Theorem~2}.

\begin{corollary}
Let $K \in \Dbc(\CC_A)$ be a perverse sheaf on a complex abelian variety. Then 
the cohomology support loci $S_m^k(A, K)$ are finite unions of linear subvarieties of
$\Char(A)$ of codimension at least $\abs{2k}$. In particular, one has
\[
	\derH^k \bigl( A, K \tensor_{\CC} \CCrho \bigr) = 0 
\]
for general $\rho \in \Char(A)$ and $k \neq 0$.
\end{corollary}

The generic vanishing theorem implies that the \define{Euler characteristic}
\[
	\chi(A, K) = \sum_{k \in \ZZ} (-1)^k \dim \derH^k \bigl( A, K \bigr)
\]
of a perverse sheaf on an abelian variety is always nonnegative, a result originally
due to Franecki and Kapranov \cite{FK}*{Corollary~1.4}. Indeed, from the deformation
invariance of the Euler characteristic, we get
\[
	\chi(A, K) = \chi \bigl( A, K \tensor_{\CC} \CCrho \bigr) = 
		\dim \derH^0 \bigl( A, K \tensor_{\CC} \CCrho \bigr) \geq 0
\]
for a general character $\rho \in \Char(A)$. For \emph{simple} perverse
sheaves with $\chi(A, K) = 0$, we have the following structure theorem
\cite{Weissauer}*{Theorem~2}.

\begin{theorem} \label{thm:c-simple}
Let $K \in \Dbc(\CC_A)$ be a simple perverse sheaf. If $\chi(A, K) = 0$, then there
exists an abelian variety $B$, a surjective morphism $f \colon A \to B$ with connected
fibers, and a simple perverse sheaf $K' \in \Dbc(\CC_B)$ with $\chi(B, K') > 0$, such that
\[
	K \simeq \fu K' \tensor_{\CC} \CCrho
\]
for some character $\rho \in \Char(A)$.
\end{theorem}

\begin{proof}
This again follows from the Riemann-Hilbert correspondence and the analogous 
result for simple holonomic $\Dmod_A$-modules in \corollaryref{cor:h-simple}.
\end{proof}

\subsection{Acknowledgement}

This work was supported by the World Premier International Research Center Initiative
(WPI Initiative), MEXT, Japan, and by NSF grant DMS-1331641.  I thank Mihnea Popa and
Pierre Schapira for their comments about the paper, and Takuro Mochizuki, Kiyoshi
Takeuchi, Giovanni Morando, and Kentaro Hori for useful discussions. I am also very
grateful to my parents-in-law for their hospitality while I was writing the 
first version of this paper.

\section{The Fourier-Mukai transform}

In this chapter, we recall Laumon's construction of the Fourier-Mukai
transform for algebraic $\Dmod$-modules on a complex abelian variety \cite{Laumon}.
Using a different approach, Rothstein obtained the same results in \cite{Rothstein}.

\subsection{Operations on D-modules}

Let $A$ be a complex abelian variety; we usually put $g = \dim A$. Before introducing the
Fourier-Mukai transform,
it may be helpful to say a few words about $\Dmod_A$, the sheaf of linear
differential operators of finite order. Recall that the tangent bundle of $A$ is
trivial; $\Dmod_A$ is therefore generated, as an $\OA$-algebra, by any basis
$\partial_1, \dotsc, \partial_g \in H^0(A, \shT_A)$, subject to the relations
\[
	\text{$\lie{\partial_i}{\partial_j} = 0$ and $\lie{\partial_i}{f} = \partial_i f$,} 
		\qquad \text{for $1 \leq i,j \leq g$ and $f \in \Gamma(U, \OA)$.}
\]
By an \define{algebraic $\Dmod_A$-module}, we mean a sheaf of left
$\Dmod_A$-modules that is quasi-coherent as a sheaf of $\OA$-modules; a
$\Dmod_A$-module is \define{holonomic} if its characteristic variety, as a subset of
the cotangent bundle $T^{\ast} A$, has dimension equal to $\dim A$ (and is therefore
a finite union of conical Lagrangian subvarieties). Finally, a \define{holonomic
complex} is a complex of $\Dmod_A$-modules $\Mmod$, whose cohomology sheaves $\shH^i
\Mmod$ are holonomic, and vanish for $i$ outside some bounded interval. We denote by
$\Dbcoh(\Dmod_A)$ the derived category of cohomologically bounded and coherent
$\Dmod_A$-modules, and by $\Dbh(\Dmod_A)$ the full subcategory of all holonomic
complexes. We refer the reader to \cite{HTT}*{Chapter~3} for additional details.

\begin{note}
Because $A$ is projective, a coherent analytic $\Dmod_A$-module is algebraic if and
only if it contains a \define{lattice}, that is to say, a coherent $\OA$-submodule
that generates it as a $\Dmod_A$-module. By a theorem of Malgrange
\cite{Malgrange}*{Theorem~3.1}, this is always the case for holonomic
$\Dmod_A$-modules; thus there is no difference between holonomic complexes of analytic
and algebraic $\Dmod_A$-modules.
\end{note}

Because it will play such an important role below, we briefly discuss the definition of the
\define{de Rham complex}, and especially the conventions about signs. For a single algebraic
$\Dmod_A$-module $\Mmod$, we define 
\[
	\DR_A(\Mmod) = \Bigl\lbrack \Mmod \to \OmA^1 \tensor \Mmod \to \dotsb \to
		\OmA^g \tensor \Mmod \Bigr\rbrack \decal{g},
\]
which we view as a complex of sheaves of $\CC$-vector spaces in the analytic
topology, placed in degrees $-g, \dotsc, 0$. The differential is given by $(-1)^g
\nabla_{\!\Mmod}$, where
\[
	\nabla_{\!\Mmod} \colon 	
		\OmA^p \tensor \Mmod \to \OmA^{p+1} \tensor \Mmod, \quad
	\omega \tensor m \mapsto d\omega \tensor m 
		+ \sum_{j=1}^g (dz_j \wedge \omega) \tensor \partial_j m;
\]
here $dz_1, \dotsc, dz_g \in H^0(A, \OmA^1)$ is the basis dual to $\partial_1,
\dotsc, \partial_g \in H^0(A, \shT_A)$. Given a complex of
algebraic $\Dmod_A$-modules $(\Mmod^{\bullet}, d)$, we define
\[
	\DR_A(\Mmod^{\bullet})
\]
to be the single complex determined by the double complex $\bigl( D^{\bullet, \bullet},
d_1, d_2 \bigr)$, whose term in bidegree $(i,j)$ is equal to
\[
	D^{i,j} = \OmA^{g+i} \tensor \Mmod^j,
\]
and whose differentials are given by the formulas
\[
	d_1 = (-1)^g \nabla_{\!\Mmod^j} \quad \text{and} \quad d_2 = \id \tensor d.
\]
Note that, according to the sign rules introduced by Deligne, the differential in the
total complex acts as $d_1 + (-1)^i d_2$ on the summand $D^{i,j}$.

The fundamental operations on algebraic $\Dmod$-modules -- such as direct and inverse
images or duality -- are described in \cite{HTT}*{Part~I}. Here we only recall the
notation. Let $f \colon A \to B$ be a morphism of abelian varieties. On the one hand,
one has the \define{direct image functor}
\[
	\fp \colon \Dbcoh(\Dmod_A) \to \Dbcoh(\Dmod_B);
\]
in case $f$ is surjective (and hence smooth), $\fp$ is given by the formula
\[
	\fp \cMmod = \derR \fl \DR_{A/B}(\cMmod),
\]
where $\DR_{A/B}(\cMmod)$ denotes the relative de Rham complex, defined
in a similar way as above, but with $g = \dim A$ replaced by the relative dimension $r =
\dim A - \dim B$. For holonomic complexes, we have an induced functor
\[
	\fp \colon \Dbh(\Dmod_A) \to \Dbh(\Dmod_B)
\]
since direct images by algebraic morphisms preserve holonomicity
\cite{HTT}*{Theorem~3.2.3}.  We also use the \define{shifted inverse image functor}
\[
	\fsi = \derL \fu \decal{\dim A - \dim B} \colon \Db(\Dmod_B) \to \Db(\Dmod_A);
\]
in general, it only preserves coherence when $f$ is surjective (and hence smooth).
According to \cite{HTT}*{Theorem~3.2.3}, we get an induced functor
\[
	\fsi \colon \Dbh(\Dmod_B) \to \Dbh(\Dmod_A).
\]
Finally, a very important role will be played by the \define{duality functor}
\[
	\DA \colon \Dbcoh(\Dmod_A) \to \Dbcoh(\Dmod_A)^{\opp}, \quad
		\DA(\Mmod^{\bullet}) = \derR \shHom_{\Dmod_A}(\Mmod^{\bullet}, \Dmod_A)
		\tensor (\Omega_A^{g})^{-1} \decal{g}.
\]
Note that a $\Dmod_A$-module $\Mmod$ is holonomic exactly when its dual
$\DA(\Mmod)$ is again a single $\Dmod_A$-module (viewed as a complex concentrated in
degree zero).

\subsection{Definition and basic properties}

We now come to the definition of the Fourier-Mukai transform. Following Mazur-Messing
\cite{MM}, we let $\Ash$ denote the moduli space of algebraic line bundles with
integrable connection on the abelian variety $A$. It is naturally a
quasi-projective algebraic variety: on the dual abelian variety $\Ah = \Pic^0(A)$,
there is a canonical extension of vector bundles 
\begin{equation} \label{eq:extension}
	0 \to \Ah \times H^0(A, \OmA^1) \to E(A) \to \Ah \times \CC \to 0,
\end{equation}
whose extension class in
\begin{align*}
	\Ext^1 \bigl( \OAh, \OAh \times H^0(A, \OmA^1) \bigr)
		&\simeq H^1(\Ah, \OAh) \tensor H^0(A, \OmA^1) \\
		&\simeq H^0(A, \shT_A) \tensor H^0(A, \OmA^1)
\end{align*}
is represented by $\sum_j \partial_j \tensor dz_j$. Then $\Ash$ is isomorphic to the
preimage of $\Ah \times \{1\}$ inside of $E(A)$, and the projection
\[
	\pi \colon \Ash \to \Ah,  \quad (L, \nabla) \mapsto L,
\] 
is a torsor for the trivial bundle $\Ah \times H^0(A, \OmA^1)$. This corresponds to
the fact that $\nabla + \omega$ is again an integrable connection for any $\omega \in
H^0(A, \OmA^1)$. Note that $\Ash$ is a group under tensor product, with
unit element $(\OA, d)$.

The generalized Fourier-Mukai transform takes bounded complexes of algebraic
$\Dmod_A$-modules to bounded complexes of quasi-coherent sheaves on $\Ash$;
we briefly describe it following the presentation in \cite{Laumon}*{\S3}. Let $P$
denote the normalized Poincar\'e bundle on the product $A \times \Ah$.
Since $\Ash$ is the moduli space of line bundles with integrable connection on $A$, the
pullback $\Psh = (\id_A \times \pi)^{\ast} P$ of the Poincar\'e bundle to the product
$A \times \Ash$ is endowed with a universal integrable connection 
\[
	\nablash \colon \Psh \to \Omega_{A \times \Ash / \Ash}^1 \tensor \Psh 
\]
relative to $\Ash$. The construction of $\nablash$ can be found in
\cite{MM}*{Chapter~I}.  An algebraic left $\Dmod_A$-module $\Mmod$ may be
interpreted as a quasi-coherent sheaf of $\OA$-modules with integrable connection
$\nabla \colon \Mmod \to \Omega_A^1 \tensor \Mmod$; then 
\[
	p_1^{\ast} \nabla \tensor \id + \id \tensor \nablash
\]
defines a relative integrable connection on the tensor product $p_1^* \Mmod
\tensor_{\OA} \Psh$, and we denote the resulting algebraic $\Dmod_{A \times \Ash /
\Ash}$-module by the symbol $p_1^{\ast} \Mmod \tensor (\Psh, \nablash)$.
Given a complex of algebraic $\Dmod_A$-modules $(\Mmod^{\bullet}, d)$, we define
\[
	\DR_{A \times \Ash / \Ash} 
		\Bigl( p_1^{\ast} \Mmod^{\bullet} \tensor (\Psh, \nablash) \Bigr)
\]
to be the single complex determined by the double complex $\bigl( D^{\bullet, \bullet},
d_1, d_2 \bigr)$, whose term in bidegree $(i,j)$ is equal to
\[
	D^{i,j} = \Omega_{A \times \Ash / \Ash}^{g+i} \tensor p_1^{\ast} \Mmod^j \tensor \Psh,
\]
and whose differentials are given by the formulas
\[
	d_1 = (-1)^g \bigl( p_1^{\ast} \nabla_{\!\Mmod^j} \tensor \id + \id \tensor \nablash \bigr) 
		\quad \text{and} \quad
	d_2 = \id \tensor d \tensor \id.
\]
We then define the \define{Fourier-Mukai transform} of the complex $\Mmod^{\bullet}$
by the formula
\begin{equation} \label{eqn:laumon}
	\FM_A(\Mmod^{\bullet}) = \derR (p_2)_{\ast} \DR_{A \times \Ash / \Ash}
		 \Bigl( p_1^{\ast} \Mmod^{\bullet} \tensor (\Psh, \nablash) \Bigr)
\end{equation}
Because every differential in the relative de Rham complex is $\OAsh$-linear,
$\FM_A(\Mmod^{\bullet})$ is naturally a complex of algebraic quasi-coherent sheaves
on $\Ash$.  The following fundamental theorem was proved by Laumon 
\cite{Laumon}*{Th\'eor\`em~3.2.1 and Corollaire~3.2.5}, and, using a totally different
method, by Rothstein \cite{Rothstein}*{Theorem~6.2}.

\begin{theorem}[Laumon, Rothstein] \label{thm:Laumon-Rothstein}
The Fourier-Mukai transform gives rise to an equivalence of categories
\begin{equation} \label{eq:FM}
	\FM_A \colon \Dbcoh(\Dmod_A) \to \Dbcoh(\OAsh)
\end{equation}
between the bounded derived category of coherent algebraic $\Dmod_A$-modules and 
the bounded derived category of coherent algebraic sheaves on $\Ash$. 
\end{theorem}

The Fourier-Mukai transform is compatible with various operations on $\Dmod$-modules;
here, taken from Laumon's paper, is a list of the basic properties that we will use.

\begin{theorem}[Laumon] \label{thm:Laumon}
The Fourier-Mukai transform for algebraic $\Dmod$-modules on abelian
varieties enjoys the following properties.
\begin{aenumerate}
\item \label{en:Laumon1}
For $(L, \nabla) \in \Ash$, denote by $t_{(L, \nabla)} \colon \Ash \to \Ash$ the
translation morphism. Then one has a canonical and functorial isomorphism
\[
	\FM_A \bigl( \argbl \tensor_{\OA} (L, \nabla) \bigr) =
		\derL (t_{(L, \nabla)})^{\ast} \circ \FM_A.
\]
\item \label{en:Laumon4}
One has a canonical and functorial isomorphism
\[
	\FM_A \circ \DA = \langle -1_{\Ash} \rangle^{\ast} \, 
		\derR \shHom \bigl( \FM_A(\argbl), \OAsh \bigr).
\]
\item \label{en:Laumon2}
For a morphism $f \colon A \to B$ of abelian varieties, denote by $\fsh \colon \Bsh \to \Ash$
the induced morphism. Then one has canonical and functorial isomorphisms
\begin{align*}
	\derL (\fsh)^{\ast} \circ \FM_A &= \FM_B \circ \fp, \\
	\derR \fsh_{\ast} \circ \FM_B &= \FM_A \circ \fsi.
\end{align*}
(Note that $\fsi$ only preserves coherence when $f$ is smooth.)
\item \label{en:Laumon5}
One has a canonical and functorial isomorphism
\[
	\FM_A \circ \bigl( \Dmod_A \tensor_{\OA} (\argbl) \bigr) =
		\derL \piu \circ \derR \Phi_P,
\]
where $\pi \colon \Ash \to \Ah$ denotes the projection, and $\derR \Phi_P \colon
\Dbcoh(\OA) \to \Dbcoh(\OAh)$ is the usual Fourier-Mukai transform for coherent
sheaves from \cite{Mukai}.
\end{aenumerate}
\end{theorem}

\begin{proof}
\ref{en:Laumon1} is immediate from the properties of the normalized Poincar\'e bundle on
$A \times \Ash$. \ref{en:Laumon2} is proved in
\cite{Laumon}*{Proposition~3.3.2}; note that ``$g-1-g_2$'' should read ``$g_1 -
g_2$.'' The compatibility of the Fourier-Mukai transform with duality in
\ref{en:Laumon4} can be found in \cite{Laumon}*{Proposition~3.3.4}. Lastly,
\ref{en:Laumon5} is proved in \cite{Laumon}*{Proposition~3.1.2}.
\end{proof}

\subsection{The space of generalized connections}

During the proof of \theoremref{thm:h-t-FM} in \subsecref{subsec:codimension} below, it
will be necessary to compare the Fourier-Mukai transform of a $\Dmod$-module to that
of the associated graded object $\gr_{\bullet}^F \! \Mmod$, for some choice of good
filtration $F_{\bullet} \Mmod$.  It is convenient to introduce the Rees algebra
\[
	\Rmod_A = \bigoplus_{k \in \ZZ} F_k \Dmod_A \tensor z^k
		\subseteq \Dmod_A \lbrack z \rbrack,
\]
and to replace a filtered $\Dmod$-module by the associated graded $\Rmod$-module
\[
	R_F \Mmod = \bigoplus_{k \in \ZZ} F_k \Mmod \tensor z^k
		\subseteq \Mmod \tensor_{\OA} \OA \lbrack z, z^{-1} \rbrack.
\]
We shall extend the Fourier-Mukai transform to this setting; the role of $\Ash$ is
played by $E(A)$, the moduli space of line bundles with generalized connection.
Recall that $E(A)$ was defined by the extension in \eqref{eq:extension}; we begin by
explaining another construction, whose idea is originally due to Deligne and Simpson
(see \cite{Bonsdorff}).

\begin{definition}
Let $X$ be a complex manifold, and $\lambda \colon X \to \CC$ a holomorphic function. 
A \define{generalized connection with parameter $\lambda$}, or more briefly a
\define{$\lambda$-connection}, on a locally free sheaf of $\OX$-modules $\shE$ is a
$\CC$-linear morphism of sheaves 
\[
	\nabla \colon \shE \to \Omega_X^1 \tensor_{\OX} \shE
\]
that satisfies the Leibniz rule with parameter $\lambda$, which is to say that
\[
	\nabla(f \cdot s) = f \cdot \nabla s + df \tensor \lambda s
\]
for local sections $f \in \Gamma(U, \OX)$ and $s \in \Gamma(U, \shE)$. A
$\lambda$-connection is called
\define{integrable} if its $\OX$-linear curvature operator $\nabla \circ \nabla
\colon \shE \to \Omega_X^2 \tensor_{\OX} \shE$ is equal to zero.
\end{definition}

\begin{example}
An integrable $1$-connection is an integrable connection in the usual sense; an
integrable $0$-connection is the structure of a Higgs bundle on
$\shE$.
\end{example}

On an abelian variety $A$, the moduli space $E(A)$ of line bundles with
integrable $\lambda$-connection (for all $\lambda \in \CC$) may be constructed 
as follows. Observe first that any $\lambda$-connection on a line bundle $L \in
\Pic^0(A)$ is automatically integrable. To construct the moduli
space, let $\mm_A \subseteq \OA$ denote the ideal sheaf of the unit element
$0_A \in A$.  Restriction of differential forms induces an isomorphism
\[
	\mm_A / \mm_A^2 \simeq H^0(A, \Omega_A^1) \tensor \OA / \mm_A,
\]
and therefore determines an extension of coherent sheaves
\[
	0 \to H^0(A, \Omega_A^1) \tensor \OA / \mm_A \to \OA / \mm_A^2 
		\to \OA / \mm_A \to 0.
\]
Let $P$ be the normalized Poincar\'e bundle on the product $A \times \Ah$, and denote
by $\derR \Phi_P \colon \Dbcoh(\OA) \to \Dbcoh(\OAh)$ the Fourier-Mukai transform.
Then $\derR \Phi_P \bigl( \OA / \mm_A^2 \bigr)$ is a locally free sheaf $\shE(A)$,
and so we obtain an extension of locally free sheaves 
\begin{equation} \label{eq:ext-sheaf}
	0 \to H^0(A, \Omega_A^1) \tensor \OAh \to \shE(A) \to \OAh \to 0
\end{equation}
on the dual abelian variety $\Ah$. The corresponding extension of vector bundles is
the one in \eqref{eq:extension}. By construction, $E(A)$ comes with two algebraic
morphisms $\pi \colon E(A) \to \Ah$ and $\lambda \colon E(A) \to \CC$.
The following lemma shows that there is a universal line bundle with generalized
connection on $A \times E(A)$.

\begin{lemma} \label{lem:connection}
Let $\Pg = (\id \times \pi)^{\ast} P$ denote the pullback of the Poincar\'e bundle to
$A \times E(A)$. Then there is a canonical generalized relative connection
\[
	\nablag \colon \Pg \to \Omega_{A \times E(A) / E(A)}^1 \tensor \Pg
\]
that satisfies the Leibniz rule $\nablag(f \cdot s) = f \cdot \nablag s
+ d_{A \times E(A) / E(A)} f \tensor \lambda s$.
\end{lemma}

\begin{proof}
Let $\shI$ denote the ideal sheaf of the diagonal in $A \times A$. Let $Z$ be the
non-reduced subscheme of $A \times A \times E(A)$ defined by the ideal sheaf
$\shI^2 \cdot \shO_{A \times A \times E(A)}$. We have a natural exact sequence
\begin{equation} \label{eq:seq-connection}
	0 \to \Pg \tensor H^0(A, \Omega_A^1) \to (p_{13})_{\ast} \bigl( 
		\shO_Z \tensor p_{23}^{\ast} \Pg \bigr) \to \Pg \to 0,
\end{equation}
and a generalized relative connection is the same thing as a morphism of sheaves
\[
	\Pg \to (p_{13})_{\ast} \bigl( \shO_Z \tensor p_{23}^{\ast} \Pg \bigr)
\]
whose composition with the morphism to $\Pg$ acts as multiplication by $\lambda$. In
fact, there is a canonical choice, which we shall now describe.
Consider the morphism
\[
	f \colon A \times A \to A \times A, \quad f(a, b) = (a, a+b).
\]
Since $f \times \id_{E(A)}$ induces an isomorphism between the first infinitesimal
neighborhood of $A \times \{0_A\} \times E(A)$ and the subscheme $Z$, we have
\begin{align*}
	(f \times \id_{E(A)})^{\ast} \bigl( \shO_Z \tensor p_{23}^{\ast} \Pg \bigr)
		&= p_2^{\ast}(\OA / \mm_A^2) \tensor (m \times \id_{E(A)})^{\ast} \Pg \\
		&= p_2^{\ast}(\OA / \mm_A^2) \tensor 
			p_{13}^{\ast} \Pg \tensor p_{23}^{\ast} \Pg,
\end{align*}
due to the well-known fact that the Poincar\'e bundle satisfies
\[
	(m \times \id_{\Ah})^{\ast} P = p_{13}^{\ast} P \tensor p_{23}^{\ast} P
\]
on $A \times A \times \Ah$. Since $p_{13} \circ (f \times \id_{E(A)}) = p_{13}$, we
conclude that we have
\[
	(p_{13})_{\ast} \bigl( \shO_Z \tensor p_{23}^{\ast} \Pg \bigr)
	= \Pg \tensor p_2^{\ast} \pi^{\ast} \derR \Phi_P \bigl( \OA / \mm_A^2 \bigr)
	= \Pg \tensor p_2^{\ast} \pi^{\ast} \shE(A)
\]
on $A \times E(A)$; more precisely, \eqref{eq:seq-connection} is isomorphic to the
tensor product of $\Pg$ and the pullback of \eqref{eq:ext-sheaf} by $\pi \circ p_2$.

Now the pullback of the exact sequence \eqref{eq:ext-sheaf} to $E(A)$ obviously has a
splitting of the type we are looking for: indeed, the tautological section of
$\pi^{\ast} \shE(A)$ gives a morphism $\shO_{E(A)} \to \pi^{\ast} \shE(A)$ whose
composition with the projection to $\shO_{E(A)}$ is multiplication by $\lambda$. Thus
we obtain a canonical morphism $\Pg \to \Pg \tensor p_2^{\ast} \pi^{\ast} \shE(A)$
and hence, by the above, the desired generalized relative connection.
\end{proof}

At any point $e \in E(A)$, we thus obtain a $\lambda(e)$-connection
on the line bundle corresponding to $\pi(e) \in \Pic^0(A)$. This shows that $E(A)$ is
the moduli space of (topologically trivial) line bundles with integrable generalized
connection.  Using the properties of the Picard scheme, one can show that $E(A)$ is a
fine moduli space in the obvious sense; but since we do not need this fact below, we
shall not present the proof.

We close this section with two simple lemmas that describe how $E(A)$ and $(\Pg,
\nablag)$ behave under restriction to the fibers of $\lambda \colon E(A) \to \CC$.

\begin{lemma} \label{lem:compatible-1}
We have $\lambda^{-1}(1) = \Ash$, and the restriction of $(\Pg, \nablag)$ to
$A \times \Ash$ is equal to $(\Psh, \nablash)$.
\end{lemma}

\begin{proof}
This follows from the construction of $\Ash$ and $\nablash$ in \cite{MM}*{Chapter~I}.
\end{proof}

Recall that the cotangent bundle of $A$ satisfies $T^{\ast} A = A \times H^0(A,
\Omega_A^1)$, and consider the following diagram:
\begin{equation} \label{eq:diag-cotangent}
\begin{tikzcd}
	A \times H^0(A, \Omega_A^1) & 
	A \times \Ah \times H^0(A, \Omega_A^1) 
		\arrow{l}[above]{p_{13}} \arrow{r}{p_{23}} \arrow{d}{p_{12}} &
	\Ah \times H^0(A, \Omega_A^1) \\
	& A \times \Ah 
\end{tikzcd}
\end{equation}

\begin{lemma} \label{lem:compatible-0}
We have $\lambda^{-1}(0) = \Ah \times H^0(A, \Omega_A^1)$, and the restriction of
$(\Pg, \nablag)$ to $A \times \Ah \times H^0(A, \Omega_A^1)$ is equal to the
Higgs bundle
\[
	\bigl( p_{12}^{\ast} P, \, p_{23}^{\ast} \theta_A \bigr),
\]
where $\theta_A$ denotes the tautological holomorphic one-form on $T^{\ast} A$.
\end{lemma}

\begin{proof}
This follows easily from the proof of \lemmaref{lem:connection}.
\end{proof}

\subsection{The extended Fourier-Mukai transform}

We shall now describe an extension of the Fourier-Mukai transform to the case of 
algebraic $\Rmod_A$-modules. Since we only need a very special case in this paper, we
leave a more careful discussion to a future publication. Recall that $\Dmod_A$ is
generated as an $\OA$-algebra by a basis of vector fields $\partial_1, \dotsc,
\partial_g \in H^0(A, \shT_A)$, subject to the relations
\[
	\lie{\partial_i}{\partial_j} = 0 \quad \text{and} \quad
	\lie{\partial_i}{f} = \partial_i f.
\]
Likewise, $\Rmod_A$ is generated as an $\OA \lbrack z \rbrack$-algebra by $z
\partial_1, \dotsc, z \partial_g$, subject to 
\[
	\lie{z \partial_i}{z \partial_j} = 0 \quad \text{and} \quad
	\lie{z \partial_i}{f} = z \cdot \partial_i f. 
\]
It is not hard to show that $\Rmod_A$ is isomorphic to $(p_1)_{\ast} \Rmod_{A \times
\CC / \CC}$, where $\Rmod_{A \times \CC / \CC}$ denotes the $\shO_{A \times
\CC}$-subalgebra of $\Dmod_{A \times \CC / \CC}$ generated by $z \shT_{A \times \CC /
\CC}$. If $M$ is an algebraic $\Rmod_A$-module, then the associated quasi-coherent
sheaf $\Mmodt$ on $A \times \CC$ is naturally a module over $\Rmod_{A \times \CC /
\CC}$, and vice versa.
	
Now fix an algebraic $\Rmod_A$-module $M$. Let $\Rmod_{A \times E(A) / E(A)}$ denote
the subalgebra of $\Dmod_{A \times E(A) / E(A)}$ generated by $\lambda \shT_{A \times
E(A) / E(A)}$. The tensor product
\[
	(\id \times \lambda)^{\ast} \Mmodt \tensor_{\shO_{A \times E(A)}} \Pg
\]
naturally has the structure of $\Rmod_{A \times E(A) / E(A)}$-module on $A \times
E(A)$: concretely, the module structure is given by $\lambda(m \tensor s) = (\lambda
m) \tensor s = m \tensor \lambda s$ and $\lambda \partial_i(m \tensor s) = (z
\partial_i m) \tensor s + m \tensor \nablag_{\partial_i}(s)$.
We may therefore consider the relative de Rham complex
\[
	\DR_{A \times E(A) / E(A)} \Bigl( (\id \times \lambda)^{\ast} \Mmodt
		\tensor_{\shO_{A \times E(A)}} \bigl( \Pg, \nablag \bigr) \Bigr),
\]
which is defined just as in the case of $\Dmod$-modules.

\begin{definition} \label{def:FM-R}
The \define{Fourier-Mukai transform} of an algebraic $\Rmod_A$-module $M$ is
\[
	\FMt_A(M) = \derR (p_2)_{\ast} 
		\DR_{A \times E(A) / E(A)} \Bigl( (\id \times \lambda)^{\ast} \Mmodt 
			\tensor_{\shO_{A \times E(A)}} \Pg \Bigr);
\]
it is an object of $\Db \bigl( \shO_{E(A)} \bigr)$, the bounded derived category of
algebraic quasi-coherent sheaves on $E(A)$.
\end{definition}

\begin{note}
Using the general formalism in \cite{PR}, one can show that the Fourier-Mukai transform
induces an equivalence of categories
\[
	\FMt_A \colon \Db(\Rmod_A) \to \Db \bigl( \shO_{E(A)} \bigr),
\]
Since this fact will not be used below, we shall omit the proof.
\end{note}

\begin{lemma}
If $M$ is a coherent algebraic $\Rmod_A$-module, $\FMt_A(M) \in \Dbcoh(\shO_{E(A)})$. 
\end{lemma}

\begin{proof}
The proof is the same as in the case of $\Dmod_A$-modules; for more details, refer to 
\cite{Laumon}*{Proposition~3.1.2 and Corollaire~3.1.3}.
\end{proof}

\subsection{Compatibility}
\label{subsec:compatible}

Just as $\Rmod_A$-modules interpolate between $\Dmod_A$-modules and quasi-coherent
sheaves on the cotangent bundle $T^{\ast} A$, the extended Fourier-Mukai transform in
\definitionref{def:FM-R} interpolates between the Fourier-Mukai transform for
$\Dmod_A$-modules and the usual Fourier-Mukai transform for quasi-coherent sheaves.
The purpose of this section is to make that relationship precise. 

Throughout the discussion, let $\Mmod$ be a coherent algebraic $\Dmod_A$-module and
$F_{\bullet} \Mmod$ a good filtration of $\Mmod$ by $\OA$-coherent subsheaves. The
graded $\Sym \shT_A$-module 
\[
	\gr_{\bullet}^F \! \Mmod = \bigoplus_{k \in \ZZ} F_k \Mmod / F_{k-1} \Mmod
\]
is then coherent over $\Sym \shT_A$, and therefore defines a coherent sheaf on the
cotangent bundle $T^{\ast} A$ that we shall denote by the symbol $\gr^F \!\! \Mmod$.
Now consider once more the Rees module
\[
	R_F \Mmod = \bigoplus_{k \in \ZZ} F_k \Mmod \cdot z^k 
		\subseteq \Mmod \tensor_{\OA} \OA \lbrack z, z^{-1} \rbrack,
\]
which is a graded $\Rmod_A$-module, coherent over $\Rmod_A$. The associated quasi-coherent
sheaf on $A \times \CC$, which we shall denote by the symbol $\ReesM$, is equivariant
for the natural $\CC^{\ast}$-action on the product.  Moreover, it is easy to see that
the restriction of $\ReesM$ to $A \times \{1\}$ is a $\Dmod_A$-module isomorphic to
$\Mmod$, while the restriction to $A \times \{0\}$ is a graded $\Sym
\shT_A$-module isomorphic to $\gr_{\bullet}^F \! \Mmod$.

\begin{proposition} \label{prop:compatible}
Let $\Mmod$ be a coherent algebraic $\Dmod_A$-module with good filtration
$F_{\bullet} \Mmod$.  Then the extended Fourier-Mukai transform $\FMt_A(R_F \Mmod)
\in \Dbcoh(\shO_{E(A)})$ of the associated graded $\Rmod_A$-module has the following
properties:
\begin{renumerate}
\item It is equivariant for the natural $\CC^{\ast}$-action on the
vector bundle $E(A)$.
\label{en:FM-R-i}
\item Its restriction to $\Ash = \lambda^{-1}(1)$ is canonically
isomorphic to $\FM_A(\Mmod)$.
\label{en:FM-R-ii}
\item Its restriction to $A \times H^0(A, \Omega_A^1) =
\lambda^{-1}(0)$ is canonically isomorphic to
\[
	\derR (p_{23})_{\ast} \Bigl( p_{12}^{\ast} P \tensor p_1^{\ast} \Omega_A^g
			\tensor p_{13}^{\ast} \gr^F \!\! \Mmod \Bigr),
\]
where the notation is as in the diagram in \eqref{eq:diag-cotangent} above.
\label{en:FM-R-iii}
\end{renumerate}
\end{proposition}

\begin{proof}
\ref{en:FM-R-i} is true because $R_F \Mmod$ is a graded $\Rmod_A$-module, and because
$(\Pg, \nablag)$ and the relative de Rham complex are obviously
$\CC^{\ast}$-equivariant. \ref{en:FM-R-ii} follows directly from the definition of the
Fourier-Mukai transform, using the base change formula for the morphism $\lambda
\colon E(A) \to \CC$ and \lemmaref{lem:compatible-1}.

The proof of \ref{en:FM-R-iii} is a little less obvious, and so we give some details.
By base change, it suffices to show that the restriction of the relative de Rham complex
\[
	\DR_{A \times E(A) / E(A)} \Bigl( (\id \times \lambda)^{\ast} \ReesM
		\tensor_{\shO_{A \times E(A)}} \bigl( \Pg, \nablag \bigr) \Bigr)
\]
to $A \times \Ah \times H^0(A, \Omega_A^1)$ is a resolution of the
coherent sheaf $p_{12}^{\ast} P \tensor p_1^{\ast} \Omega_A^g \tensor p_{13}^{\ast}
\gr^F \!\! \Mmod$. After a short computation, one finds that this restriction is
isomorphic to the tensor product of $p_{12}^{\ast} P$ and the pullback, via $p_{13}$,
of the complex
\begin{equation} \label{eq:cx-cotangent}
	\biggl\lbrack 
		p_1^{\ast} \Bigl( \gr_{\bullet}^F \! \Mmod \Bigr) \to 
			p_1^{\ast} \Bigl( \Omega_A^1 \tensor_{\OA} \gr_{\bullet}^F \! \Mmod \Bigr) 
		\to \dotsb \to 
			p_1^{\ast} \Bigl( \Omega_A^g \tensor_{\OA} \gr_{\bullet}^F \! \Mmod \Bigr)
	\biggr\rbrack \decal{g},
\end{equation}
with differential $p_1^{\ast} \Bigl( \Omega_A^k \tensor \gr_{\bullet}^F \Mmod \Bigr)
\to p_1^{\ast} \Bigl( \Omega_A^{k+1} \tensor \gr_{\bullet}^F \Mmod \Bigr)$ given by
the formula
\[
	\omega \tensor m \mapsto (-1)^g (\theta_A \wedge \omega) \tensor m + 
		(-1)^g \sum_{j=1}^g (dz_j \wedge \omega) \tensor \partial_j m .
\]
But since $\gr^F \!\! \Mmod$ is the coherent sheaf on $A \times H^0(A, \Omega_A^1)$
corresponding to the $\Sym \shT_A$-module $\gr_{\bullet}^F \! \Mmod$, the complex in
\eqref{eq:cx-cotangent} is a resolution of the coherent sheaf $p_1^{\ast} \Omega_A^g
\tensor \gr^F \!\! \Mmod$, and so we get the desired result.
\end{proof}

\section{The structure theorem}

The purpose of this chapter is to prove the structure theorems for cohomology support
loci of holonomic and constructible complexes. What relates the two sides is the
Riemann-Hilbert correspondence.

\subsection{Cohomology of constructible complexes}
\label{subsec:constructible}

In this section, we describe an analogue of the Fourier-Mukai transform for
constructible complexes on $A$, and use it to prove that cohomology support loci are
algebraic subvarieties of $\Char(A)$. We refer the reader to \cite{HTT}*{Section~4.5}
and to \cite{Dimca}*{Chapter~4} for details about constructible complexes and
perverse sheaves.

The abelian variety $A$ may be presented as a quotient $V / \Lambda$, where $V$
is a complex vector space of dimension $g$, and $\Lambda \subseteq V$ is a lattice of
rank $2g$. Note that $V$ is isomorphic to the tangent space of $A$ at the unit
element, while $\Lambda$ is isomorphic to the fundamental group $\pi_1(A, 0)$.
We shall denote by $\kLambda$ the group ring of $\Lambda$ with coefficients in
a subfield $k \subseteq \CC$; thus
\[
	\kLambda = \bigoplus_{\lambda \in \Lambda} k e_{\lambda},
\]
with $e_{\lambda} \cdot e_{\mu} = e_{\lambda + \mu}$. A choice of basis for $\Lambda$
shows that $\kLambda$ is isomorphic to the ring of Laurent polynomials in $2g$
variables.  Any character $\rho \colon \Lambda \to \CCst$
extends uniquely to a homomorphism of $\CC$-algebras
\[
	\CLambda \to \CC, \quad e_{\lambda} \mapsto \rho(\lambda),
\]
whose kernel is a maximal ideal $\mmrho \subseteq \CLambda$; concretely, $\mmrho$ is 
generated by the elements $e_{\lambda} - \rho(\lambda)$, for $\lambda \in \Lambda$.
It is easy to see that any maximal ideal of $\CLambda$ is of this form;
this means that $\Char(A)$ is the set of $\CC$-valued points of the scheme $\Spec
\CLambda$, and therefore naturally an affine complex algebraic variety.

For any $\kLambda$-module $M$, multiplication by the ring elements $e_{\lambda}$
determines a natural action of $\Lambda$ on the $k$-vector space $M$.  By the
well-known correspondence between representations of the fundamental group and
local systems, it thus gives rise to a local system on $A$.

\begin{definition}
For a $\kLambda$-module $M$, we denote by the symbol $\shL_M$ the corresponding local
system of $k$-vector spaces on $A$.
\end{definition}

Since $\kLambda$ is commutative, $\shL_M$ is actually a \emph{local system of
$\kLambda$-modules}. The most important example is $\shLC$, which
is a local system of $\CLambda$-modules of rank one; one can show that it is
isomorphic to the direct image with proper support $\pi_! \, \CC_V$ of the constant
local system on the universal covering space $\pi \colon V \to A$.
This device allows us to construct $\CLambda$-modules, and hence
quasi-coherent sheaves on $\Char(A)$, by twisting a complex of sheaves of
$\CC$-vector spaces by a local system of the form $\shL_M$, and pushing
forward along the morphism $p \colon A \to \pt$ to a point.

\begin{proposition}
Let $k$ be a field, and let $K \in \Dbc(k_A)$ be a constructible complex of sheaves
of $k$-vector spaces on $A$. Then for any finitely generated $\kLambda$-module $M$,
the direct image $\derR \pl \bigl( K \tensor_k \shL_M \bigr)$ belongs to $\Dbcoh
\bigl( \kLambda \bigr)$.
\end{proposition}

\begin{proof}
Since $K$ is a constructible complex of sheaves of $k$-vector spaces, the tensor
product $K \tensor_k \shL_M$ is a constructible complex of sheaves of
$\kLambda$-modules. By \cite{Dimca}*{Corollary~4.1.6}, its direct image is thus an
object of $\Dbcoh \bigl( \kLambda \bigr)$.
\end{proof}

To understand how $\derR \pl \bigl( K \tensor_k \shL_M \bigr)$ depends on $M$, we
need the following auxiliary result. Recall that a \define{fine sheaf} on a
manifold is a sheaf admitting partitions of unity; such sheaves are acyclic for
direct image functors.

\begin{lemma} \label{lem:functor}
Let $\shF$ be a fine sheaf of $\CC$-vector spaces on $A$. Then the space of global
sections $H^0 \bigl( A, \shF \tensor_{\CC} \shLC \bigr)$ is a flat $\CLambda$-module,
and for every $\CLambda$-module $M$, one has 
\[
	H^0 \bigl( A, \shF \tensor_{\CC} \shL_M \bigr) \simeq
		H^0 \bigl( A, \shF \tensor_{\CC} \shLC \bigr) \tensor_{\CLambda} M,
\]
functorially in $M$.
\end{lemma}

\begin{proof}
Each sheaf of the form $\shF \tensor_{\CC} \shL_M$ is clearly again a
fine sheaf. Consequently, $M \mapsto H^0 \bigl( A, \shF \tensor_{\CC} \shL_M
\bigr)$ is an exact functor from the category of $\CLambda$-modules to the category
of $\CLambda$-modules. Since the functor also preserves arbitrary direct sums, the
result follows from the Eilenberg-Watts theorem in homological algebra
\cite{Watts}*{Theorem~1}. 
\end{proof}

\begin{proposition} \label{prop:iso-c}
Let $K \in \Dbc(\CC_A)$. Then for every $\CLambda$-module $M$, one has an
isomorphism
\[
	\derR \pl \bigl( K \tensor_{\CC} \shL_M \bigr) \simeq
		\derR \pl \bigl( K \tensor_{\CC} \shLC \bigr) \Ltensor_{\CLambda} M,
\]
functorial in $M$.
\end{proposition}

\begin{proof}
We begin by choosing a bounded complex $(\shF^{\bullet}, d)$ of fine sheaves
quasi-isomorphic to $K$. One way to do this is as follows. By the Riemann-Hilbert
correspondence, $K \simeq \DR_A(\Mmod)$ for some $\Mmod \in
\Dbrh(\Dmod_A)$; if we now let $\shA_A^k$ denote the sheaf of smooth $k$-forms on the
complex manifold $A$, then by the Poincar\'e lemma, 
\[
	\Bigl\lbrack \shA_A^0 \tensor_{\OA} \Mmod \to \shA_A^1 \tensor_{\OA} \Mmod 
		\to \dotsb \to \shA_A^{2g} \tensor_{\OA} \Mmod \Bigr\rbrack,
\]
with the usual degree shift by $g$, is a complex of fine sheaves quasi-isomorphic to
$K$. For any such choice, $\derR \pl \bigl( K \tensor_{\CC} \shL_M \bigr) \in \Dbcoh
\bigl( \CLambda \bigr)$ is represented by the bounded complex of $\CLambda$-modules
with terms
\[
	H^0 \bigl( A, \shF^{\bullet} \tensor_{\CC} \shL_M \bigr) \simeq
		H^0 \bigl( A, \shF^{\bullet} \tensor_{\CC} \shLC \bigr) \tensor_{\CLambda} M,
\]
and so the assertion follows from \lemmaref{lem:functor}.
\end{proof}

Now let $\rho \in \Char(A)$ be an arbitrary character; recall that
$\mmrho$ is the maximal ideal of $\CLambda$ generated by the elements $e_{\lambda} -
\rho(\lambda)$, for $\lambda \in \Lambda$. Using the notation introduced above, we
therefore have the alternative description $\CCrho \simeq \shL_{\CLambda / \mmrho}$
for the local system corresponding to $\rho$.

\begin{corollary} \label{cor:fibers}
For any character $\rho \in \Char(A)$, we have 
\[
	\derR \pl \bigl( K \tensor_{\CC} \CCrho \bigr) \simeq
	\derR \pl \bigl( K \tensor_{\CC} \shLC \bigr) 
		\Ltensor_{\CLambda} \CLambda / \mmrho
\]
as objects of $\Dbcoh(\CC)$.
\end{corollary}

We may thus consider the complex $\derR \pl \bigl( K \tensor_{\CC} \shLC \bigr) \in
\Dbcoh \bigl( \CLambda \bigr)$ as being a sort of ``Fourier-Mukai transform''
of the constructible complex $K \in \Dbc(\CC_A)$. This point of view is justified also
by its relationship with the Fourier-Mukai transform for $\Dmod$-modules in
\theoremref{thm:relationship-FM} below.%
\footnote{%
In this setting, the transform does \emph{not} determine the original constructible
complex: for example, if $K$ is any constructible sheaf whose support is a finite
union of points of $A$, then $\derR \pl \bigl( K \tensor_{\CC} \shLC \bigr)$ is a
free $\CLambda$-module of finite rank.  %
}

The results above are all that is needed to prove that the cohomology support loci of
a constructible complex are algebraic subsets of $\Char(A)$.

\begin{theorem} \label{thm:alg-c}
If $K \in \Dbc(\CC_A)$, then each cohomology support locus $S_m^k(A, K)$ is an
algebraic subset of $\Char(A)$.
\end{theorem}

\begin{proof}
Recall that $\Char(A)$ is the complex manifold associated to the complex algebraic
variety $\Spec \CLambda$. Thus $\derR \pl \bigl( K \tensor_{\CC} \shLC
\bigr) \in \Dbcoh \bigl( \CLambda \bigr)$ determines an object in the bounded derived
category of coherent algebraic sheaves on $\Char(A)$, whose fiber at any closed point
$\rho$ computes the hypercohomology of $K \tensor_{\CC} \CCrho$, according
to \corollaryref{cor:fibers}.  We conclude that
\[
	S_m^k(A, K) = \Menge{\rho \in \Char(A)}{\dim \derH^k \Bigl( \derR \pl 		
		\bigl( K \tensor_{\CC} \shLC \bigr) \Ltensor_{\CLambda} 
			\CLambda/\mmrho \Bigr) \geq m},
\]
and by \lemmaref{lem:perfect} below, this description implies that $S_m^k(A, K)$ is an
algebraic subset of $\Char(A)$.
\end{proof}

For the convenience of the reader, we include the following elementary lemma.

\begin{lemma} \label{lem:perfect}
Let $E$ be a perfect complex on a scheme $X$. For each $k,m \in \ZZ$, 
\[
	\menge{x \in X}{\dim \derH^k \bigl( E \Ltensor_{\OX} \shO_{X,x}/\mm_x \bigr) \geq m}
\]
is the set of closed points of a closed subscheme of $X$.
\end{lemma}

\begin{proof}
Since the statement is local, it suffices to consider the case of a complex
\[
\begin{tikzcd}[column sep=scriptsize]
	E^{k-1} \arrow{r}{f} & E^k \arrow{r}{g} & E^{k+1}
\end{tikzcd}
\]
of free $R$-modules of finite rank. For any prime ideal $P \subseteq R$, we shall use
the notation 
\[
\begin{tikzcd}[column sep=scriptsize]
	E^{k-1} \tensor_R R/P \arrow{r}{f_P} & E^k \tensor_R R/P \arrow{r}{g_P} & 
		E^{k+1} \tensor_R R/P
\end{tikzcd}
\]
for the tensor product of the above complex by $R/P$. Now $\rk(\ker
g_P/\im f_P) \geq m$ is clearly equivalent to having $\rk(\ker g_P) \geq m+i$ and
$\rk(\im f_P) \leq i$ for some $i \geq 0$. Consequently, the set of prime ideals $P
\in \Spec R$ for which $\rk(\ker g_P/\im f_P) \geq m$ is equal to
\[
	\bigcup_{i \geq 0} \menge{P}{\rk(\ker g_P) \geq m+i}
				\cap \menge{P}{\rk(\ker f_P) \geq \rk(E^{k-1}) - i}
\]
But this set can be defined by the vanishing of certain minors of $f$ and $g$,
and is therefore naturally a closed subscheme of $\Spec R$.
\end{proof}

For the more arithmetic questions in \subsecref{subsec:geometric}, we make the
following observation about fields of definition.

\begin{proposition} \label{prop:subfield}
Let $k$ be any subfield of $\CC$. If $K \in \Dbc(k_A)$ is a constructible complex of
sheaves of $k$-vector spaces, then $\derR \pl \bigl( K \tensor_k \shLC \bigr)$ is
defined over $k$.
\end{proposition}

\begin{proof}
Indeed, we have $\shLC = \shLk \tensor_k \CC$, and therefore
\[
	\derR \pl \bigl( K \tensor_k \shLC \bigr) \simeq 	
		\derR \pl \bigl( K \tensor_k \shLk \bigr) \tensor_k \CC
\]
is obtained by extension of scalars from an object of $\Dbcoh \bigl( \kLambda
\bigr)$.
\end{proof}

\subsection{Comparison theorems}

Recall that we have a biholomorphic mapping 
\[
	\Phi \colon \Ash \to \Char(A)
\]
that takes a line bundle with integrable connection to the corresponding local system
of rank one. In this section, we relate the Fourier-Mukai transform of a holonomic 
complex $\Mmod \in \Dbh(\Dmod_A)$ and the transform of the constructible complex
$\DR_A(\Mmod)$. Our first result is purely set-theoretic, and concerns the
relationship between the cohomology support loci of $\Mmod$ and $\DR_A(\Mmod)$.

\begin{theorem} \label{thm:relationship}
Let $\Mmod \in \Dbh(\Dmod_A)$ be a holonomic complex on $A$. Then 
\[
	\Phi \bigl( A, S_m^k(A, \Mmod) \bigr) = S_m^k \bigl( A, \DR_A(\Mmod) \bigr).
\]
for every $k,m \in \ZZ$.
\end{theorem}

\begin{proof}
Let $\bigl( \cMmod, d \bigr)$ be the given holonomic complex. For any line bundle
with integrable connection $(L, \nabla)$, the associated local system $\ker \nabla$
is a subsheaf of $L$, and we have $(\ker \nabla) \tensor_{\CC} \OA = L$. This means
that the natural morphism of sheaves
\[
	\bigl( \OmA^{g+i} \tensor_{\OA} \Mmod^j \bigr) \tensor_{\CC} \ker \nabla
		\to \OmA^{g+i} \tensor_{\OA} \bigl( \Mmod^j \tensor_{\OX} L \bigr) 
\]
is an isomorphism for every $i,j \in \ZZ$. Since it is compatible with the
differentials $d_1 = (-1)^g \nabla_{\! \Mmod^j}$ and $d_2 = \id \tensor d$, we obtain
an isomorphism of complexes
\[
	\DR_A(\cMmod) \tensor_{\CC} \ker \nabla \to 
		\DR_A \bigl( \cMmod \tensor_{\OA} (L, \nabla) \bigr),
\]
and therefore the desired relation between their hypercohomology groups.
\end{proof}

The second result is much stronger, and directly relates the two complexes 
$\FM_A(\Mmod)$ and $\derR \pl \bigl( \DR_A(\Mmod) \tensor_{\CC} \shLC \bigr)$. This
only makes sense on the level of coherent analytic sheaves, because $\Phi$ is not
an algebraic morphism.

\begin{theorem} \label{thm:relationship-FM}
Let $\Mmod \in \Dbh(\Dmod_A)$ be a holonomic complex on $A$. Then the complex of
coherent analytic sheaves $\derR \Phi_{\ast} \FM_A(\Mmod)$ is quasi-isomorphic to the
complex of coherent analytic sheaves determined by $\derR \pl \bigl( \DR_A(\Mmod)
\tensor_{\CC} \shLC \bigr)$.
\end{theorem}

The algebra of global holomorphic functions on the Stein manifold $\Char(A)$ is
naturally isomorphic to $\CcLambda$, the algebra of convergent Laurent series in the
variables $e_{\lambda}$; recall that $e_{\lambda}$ is the function taking a character
$\rho$ to $\rho(\lambda)$. Since $\CcLambda$ is in an evident manner a module over
$\CLambda$, it determines a local system $\shLCc$ on the abelian variety $A$.
From \propositionref{prop:iso-c}, we know that
\[
	\derR \pl \bigl( \DR_A(\Mmod) \tensor_{\CC} \shLC \bigr) 
		\tensor_{\CLambda} \CcLambda
		\simeq \derR \pl \bigl( \DR_A(\Mmod) \tensor_{\CC} \shLCc \bigr) 
\]
The idea of the proof of \theoremref{thm:relationship-FM} is to relate this
complex of finitely generated $\CcLambda$-modules to the complex of global sections
of $\derR \Phi_{\ast} \FM_A(\Mmod)$. 

We begin by recalling the analytic construction of the Poincar\'e bundle on $A \times
\Ash$. Since $\Lambda \tensor_{\ZZ} \RR = V$, the canonical isomorphism
\[
	\Hom_{\RR}(V, \CC) = \Hom_{\ZZ}(\Lambda, \CC)
\]
allows us to identify a group homomorphism $f \colon \Lambda \to \CC$ with the induced
$\RR$-linear functional $f \colon V \to \CC$. Now any such $f \in \Hom_{\RR}(V, \CC)$ 
gives rise to the translation-invariant complex-valued one-form $df$ on the abelian
variety $A$, and therefore determines a line bundle with integrable connection on $A$:
the underlying smooth line bundle is $A \times \CC$; the complex structure is given
by $\dbar - (df)^{0,1}$, and the connection by $\del - (df)^{1,0}$. Briefly, we say
that the line bundle with integrable connection is given by the operator $d - df$;
the local system of its flat sections corresponds to the character $\rho_f \colon
\Lambda \to \CCst$, $\rho_f(\lambda) = e^{f(\lambda)}$. We thus have a commutative
diagram:
\[
\begin{tikzcd}[column sep=large,row sep=tiny]
	& \Ash \arrow{dd}{\Phi} \\
	\Hom_{\RR}(V, \CC) 
		\arrow[bend left=10]{ur}[pos=0.8]{f \mapsto d-df} 
		\arrow[bend right=10]{dr}[pos=0.35]{f \mapsto e^f} & \\
	& \Char(A)
\end{tikzcd}
\]
Both horizontal arrows are surjective homomorphisms of complex Lie groups; their
kernels are equal to the subspace $\Hom_{\ZZ} \bigl( \Lambda, \ZZ(1)
\bigr)$, where $\ZZ(1) = 2 \pi i \cdot \ZZ \subseteq \CC$. Note that if $f(\Lambda)
\subseteq \ZZ(1)$, then $e^f$ descends to a nowhere vanishing real analytic function
on $A$. On the product $V \times \Hom_{\RR}(V, \CC)$, one has the tautological function
\[
	F \colon V \times \Hom_{\RR}(V, \CC) \to \CC, \quad F(v, f) = f(v).
\]

Now we can describe the Poincar\'e bundle $\Psh$ and the universal relative
connection $\nablash \colon \Psh \to \Omega_{A \times \Ash/\Ash}^1 \tensor \Psh$ on
$A \times \Ash$. The discrete group $\Lambda \times \Hom_{\ZZ} \bigl( \Lambda, \ZZ(1)
\bigr)$ acts smoothly on the trivial bundle $V \times \Hom_{\RR}(V, \CC)
\times \CC$ by the formula
\[
	(\lambda, \phi) \cdot (v, f, z) 
		= \left( v + \lambda, f + \phi, e^{-\phi(v)} z \right), 
\]
and the quotient is a smooth line bundle on $A \times \Ash$. The original bundle
has both a complex structure and a relative integrable connection, defined by the
operator
\[
	\bigl( d_V - d_V F \bigr) + \dbar_{\Hom_{\RR}(V, \CC)};
\]
one can show that the operator descends to the quotient, and endows the smooth line
bundle from above with a complex structure and a connection relative to $\Ash$. This
gives a useful model for the  pair $(\Psh, \nablash)$.

Our first concern is to describe the subsheaf $\ker \nablash$ of $\Psh$. To that end,
we introduce an auxiliary sheaf $\shP$ on $A \times \Char(A)$ by defining
\[
	\shP(U) = 
		\mengge{s \in \shO_{V \times \Char(A)} \bigl( (\pi \times \id)^{-1}(U) \bigr)}{%
			\parbox[c]{.415\textwidth}{%
				$s(v+\lambda, \rho) = \rho(\lambda) s(v, \rho)$ for 
				$\lambda \in \Lambda$ \\ and $s(v, \rho)$ is locally constant in $v$}}.
\]

\begin{lemma} \label{lem:shP}
One has $R^i (p_1)_{\ast} \shP = 0$ for $i > 0$, and the sheaf $(p_1)_{\ast} \shP$ is
canonically isomorphic to $\shLCc$.
\end{lemma}

\begin{proof}
Let $\pi \colon V \to A$ denote the quotient mapping. As with any representation of
$\Lambda$, the sections of the local system $\shLCc$ over an open set $U \subseteq A$
are given by 
\[
	\shLCc(U) = \mengge{\ell \colon \pi^{-1}(U) \to \CcLambda}{%
		\parbox[c]{.33\textwidth}{%
		$\ell$ is locally constant, and \\
		$\ell(v + \lambda) = e_{\lambda} \ell(v)$ for $\lambda \in \Lambda$}}.
\]
On the other hand, a section of $(p_1)_{\ast} \shP$ over $U$ is a holomorphic function 
\[
	s \in \shO_{V \times \Char(A)} \bigl( \pi^{-1}(U) \times \Char(A) \bigr)
\]
that satisfies $s(v + \lambda, \rho) = \rho(\lambda) s(v, \rho)$ and is
locally constant in its first argument. Consequently, $\ell(v) = f(v, \argbl)$ is a
holomorphic function on $\Char(A)$ and therefore an element of the ring $\CcLambda$.
The resulting function $\ell \colon \pi^{-1}(U) \to \CcLambda$ is locally constant and
satisfies $\ell(v + \lambda) = e_{\lambda} \ell(v)$, which means that $\ell \in
\shLCc(U)$. In this way, we obtain an isomorphism of sheaves between $(p_1)_{\ast}
\shP$ and $\shLCc$.

To prove the vanishing of the higher direct image sheaves of $\shP$, it suffices to
show that $H^i \bigl( D \times \Char(A), \shP \bigr) = 0$ whenever $D \subseteq A$ is
a sufficiently small polydisk. Since $D$ is connected, one has
\[
	\shP(D \times U) \simeq \shO_{\Char(A)}(U),
\]
and because every open covering of $D \times \Char(A)$ can be refined into an open
covering by sets of this type, the desired vanishing follows from the fact that
$\Char(A)$ is a Stein manifold.
\end{proof}

The sheaf $\shP$ is related to $\ker \nablash$ in the following simple way.

\begin{lemma} \label{lem:nablash}
The subsheaf $\ker \nablash \subseteq \Psh$ has the following properties:
\begin{aenumerate}
\item \label{en:nablash-a}
$(\ker \nablash) \tensor_{\CC} \shO_{A \times \Ash} = \Psh$.
\item \label{en:nablash-b}
Its pushforward via $\id \times \Phi \colon A \times \Ash \to A \times \Char(A)$ is
canonically isomorphic to the sheaf $\shP$.
\end{aenumerate}
\end{lemma}

\begin{proof}
We use the analytic description of the Poincar\'e bundle above. It suffices to prove
\ref{en:nablash-a} after pulling back to the universal covering space $V \times
\Hom_{\RR}(V, \CC)$; but there, it becomes obvious, because the function $e^F$, which
takes $(v, f)$ to $e^{f(v)}$, is a nowhere vanishing global section of the
pullback of $\ker \nablash$.

To prove \ref{en:nablash-b}, we shall exhibit a canonical isomorphism of sheaves
\[
	\shP \to (\id \times \Phi)_{\ast}(\ker \nablash).
\]
Fix an open subset $U \subseteq A \times \Char(A)$. By definition, a section of
$\shP$ over $U$ is a holomorphic function $s \colon (\pi \times \id)^{-1}(U) \to
\CC$ that is locally constant in its first argument and satisfies $s(v + \lambda,
\rho) = \rho(\lambda) s(v, \rho)$ for every $\lambda \in \Lambda$. We claim that the
function $\sigma(f, v) = e^{-f(v)} s(v, e^f)$ is a section of $\ker \nablash$ over
$(\id \times \Phi)^{-1}(U)$. This claim is easily verified: indeed, the properties of $s$
imply that 
\[
	\bigl( d_V - d_V F + \dbar_{\Hom_{\RR}(V, \CC)} \bigr) \sigma = 0;
\]
they also imply that, for every $(\lambda, \phi) \in \Lambda \times \Hom_{\ZZ} \bigl(
\Lambda, \ZZ(1) \bigr)$, 
\begin{align*}
	\sigma(v + \lambda, f + \phi) 
		&= e^{-(f + \phi)(v + \lambda)} s(v + \lambda, e^f) \\
		&= e^{-(f + \phi)(v + \lambda)} e^{f(\lambda)} s(v, e^f)
		= e^{-\phi(v)} \sigma(v, f),
\end{align*}
because $e^{\phi}$ restricts to the trivial character on $\Lambda$. The construction
is obviously reversible, and so we obtain an isomorphism
\[
	\shP(U) \to (\ker \nablash) \bigl( (\id \times \Phi)^{-1}(U) \bigr),
\]
which induces the desired isomorphism of sheaves.
\end{proof}

We can now complete the proof of \theoremref{thm:relationship-FM} by imitating
the argument from the set-theoretic result about cohomology support loci.

\begin{proof}[Proof of \theoremref{thm:relationship-FM}]
Let $(\cMmod, d)$ be the given holonomic complex; to shorten the notation, set $K =
\DR_A(\cMmod)$. Let $p_1 \colon A \times \Ash \to A$ denote the first projection. From
\lemmaref{lem:nablash}, we obtain canonical isomorphisms
\[
	p_1^{\ast} \bigl( \OmA^{g+i} \tensor \Mmod^j \bigr) \tensor_{\CC} \ker \nablash
		\simeq \Omega_{A \times \Ash/\Ash}^{g+i} \tensor p_1^{\ast} \Mmod^j \tensor \Psh.
\]
They are compatible with the differentials in both double complexes, and therefore
induce an isomorphism of complexes
\[
	p_1^{\ast} \DR_A(\cMmod) \tensor_{\CC} \ker \nablash \simeq 
		\DR_{A \times \Ash / \Ash} \bigl( \cMmod \tensor_{\shO_{A \times \Ash}} 
		(\Psh, \nablash) \bigr).
\]
As a complex of coherent analytic sheaves, $\derR \Phi_{\ast} \FM_A(\cMmod)$ is
thus isomorphic to
\begin{align*}
	\derR \Phi_{\ast} \derR (p_2)_{\ast} \bigl( p_1^{\ast} K
		\tensor_{\CC} \ker \nablash \bigr)
	&\simeq \derR (p_2)_{\ast} \derR (\id \times \Phi)_{\ast} \bigl( p_1^{\ast} K
		\tensor_{\CC} \ker \nablash \bigr) \\
	&\simeq \derR (p_2)_{\ast} \bigl( p_1^{\ast} K
		\tensor_{\CC} \derR (\id \times \Phi)_{\ast}(\ker \nablash) \bigr) \\
	&\simeq \derR (p_2)_{\ast} \bigl( p_1^{\ast} K
		\tensor_{\CC} \shP \bigr).
\end{align*}
If we now push forward to a point, and invoke the result of \lemmaref{lem:shP}, we
obtain a complex of $\CcLambda$-modules quasi-isomorphic to
\[
	\derR \pl \bigl( K \tensor_{\CC} \shLCc \bigr) \simeq
		\derR \pl \bigl( K \tensor_{\CC} \shLC \bigr) \tensor_{\CLambda} \CcLambda,
\]
where $p \colon \Char(A) \to \pt$ denotes the morphism to a point. Because
$\Char(A)$ is a Stein manifold, this suffices to conclude the proof.
\end{proof}

\subsection{Structure theorem}
\label{subsec:structure}

The goal of this section is to prove \theoremref{thm:h-linear} and
\theoremref{thm:c-linear}, which together describe the structure of cohomology
support loci for holonomic complexes of $\Dmod_A$-modules and for constructible complexes
of sheaves of $\CC$-vector spaces.

Our proof is based on the observation that $\Phi \colon \Ash \to \Char(A)$ is an
isomorphism of complex Lie groups, but not of complex algebraic varieties. (In fact,
$\Ash$ does not have any non-constant global algebraic functions
\cite{Laumon}*{Th\'eor\`eme~2.4.1}, whereas $\Char(A)$ is affine.) Linear
subvarieties are clearly algebraic in both models, because
\[
	\Phi \Bigl( (L, \nabla) \tensor 
		\im \bigl( \fsh \colon \Bsh \to \Ash \bigr) \Bigr) = 
		 \Phi(L, \nabla) \cdot \im \bigl( \Char(f) \colon \Char(B) \to \Char(A) \bigr).
\]
Finite unions of linear subvarieties are the only closed subsets
with this property; this is the content of the following result by Simpson
\cite{Simpson}*{Theorem~3.1}.

\begin{theorem}[Simpson] \label{thm:simpson}
Let $Z$ be a closed algebraic subset of $\Ash$. If $\Phi(Z)$ is again a closed
algebraic subset of $\Char(A)$, then $Z$ is a finite union of linear subvarieties of
$\Ash$, and $\Phi(Z)$ is a finite union of linear subvarieties of $\Char(A)$.
\end{theorem}

In \theoremref{thm:alg-c}, we have already seen that cohomology support loci are
algebraic subsets of $\Char(A)$. With the help of the Fourier-Mukai transform, it is
easy to show that they are also algebraic subsets of $\Ash$.

\begin{proposition} \label{prop:alg-h}
If $\Mmod \in \Dbcoh(\Dmod_A)$, then each cohomology support locus $S_m^k(A, \Mmod)$ is
an algebraic subset of $\Ash$.
\end{proposition}

\begin{proof}
Since $\Ash$ is a quasi-projective algebraic variety, we may represent $\FM_A(\Mmod)$
by a bounded complex $(\shE^{\bullet}, d)$ of locally free sheaves on $\Ash$. Now let
$(L, \nabla)$ be a line bundle with integrable connection, and let $i_{(L, \nabla)}$
denote the inclusion morphism. By the base change theorem,
\[
	\derR i_{(L, \nabla)}^{\ast} \FM_A(\Mmod) 
		\simeq \DR_A \bigl( \Mmod \tensor_{\OA} (L, \nabla) \bigr),
\]
and so we have
\[
	S_m^k(A, \Mmod) = \Menge{(L, \nabla) \in \Ash}%
		{\dim \derH^k \bigl( i_{(L, \nabla)}^{\ast} (\shE^{\bullet}, d) \bigr) \geq m}.
\]
Because of \lemmaref{lem:perfect}, this description shows that $S_m^k(A, \Mmod)$ is an
algebraic subset of $\Ash$, as claimed.
\end{proof}

We can now prove the two structure theorems from the introduction.

\begin{proof}[Proof of \theoremref{thm:h-linear}]
Let $\Mmod \in \Dbh(\Dmod_A)$ be a holonomic complex. Then $\DR_A(\Mmod)$ is
constructible, and we have
\[
	\Phi \bigl( S_m^k(A, \Mmod) \bigr) = S_m^k \bigl( A, \DR_A(\Mmod) \bigr)
\]
by \theoremref{thm:relationship}. \propositionref{prop:alg-h} shows that $S_m^k(A, \Mmod)$
is an algebraic subset of $\Ash$; \theoremref{thm:alg-c} shows that $S_m^k \bigl(
\DR_A(\Mmod) \bigr)$ is an algebraic subset of $\Char(A)$. We conclude from Simpson's
\theoremref{thm:simpson} that both must be finite unions of linear subvarieties of
$\Ash$ and $\Char(A)$, respectively. The assertion about objects of geometric origin
is proved in \subsecref{subsec:geometric} below.
\end{proof}

\begin{proof}[Proof of \theoremref{thm:FM-linear}]
This follows from \theoremref{thm:relationship-FM} and Simpson's results by the
same argument as the one just given.
\end{proof}

\subsection{Objects of geometric origin}
\label{subsec:geometric}

In this section, we study cohomology support loci for semisimple regular holonomic
$\Dmod_A$-modules of geometric origin, as defined in \cite{BBD}*{6.2.4}. To begin
with, recall the following definition for mixed Hodge modules, due to Saito
\cite{Saito-MM}*{Definition~2.6}.

\begin{definition}
A mixed Hodge module is said to be \define{of geometric origin} if it is 
obtained by applying several of the standard cohomological functors $H^i \fl$,
$H^i f_!$, $H^i \fu$, $H^i f^!$, $\psi_g$, $\phi_{g, 1}$, $\mathbf{D}$, $\boxtimes$,
$\oplus$, $\otimes$, and $\shHom$ to the trivial Hodge structure $\QQ^H$ of weight
zero, and then taking subquotients in the category of mixed Hodge modules.
\end{definition}

One of the results of Saito's theory is that any semisimple perverse sheaf of
geometric origin, in the sense of \cite{BBD}*{6.2.4}, is a direct summand of a
perverse sheaf underlying a mixed Hodge module of geometric origin.
Consequently, any semisimple regular holonomic $\Dmod$-module of geometric origin is
a direct summand of a $\Dmod$-module that underlies a polarizable Hodge module of
geometric origin.

\begin{theorem} \label{thm:geom-origin}
Let $\Mmod$ be a semisimple regular holonomic $\Dmod_A$-module of geometric origin.
Then each cohomology support locus $S_m^k(A, \Mmod)$ is a finite union of arithmetic
linear subvarieties of $\Ash$.
\end{theorem}

We introduce some notation that will be used during the proof. For any field
automorphism $\sigma \in \Aut(\CC/\QQ)$, we obtain from $A$ a new complex abelian
variety $\Asig$. Likewise, an algebraic line bundle $(L, \nabla)$ with integrable
connection on $A$ gives rise to $(\Lsig, \nablasig)$ on $\Asig$, and so we have a
natural isomorphism of abelian groups
\[
	\csig \colon \Ash \to \Asigsh.
\]
Now recall the following notion, due in a slightly different form to Simpson, who
modeled it on Deligne's definition of absolute Hodge classes.

\begin{definition} \label{def:absolute}
A closed subset $Z \subseteq \Ash$ is said to be \define{absolute closed} if, for
every field automorphism $\sigma \in \Aut(\CC/\QQ)$, the set
\[
	\Phi \bigl( \csig(Z) \bigr) \subseteq \Char(\Asig)
\]
is closed and defined over $\QQb$.
\end{definition}

The following theorem about absolute closed subsets is also due to Simpson.

\begin{theorem}[Simpson]
An absolute closed subset of $\Ash$ is a finite union of arithmetic linear
subvarieties.
\end{theorem}

\begin{proof}
Simpson's definition \cite{Simpson}*{p.~376} of absolute closed sets actually
contains several additional conditions (related to the space of Higgs bundles); but
as he explains, a strengthening of \cite{Simpson}*{Theorem~3.1}, added in proof,
makes these conditions unnecessary. In fact, the proof of
\cite{Simpson}*{Theorem~6.1} goes through unchanged with only the
assumptions in \definitionref{def:absolute}.
\end{proof}

With the help of Simpson's result, the proof of \theoremref{thm:geom-origin} is
straightforward. We first establish the following lemma.

\begin{lemma} \label{lem:MHM}
Let $M \in \MHM(A)$ be a mixed Hodge module, with underlying filtered
$\Dmod_A$-module $(\Mmod, F)$. Then the cohomology support loci of the
perverse sheaf $\DR_A(\Mmod)$ are algebraic subsets of $\Char(A)$ that are defined
over $\QQb$.
\end{lemma}

\begin{proof}
By definition, a mixed Hodge module has an underlying perverse sheaf $\rat M$ with
coefficients in $\QQ$, and $\DR_A(\Mmod) \simeq (\rat M) \tensor_{\QQ} \CC$. By
\propositionref{prop:subfield}, it follows that $\derR \pl \bigl( \DR_A(\Mmod)
\tensor_{\CC} \shL_R \bigr) \in \Dbcoh(R)$ is obtained by extension of scalars from
an object of $\Dbcoh \bigl( \QQ \lbrack \Lambda \rbrack \bigr)$.
The assertion about cohomology support loci now follows easily from
\corollaryref{cor:fibers}.
\end{proof}

The same result is true for any holonomic $\Dmod_A$-module with
$\QQb$-structure; that is to say, for any holonomic $\Dmod_A$-module whose de Rham
complex is the complexification of a perverse sheaf with coefficients in $\QQb$.
This is what Mochizuki calls a ``pre-Betti structure'' in \cite{Mochizuki-B}.

\begin{lemma} \label{lem:subquotient}
Let $K \in \Dbc(\QQb_A)$ be a perverse sheaf with coefficients in $\QQb$. Any
simple subquotient of $K \tensor_{\QQb} \CC$ is the complexification of a simple
subquotient of $K$.
\end{lemma}

\begin{proof}
We only have to show that if $K \in \Dbc(\QQb_A)$ is a simple perverse sheaf, then $K
\tensor_{\QQb} \CC \in \Dbc(\CC_A)$ is also simple. By the classification of simple
perverse sheaves, there is an irreducible locally closed subvariety $X \subseteq A$,
and an irreducible representation $\rho \colon \pi_1(X) \to \GL_n(\QQb)$, such that
$K$ is the intermediate extension of the local system associated to $\rho$. Since
$\QQb$ is algebraically closed, $\rho$ remains irreducible over $\CC$,
proving that $K \tensor_{\QQb} \CC$ is still simple.
\end{proof}

\begin{proof}[Proof of \theoremref{thm:geom-origin}]
We first show that this holds when $\Mmod$ underlies a mixed Hodge module $M$
obtained by iterating the standard cohomological functors (but without taking
subquotients). Fix two integers $k, m$, and set $Z = S_m^k(A, \Mmod)$. In light of
\lemmaref{lem:MHM}, it suffices to prove that each set $\csig(Z)$ is equal to
$S_m^k(\Asig, \Mmodsig)$ for some polarizable Hodge module $\Msig \in \MHM(\Asig)$. But
since $M$ is of geometric origin, this is obviously the case; indeed, we can obtain
$\Msig$ by simply applying $\sigma$ to the finitely many algebraic varieties and
morphisms involved in the construction of $M$. 

Now suppose that $\Mmod$ is an arbitrary semisimple regular holonomic
$\Dmod_A$-module of geometric origin. Then $\Mmod$ is a direct sum of simple
subquotients of $\Dmod_A$-modules underlying mixed Hodge modules of geometric origin.
By the same argument as before, it suffices to show that $\Mmodsig$ is defined over
$\QQb$ for any $\sigma \in \Aut(\CC/\QQ)$. Now the perverse sheaf $\DR_A(\Mmodsig)$
is again a direct sum of simple subquotients of perverse sheaves underlying mixed Hodge
modules; by \lemmaref{lem:subquotient}, it is therefore the complexification of a
perverse sheaf with coefficients in $\QQb$. We then conclude the proof as above.
\end{proof}

\section{Codimension bounds and perverse coherent sheaves}

After a brief review of perverse coherent sheaves, we study the problem of how the
Fourier-Mukai transform for holonomic complexes interacts with various t-structures.

\subsection{Perverse coherent sheaves}
\label{subsec:perverse}

Let $X$ be a smooth complex algebraic variety. In this section, we recall the
construction of perverse t-structures on the bounded derived category $\Dbcoh(\OX)$
of coherent algebraic sheaves, following \cite{Kashiwara}.  For a (possibly
non-closed) point $x$ of the scheme $X$, we denote the residue field at the point by
$\kappa(x)$, the inclusion morphism by $i_x \colon \Spec \kappa(x) \into X$,
and the codimension of the closed subvariety $\overline{ \{x\} }$ by $\codim(x) =
\dim \shO_{X,x}$.

\begin{definition} 
A \define{supporting function} on $X$ is a function $p \colon X \to \ZZ$ from the
underlying topological space of the scheme $X$ to the set of integers, with the property
that $p(y) \geq p(x)$ whenever $y \in \overline{ \{x\} }$. 
\end{definition}
Given a supporting function, Kashiwara defines two families of subcategories
\begin{align*}
	\pDtcoh{p}{\leq k}(\OX) &= 
		\menge{F \in \Dbcoh(\OX)}%
			{\text{$\derL i_x^{\ast} F \in \Dtcoh{\leq k+p(x)} %
			\bigl( \kappa(x) \bigr)$ for all $x \in X$}}, \\
	\pDtcoh{p}{\geq k}(\OX) &= 
		\menge{F \in \Dbcoh(\OX)}%
			{\text{$\derR i_x^! F \in \Dtcoh{\geq k+p(x)} %
			\bigl( \kappa(x) \bigr)$ for all $x \in X$}}.
\end{align*}
The following fundamental result is proved in \cite{Kashiwara}*{Theorem~5.9} and, by
a different method, in \cite{AB}*{Theorem 3.10}.

\begin{theorem}[Kashiwara, Arinkin-Bezrukavnikov] \label{thm:kashiwara}
The above subcategories define a bounded t-structure on $\Dbcoh(\OX)$ if, and only
if, the supporting function has the property that 
\[	
	p(y) - p(x) \leq \codim(y) - \codim(x)
\]
for every pair of (possibly non-closed) points $x,y \in X$ with $y \in \overline{
\{x\} }$.
\end{theorem}

For example, the function $p=0$ corresponds to the standard t-structure on
$\Dbcoh(\OX)$. An equivalent way of putting the condition in
\theoremref{thm:kashiwara} is that the dual function $\hat{p}(x) = \codim(x) - p(x)$
should again be a supporting function. If that is the case, one has the identities
\begin{align*}
	\pDtcoh{\hat{p}}{\leq k}(\OX) &=
		 \derR \shHom \Bigl( \pDtcoh{p}{\geq -k}(\OX), \OX \Bigr) \\
	\pDtcoh{\hat{p}}{\geq k}(\OX) &=
		 \derR \shHom \Bigl( \pDtcoh{p}{\leq -k}(\OX), \OX \Bigr),
\end{align*}
which means that the duality functor $\derR \shHom(\argbl, \OX)$ exchanges the
two perverse t-structures defined by $p$ and $\hat{p}$. 

\begin{definition}
The heart of the t-structure defined by $p$ is denoted
\[
	\pCoh{p}(\OX) = \pDtcoh{p}{\leq 0}(\OX) \cap \pDtcoh{p}{\geq 0}(\OX),
\]
and is called the abelian category of \define{$p$-perverse coherent sheaves}. 
\end{definition}

We are interested in a special case of \theoremref{thm:kashiwara}, namely the set of
objects $F \in \Dbcoh(\OX)$ with $\codim \Supp \shH^i(F) \geq 2i$ for all $i \geq 0$.
Define a function
\[
	m \colon X \to \ZZ, \quad 
		m(x) = \left\lfloor \tfrac{1}{2} \codim(x) \right\rfloor.
\]
It is easily verified that both $m$ and the dual function
\[
	\mh \colon X \to \ZZ, \quad
		\mh(x) = \left\lceil \tfrac{1}{2} \codim(x) \right\rceil
\]
are supporting functions. As a consequence of \theoremref{thm:kashiwara}, $m$ defines a
bounded t-structure on $\Dbcoh(\OX)$; objects in the heart $\mCoh(\OX)$ will be
called \define{$m$-perverse coherent sheaves}.%
\footnote{%
We use this letter because $m$ and $\mh$ are as close as one can get to ``middle
perversity''. There is of course no actual middle perversity for coherent sheaves,
because the equality $p = \hat{p}$ cannot hold unless $X$ is a point.%
}

The next lemma follows easily from \cite{Kashiwara}*{Lemma~5.5}.
\begin{lemma} \label{lem:m-structure}
The perverse t-structures defined by $m$ and $\mh$ satisfy
\begin{align*}
	\mDtcoh{\leq k}(\OX) &= \menge{F \in \Dbcoh(X)}%
		{\text{$\codim \Supp \shH^i(F) \geq 2(i-k)$ for all $i \in \ZZ$}} \\
	\mhDtcoh{\leq k}(\OX) &= \menge{F \in \Dbcoh(X)}%
		{\text{$\codim \Supp \shH^i(F) \geq 2(i-k)-1$ for all $i \in \ZZ$}}.
\end{align*}
By duality, this also describes the subcategories with $\geq k$.
\end{lemma}

Consequently, an object $F \in \Dbcoh(\OX)$ is an $m$-perverse coherent sheaf
precisely when $\codim \Supp \shH^i(F) \geq 2i$ and $\codim \Supp R^i \shHom(F, \OX)
\geq 2i-1$ for every integer $i \geq 0$. This shows one more time that the category
of $m$-perverse coherent sheaves is not preserved by the duality functor $\derR
\shHom(\argbl, \OX)$.

\begin{lemma} \label{lem:m-geq}
If $F \in \mDtcoh{\geq 0}(\OX)$ or $F \in \mhDtcoh{\geq 0}(\OX)$, then $F \in
\Dtcoh{\geq 0}(\OX)$.
\end{lemma}

\begin{proof}
This is obvious from the fact that $\mh \geq m \geq 0$. 
\end{proof}

When it happens that both $F$ and $\derR \shHom(F, \OX)$ are $m$-perverse coherent
sheaves, $F$ has surprisingly good properties.

\begin{proposition} \label{prop:surprise}
If $F \in \mDtcoh{\leq 0}(\OX)$ satisfies $\derR \shHom(F, \OX) \in \mDtcoh{\leq
0}(\OX)$, then it has the following properties:
\begin{renumerate}
\item Both $F$ and $\derR \shHom(F, \OX)$ belong to $\mCoh(\OX)$.
\item Let $r \geq 0$ be the least integer with $\shH^r(F) \neq 0$; then every
irreducible component of $\Supp \shH^r(F)$ has codimension $2r$.
\end{renumerate}
\end{proposition}

\begin{proof}
The first assertion follows directly from \lemmaref{lem:m-structure}. To prove the
second assertion, note that we have $\codim \Supp \shH^r(F) \geq 2r$. After
restricting to an open neighborhood of the generic point of any given irreducible
component, it therefore suffices to show that if $\shH^i(F) = 0$ for $i < r$,
and $\codim \Supp \shH^r(F) > 2r$, then $\shH^r(F) = 0$. Under these assumptions, we
have
\[
	\codim \Supp \shH^i(F) \geq \max(2i,2r+1) \geq i+r+1, 
\]
and therefore $\derR \shHom(F, \OX) \in \Dtcoh{\geq r+1}(\OX)$ by
\cite{Kashiwara}*{Proposition~4.3}. The same argument, applied to $\derR \shHom(F,
\OX)$, now shows that $F \in \Dtcoh{\geq r+1}(\OX)$.
\end{proof}

\subsection{Description of the t-structure}
\label{subsec:codimension}

In this section, we show that the standard t-structure on $\Dbh(\Dmod_A)$
corresponds, under the Fourier-Mukai transform $\FM_A$, to the $m$-perverse t-structure.

\begin{theorem} \label{thm:t-structure}
Let $\Mmod \in \Dbh(\Dmod_A)$ be a holonomic complex on $A$. Then one has
\begin{align*}
	\Mmod \in \Dth{\leq k}(\Dmod_A) \quad &\Longleftrightarrow \quad 
		\FM_A(\Mmod) \in \mDtcoh{\leq k}(\OAsh), \\
	\Mmod \in \Dth{\geq k}(\Dmod_A) \quad &\Longleftrightarrow \quad 
		\FM_A(\Mmod) \in \mDtcoh{\geq k}(\OAsh).
\end{align*}
\end{theorem}

The first step of the proof consists in the following ``generic vanishing theorem''
for holonomic $\Dmod_A$-modules. In the regular case, this result
is due to Kr\"amer and Weissauer \cite{KW}*{Theorem~2}, whose argument relies on the
(difficult) recent proof of Kashiwara's conjecture for semisimple perverse sheaves. By
contrast, our argument is completely elementary.

\begin{proposition} \label{prop:KW}
Let $\Mmod$ be a holonomic $\Dmod_A$-module. Then for every $i > 0$, the
support of the coherent sheaf $\shH^i \FM_A(\Mmod)$ is a proper subset of $\Ash$.
\end{proposition}

\begin{proof}
Let $F_{\bullet} \Mmod$ be a good filtration by $\OA$-coherent subsheaves; this
exists by \cite{HTT}*{Theorem~2.3.1}. As in \subsecref{subsec:compatible}, we consider the
associated coherent $\Rmod_A$-module $R_F \Mmod$ defined by the Rees construction,
and its Fourier-Mukai transform
\[
	\FMt_A(R_F \Mmod) \in \Dbcoh \bigl( \shO_{E(A)} \bigr).
\]
By \propositionref{prop:compatible}, $\FMt_A(R_F \Mmod)$ is equivariant for the
$\CC^{\ast}$-action on $E(A)$, and for any $z \neq 0$, its restriction to
$\lambda^{-1}(z) \simeq \Ash$ is isomorphic to $\FM_A(\Mmod)$. It is therefore
sufficient to prove that the restriction of $\FMt_A(R_F \Mmod)$ to $\lambda^{-1}(0) =
A \times H^0(A, \Omega_A^1)$ has the asserted property. By
\propositionref{prop:compatible}, this restriction is isomorphic to 
\begin{equation} \label{eq:formula}
	\derR (p_{23})_{\ast} \Bigl( p_{12}^{\ast} P \tensor 
		p_{13}^{\ast} \gr^F \!\! \Mmod \tensor p_1^{\ast} \Omega_A^g \Bigr),
\end{equation}
where the notation is as in \eqref{eq:diag-cotangent}. Now $\Mmod$ is holonomic, and
so each irreducible component of $\Supp(\gr^F \!\! \Mmod)$ has dimension $g$. The
support of $p_{13}^{\ast} \gr^F \!\! \Mmod$ therefore has the same dimension
as $\Ah \times H^0(A, \Omega_A^1)$, which implies that the support of the higher
direct image sheaves in \eqref{eq:formula} must be a proper subset of $\Ah \times
H^0(A, \Omega_A^1)$.
\end{proof}

Together with the structure theory for cohomology support loci and basic properties
of the Fourier-Mukai transform, this result now allows us to prove the first
equivalence asserted in \theoremref{thm:t-structure}.

\begin{lemma} \label{lem:t-structure1}
For any $\Mmod \in \Dth{\leq k}(\Dmod_A)$, one has $\FM_A(\Mmod) \in \mDtcoh{\leq k}(\OAsh)$.
\end{lemma}

\begin{proof}
The proof is by induction on $\dim A$, the statement being obviously true when $A$ is
a point. Since $\FM_A$ is triangulated, it suffices to prove the statement for $k=0$.
According to \lemmaref{lem:m-structure}, what we then need to show is the 
following: for any holonomic complex $\Mmod \in \Dth{\leq 0}(\Dmod_A)$ concentrated in
nonpositive degrees, the Fourier-Mukai transform $\FM_A(\Mmod)$ satisfies, for every
$\ell \geq 1$, the inequality
\[
	\codim \Supp \shH^{\ell} \FM_A(\Mmod) \geq 2 \ell.
\]
Let $Z$ be an any irreducible component of $\Supp \shH^{\ell} \FM_A(\Mmod)$,
for some $\ell \geq 1$. By \theoremref{thm:FM-linear}, $Z$ is a linear subvariety of
$\Ash$, hence of the form $Z = t_{(L, \nabla)}(\im \fsh)$ for a
surjective morphism $f \colon A \to B$ and a suitable point $(L, \nabla) \in \Ash$.
Furthermore, \propositionref{prop:KW} shows that $\codim Z > 0$, and therefore $\dim B
< \dim A$. Setting $r = \dim A - \dim B > 0$, we thus have $\codim Z = 2r$.

Using the properties of the Fourier-Mukai transform listed in
\theoremref{thm:Laumon}, we find that the pullback of $\FM_A(\Mmod)$ to the
subvariety $Z$ is isomorphic to
\[
	\derL (\fsh)^{\ast} \derL t_{(L, \nabla)}^{\ast} \FM_A(\Mmod) \\
		\simeq \FM_B \Bigl( \fp \bigl( \Mmod \tensor_{\OA} (L, \nabla) \bigr) \Bigr)
		\in \Dbcoh(\OBsh).
\]
From the definition of the direct image functor $\fp$, it is clear that $\fp \bigl(
\Mmod \tensor_{\OA} (L, \nabla) \bigr)$ belongs to the subcategory $\Dth{\leq
r}(\Dmod_B)$. The inductive assumption now allows us to
conclude that the restriction of $\FM_A(\Mmod)$ to $Z$ lies in the subcategory $\mDtcoh{\leq
r}(\shO_Z)$. But $Z$ is an irreducible component of $\Supp \shH^{\ell}
\FM_A(\Mmod)$; it follows that $\ell \leq r$, and consequently $\codim Z \geq 2\ell$,
as asserted.
\end{proof}

\begin{lemma} \label{lem:t-structure2}
Let $\Mmod \in \Dbh(\Dmod_A)$ be a holonomic complex. 
If its Fourier-Mukai transform satisfies $\FM_A(\Mmod) \in \Dtcoh{\leq k}(\OAsh)$,
then $\Mmod \in \Dth{\leq k}(\Dmod_A)$.
\end{lemma}

\begin{proof}
It again suffices to prove this for $k=0$. By \cite{Laumon}*{Th\'eor\`eme~3.2.1}, we 
can recover $\Mmod$ -- up to canonical isomorphism -- from its Fourier-Mukai transform as
\[
	\Mmod = \langle -1_A \rangle^{\ast} \derR (p_1)_{\ast} \Bigl(
		\Psh \tensor_{\shO_{A \times \Ash}} p_2^{\ast} \FM_A(\Mmod) \Bigr) \decal{g},
\]
where $p_1 \colon A \times \Ash \to A$ and $p_2 \colon A \times \Ash \to \Ash$ are
the two projections. If we forget about the $\Dmod_A$-module structure and only
consider $\Mmod$ as a complex of quasi-coherent sheaves of $\OA$-modules, we can use
the fact that $\pi \colon \Ash \to A$ is affine to obtain
\[
	\Mmod = \langle -1_A \rangle^{\ast} \derR (p_1)_{\ast} \Bigl(
		 P \tensor_{\shO_{A \times \Ah}} p_2^{\ast} \, \pil \FM_A(\Mmod) \Bigr) \decal{g},
\]
where now $p_1 \colon A \times \Ah \to A$ and $p_2 \colon A \times \Ah \to \Ah$.
By virtue of \theoremref{thm:FM-linear}, every irreducible
component of $\Supp \shH^{\ell} \FM_A(\Mmod)$ is contained in a linear subvariety of
codimension at least $2\ell$; consequently, every irreducible component of $\Supp \pil
\shH^{\ell} \FM_A(\Mmod)$ still has codimension at least $\ell$. From this, it is
easy to see that $\shH^i \Mmod = 0$ for $i > 0$, and hence that $\Mmod \in \Dth{\leq
0}(\Dmod_A)$.
\end{proof}

\begin{proof}[Proof of \theoremref{thm:t-structure}]
The first equivalence is proved in \lemmaref{lem:t-structure1} and
Lemma \ref{lem:t-structure2} above. The second equivalence follows from this by duality,
using the compatibility of the Fourier-Mukai transform with the duality functors for
$\Dmod_A$-modules and $\shO_{\Ash}$-modules (see \theoremref{thm:Laumon}).
\end{proof}

\subsection{Stability under truncation}
\label{subsec:truncation}

If the proposal in \subsecref{subsec:intro-conjecture} has merit, and Fourier-Mukai
transforms of holonomic complexes are indeed ``hyperk\"ahler constructible
complexes'', then a complex of coherent sheaves on $\Ash$ should belong to the
subcategory
\[
	\FM_A \bigl( \Dbh(\Dmod_A) \bigr) \subseteq \Dbcoh(\OAsh)
\]
if and only if all of its cohomology sheaves do. This is because ``constructibility''
of a complex should be defined in terms of the cohomology sheaves. 

In this section, we prove that this is indeed the case.  The result is that the
subcategory $\FM_A \bigl( \Dbh(\Dmod_A) \bigr)$ is closed under the truncation
functors $\tau_{\leq n}$ and $\tau_{\geq n}$ for the \emph{standard} t-structure on
$\Dbcoh(\OAsh)$. This is of course equivalent to the statement about cohomology
sheaves, but more convenient for doing induction.

\begin{theorem} \label{thm:truncation}
Let $F = \FM_A(\Mmod)$ for some $\Mmod \in \Dbh(\Dmod_A)$. Then for every $n \in
\ZZ$, the truncations $\tau_{\leq n} F$ and $\tau_{\geq n} F$ are again Fourier-Mukai
transforms of holonomic complexes.
\end{theorem}

\begin{proof}
It suffices to show the assertion for $\tau_{\leq 0} F$. Since $\FM_A(\tau_{\geq 1}
\Mmod) \in \Dtcoh{\geq 1}(\OAsh)$ by \lemmaref{lem:m-geq}, we may replace $\Mmod$ by
$\tau_{\leq 0} \Mmod$ and assume without loss of generality that $\Mmod \in \Dth{\leq
0}(\Dmod_A)$. Since $\codim \Supp \shH^i F \geq 2i$, each $\shH^i F$ with $i \geq 1$
is then a torsion sheaf, supported in a finite union of linear subvarieties of
$\Ash$. We shall argue that, by suitably modifying $\Mmod$, it is possible to remove
these torsion sheaves from the picture, leaving us with $\tau_{\leq 0} F$. 

To measure the difference between $F$ and $\tau_{\leq 0} F$, we introduce the set 
\[
	N(F) = \bigcup_{i \geq 1} \Supp \shH^i F, 
\] 
which is again a finite union of linear subvarieties of $\Ash$. Our goal is to reduce
the size of the set $N(F)$, because
$N(F)$ is empty if and only if $\tau_{\leq 0} F \simeq F$. Let $i \colon Z \into
\Ash$ be an irreducible component of $N(F)$ of maximal dimension, let $k \geq 1$ be
the biggest integer such that $Z$ is an irreducible component of $\Supp \shH^k F$,
and let $m \geq 1$ be the multiplicity of $\shH^k F$ along $Z$. We
shall now describe how to modify $\Mmod$ in a way that reduces the value of $m$ (and
eventually also of $k$) but leaves the set $N(F) \cup \{Z\}$ invariant. After repeating
this construction a number of times, we can remove $Z$ from the set $N(F)$, possibly
adding linear subvarieties of $Z$ of lower dimension in the process. After finitely
many steps, we thus arrive at $N(F) = \emptyset$, which is the desired outcome.

Since $Z$ is a linear subvariety, we have $Z = t_{(L, \nabla)}(\im \fsh)$,
where $f \colon A \to B$ is a morphism of abelian varieties with connected fibers,
and $(L, \nabla) \in \Ash$; to simplify the notation, we shall assume that $(L,
\nabla) = (\OA, d)$. If we set $r = \dim A - \dim B$, then $\codim Z = 2r$; moreover,
the fact that $Z$ is an irreducible component of $\Supp \shH^k F$ implies that $2r =
\codim Z \geq 2k$, and hence that $r \geq k$. 

Now consider the pullback $\derL \iu F$ to the subvariety $Z$. By construction, the
$k$-th cohomology sheaf of $\derL \iu F$ is supported on all of $Z$, while all higher
cohomology sheaves are torsion. By \theoremref{thm:Laumon}, we have $\derL \iu F
\simeq \FM_B(\fp \Mmod)$, and therefore $\shH^k \FM_B(\fp \Mmod) \neq 0$. If we now
define
\[
	\Nmod = \tau_{\geq k} \fp \Mmod \in \Dth{\lbrack k, r \rbrack}(\Dmod_B),
\]
we obtain a distinguished triangle
\[
	\FM_B(\tau_{\leq k-1} \fp \Mmod) \to \FM_B(\fp \Mmod) \to \FM_B(\Nmod)
		\to \FM_B(\tau_{\leq k-1} \fp \Mmod) \decal{1};
\]
we conclude that $\shH^k \FM_B(\Nmod)$ and $\iu \shH^k F$ are isomorphic at the
generic point of $Z$, and that $\shH^i \FM_B(\Nmod)$ is a torsion sheaf for $i >
k$. 

The adjunction morphism $\Mmod \to \fsi \fp \Mmod$ induces a morphism $\Mmod \to \fsi
\Nmod$. We choose $\Mmod' \in \Dbh(\Dmod_A)$ so as to have a distinguished triangle
\[
	\Mmod' \to \Mmod \to \fsi \Nmod \to \Mmod' \decal{1}.
\]
Since $f$ is smooth of relative dimension $r$, we have $\fsi \Nmod \in \Dth{\lbrack
k-r, 0 \rbrack}(\Dmod_A)$; consequently, $\Mmod' \in \Dth{\leq 1}(\Dmod_A)$. Let $F'
= \FM_A(\Mmod')$; we claim that $N(F') \subseteq N(F)$, and that the multiplicity of
$\shH^k F'$ along $Z$ is strictly smaller than that of $\shH^k F$.

The first part is obvious: by \theoremref{thm:Laumon}, we have $\FM_A(\fsi \Nmod)
\simeq \derR \il \FM_B(\Nmod)$, and so the support of the Fourier-Mukai transform of
$\fsi \Nmod$ is entirely contained in $Z$. Moreover, $\FM_A(\fsi \Nmod)$ belongs to
$\Dtcoh{\geq k}(\OAsh)$; the support of $\shH^k \FM_A(\fsi \Nmod)$ is equal to
$Z$, while all higher cohomology sheaves are supported in proper subvarieties of $Z$.
In particular, $\tau_{\leq 0} F' \simeq \tau_{\leq 0} F$.

To prove the assertion about the multiplicity, observe that we have an exact sequence
\[
	0 \to \shH^k F' \to \shH^k F \to \shH^k \FM_A(\fsi \Nmod).
\]
Now $\shH^k \FM_A(\fsi \Nmod) \simeq \il \shH^k \FM_B(\Nmod)$ is isomorphic to
$\iu \shH^k F$ at the generic point of $Z$. After localizing at the generic point of
$Z$, we can apply \lemmaref{lem:multiplicity} below, which implies that the
multiplicity of $\shH^k F'$ along $Z$ is strictly less than that of $\shH^k F$.

To conclude the reduction step, we now set $\Mmod'' = \tau_{\leq 0} \Mmod'$ and $F''
= \FM_A(\Mmod'')$. Then $\Mmod'' \in \Dth{\leq 0}(\Dmod_A)$, we have $\tau_{\leq 0}
F'' \simeq \tau_{\leq 0} F$, and while $N(F'') \subseteq N(F)$, the multiplicity of
$\shH^k F''$ along $Z$ is strictly smaller than that of $\shH^k F$. As explained
above, this suffices to conclude the proof.
\end{proof}

\begin{lemma} \label{lem:multiplicity}
Let $(R,\mm)$ be a local ring, and let $M$ a nonzero finitely generated $R$-module
with $\mm^n M = 0$ for some $n \geq 1$. If we set $M' = \ker(M \to M/\mm M)$,
then 
\[
	\dim_{R/\mm} M' = \dim_{R/\mm} M - \dim_{R/\mm} M/\mm M < \dim_{R/\mm} M.
\]
In particular, the multiplicity of $M'$ is strictly smaller than that of $M$.
\end{lemma}

\begin{proof}
Nakayama's lemma implies that $M/\mm M \neq 0$.
\end{proof}

Here are several immediate consequences of \theoremref{thm:truncation}.

\begin{corollary}
A complex of coherent sheaves on $\Ash$ is the Fourier-Mukai transform of a holonomic
complex if and only all of its cohomology sheaves are.
\end{corollary}

\begin{proof}
One direction is obvious because of \theoremref{thm:truncation}; the other follows
from the fact that $\Dbh(\Dmod_A)$ is a thick subcategory of $\Dbcoh(\Dmod_A)$.
\end{proof}

\begin{corollary} 
Suppose that $\shF \in \Coh(\OAsh)$ is the Fourier-Mukai transform of a holonomic
complex. Then the same is true for its reflexive hull.
\end{corollary}

\begin{proof}
The reflexive hull $\shHom \bigl( \shHom(\shF, \shO), \shO \bigr)$ is obtained
from $\shF$ by dualizing and truncating twice; both operations preserve the property
of being the Fourier-Mukai transforms of a holonomic complex.
\end{proof}

\begin{corollary}
Suppose that $\shF \in \Coh(\OAsh)$ is the Fourier-Mukai transform
of a holonomic complex. If $\shF$ is reflexive, then there is a holonomic
$\Dmod_A$-module $\Mmod$ such that $\shF \simeq \shH^0 \FM_A(\Mmod)$.
\end{corollary}

\begin{proof}
Let $\Nmod \in \Dth{\leq 0}(\Dmod_A)$ be such that $\shF = \FM_A(\Nmod)$. If we
now define $\Mmod = \shH^0 \Nmod$, then the distinguished triangle
\[
	\FM_A(\tau_{\leq -1} \Nmod) \to \shF \to \FM_A(\Mmod) 
		\to \FM_A(\tau_{\leq -1} \Nmod) \decal{1}
\]
gives us a morphism $\shF \to \shH^0 \FM_A(\Mmod)$. The usual reasoning with
t-structures shows that it is an isomorphism outside a subset of codimension at least
four; since $\shH^0 \FM_A(\Mmod)$ is easily seen to be reflexive, we conclude that
$\shF \simeq \shH^0 \FM_A(\Mmod)$. 
\end{proof}

\section{Simple holonomic D-modules}

This chapter is devoted to a more careful study of Fourier-Mukai transforms of simple
holonomic $\Dmod_A$-modules. In particular, we shall discover that they are
intersection complexes for the $m$-perverse coherent t-structure, in a sense made
precise below.

\subsection{Classification by support}

In this section, we prove a structure theorem for the Fourier-Mukai transform of a
simple holonomic $\Dmod_A$-module. The idea is that, in the case of a simple
holonomic $\Dmod$-module, the support of the Fourier-Mukai transform must be a
single linear subvariety; and if that linear subvariety is not equal to $\Ash$, then
the $\Dmod$-module in question is -- up to tensoring by a line bundle -- pulled back
from a lower-dimensional abelian variety.

\begin{theorem} \label{thm:simple-support}
Let $\Mmod$ be a simple holonomic $\Dmod_A$-module, and let $r \geq 0$ be the least
integer such that $\shH^r \bigl( \FM_A(\Mmod) \bigr) \neq 0$. Then there is an abelian
variety $B$ of dimension $\dim B = \dim A - r$, a surjective morphism $f \colon A \to
B$ with connected fibers, and a simple holonomic $\Dmod_B$-module $\Nmod$, such that
\[	
	\Mmod \tensor_{\OA} (L, \nabla) \simeq \fu \Nmod 
\]
for a suitable point $(L, \nabla) \in \Ash$. Moreover, we have $\Supp \shH^0 \bigl(
\FM_B(\Nmod) \bigr) = \Bsh$ and 
\[
	\Supp \FM_A(\Mmod) = (L, \nabla) \tensor 
		\im \bigl( \fsh \colon \Bsh \to \Ash \bigr).
\]
\end{theorem}

This result clearly implies \theoremref{thm:h-simple} from the introduction. Here is
the proof of the corollary about simple holonomic $\Dmod_A$-modules with Euler
characteristic zero.

\begin{proof}[Proof of \corollaryref{cor:h-simple}]
Let $(L, \nabla) \in \Ash$ be a generic point. Because
\[
	0 = \chi(A, \Mmod) = \chi \bigl( A, \Mmod \tensor_{\OA} (L, \nabla) \bigr)
		= \dim \derH^0 \Bigl( A, \DR_A \bigl( \Mmod \tensor_{\OA} (L, \nabla) \bigr) \Bigr),
\]
we find that the support of $\shH^0 \FM_A(\Mmod)$ is a proper subset of $\Ash$. 
Both $\FM_A(\Mmod)$ and the dual complex belong to $\mDtcoh{\leq 0}(\OAsh)$ by
\theoremref{thm:t-structure}, and so we conclude from \propositionref{prop:surprise} that
$\shH^0 \FM_A(\Mmod) = 0$. Now it only remains to apply
\theoremref{thm:simple-support}.
\end{proof}

For the proof of \theoremref{thm:simple-support}, we need two small lemmas. The
first describes the inverse image of a simple holonomic $\Dmod$-module.

\begin{lemma} \label{lem:inverse-simple}
Let $f \colon A \to B$ be a surjective morphism of abelian varieties, with connected
fibers. If $\Nmod$ is a simple holonomic $\Dmod_B$-module, then $\fu \Nmod$ is a
simple holonomic $\Dmod_A$-module.
\end{lemma}

\begin{proof}
Since $f$ is smooth, $\fu \Nmod = \OA \tensor_{f^{-1} \shO_B} f^{-1} \Nmod$ is a
holonomic $\Dmod_A$-module. According to the classification of simple holonomic
$\Dmod$-modules \cite{HTT}*{Theorem~3.4.2}, there is a locally closed subvariety $X
\subseteq B$, and an irreducible representation $\rho \colon \pi_1(X) \to \GL(V)$,
such that $\Nmod$ is the minimal extension of the integrable connection on $X$
associated to $\rho$.  Now $f$ has connected fibers, and so the map on fundamental
groups
\[
	\fl \colon \pi_1 \bigl( f^{-1}(X) \bigr) \to \pi_1(X)
\]
is surjective. Clearly, the pullback $\fu \Nmod$ is equal, over $f^{-1}(X)$, to the
integrable connection associated to the representation $\rho \circ \fl \colon \pi_1
\bigl( f^{-1}(X) \bigr) \to \pi_1(X) \to \GL(V)$. This representation is still 
irreducible because $\fl$ is surjective; to conclude the proof, we shall argue that
$\fu \Nmod$ is the minimal extension.

By \cite{HTT}*{Theorem~3.4.2}, it suffices to show that $\fu \Nmod$ has no
submodules or quotient modules that are supported outside of $f^{-1}(X)$. Suppose
that $\Mmod \into \fu \Nmod$ is such a submodule. We have $\fsi \Nmod = \fu \Nmod
\decal{r}$, where $r = \dim A - \dim B$; by adjunction, the morphism $\Mmod \into \fu
\Nmod$ corresponds to a morphism $\fp \Mmod \decal{r} \to \Nmod$, which factors
uniquely as
\[
	\fp \Mmod \decal{r} \to \shH^r \fp \Mmod \to \Nmod.
\]
Since $\shH^r \fp \Mmod$ is supported outside of $X$, this morphism must be zero;
consequently, $\Mmod = 0$. A similar result for quotient modules can be derived by
applying the duality functor, using \cite{HTT}*{Theorem~2.7.1}. This shows that $\fu
\Nmod$ is the minimal extension of a simple integrable connection, and hence simple.
\end{proof}

The second lemma deals with restriction to an irreducible component of the support of
a complex.

\begin{lemma} \label{lem:support}
Let $X$ be a scheme, and let $F \in \Dbcoh(\OX)$. Suppose that $Z$ is an irreducible
component of the support of $\shH^r(F)$, but not of any $\shH^i(F)$ with $i > r$. Let
$i \colon Z \into X$ be the inclusion. Then the morphism
\[
	\shH^r(F) \to \shH^r \bigl( \derR \il \derL \iu F \bigr)
\]
induced by adjunction is nonzero at the generic point of $Z$.
\end{lemma}

\begin{proof}
After localizing at the generic point of $Z$, we may assume that $X = \Spec R$ for a
local ring $(R, \mm)$, and that $F \in \Dbcoh(R)$ is represented by a complex
\[
\begin{tikzcd}[column sep=small]
\dotsb \rar & F^{r-2} \rar{d} & F^{r-1} \rar{d} & F^r
\end{tikzcd}
\]
of finitely generated free $R$-modules. Set $M = \shH^r(F) = F^r / d F^{r-1}$, which
is a finitely generated $R$-module with $M \neq 0$. Then $\shH^r \bigl( \derR \il \derL
\iu F \bigr) \simeq M / \mm M$, and the morphism $M \to M / \mm M$ is
nonzero by Nakayama's lemma.
\end{proof}

We can now prove our structure theorem for simple holonomic $\Dmod_A$-modules.%
\footnote{%
The proof can be simplified by using the fact that semi-simplicity is preserved under
projective direct images. This is a very hard theorem by Sabbah and Mochizuki -- see
\cite{Sabbah} for the history of this result -- and so it seemed better to give a
more elementary proof.%
}

\begin{proof}[Proof of \theoremref{thm:simple-support}] 
Let $F = \FM_A(\Mmod) \in \Dbcoh(\OAsh)$. \theoremref{thm:t-structure} shows that $F
\in \mCoh(\OAsh)$; by duality, it follows that $\derR \shHom(F, \OAsh) \in
\mCoh(\OAsh)$, too. According to \propositionref{prop:surprise}, the codimension of
the support of $\shH^r(F)$ is therefore equal to $2r$; moreover, each irreducible
component of $\Supp \shH^r(F)$ is a linear subvariety by \theoremref{thm:FM-linear}.
After tensoring $\Mmod$ by a suitable line bundle with integrable connection, we may
therefore assume that one irreducible component of the support of $\shH^r(F)$ is
equal to $\im \fsh$, for a surjective morphism of abelian varieties $f \colon A \to
B$ with connected fibers and $\dim B = \dim A - r$.

To produce the required simple $\Dmod_B$-module, consider the direct image $\fp
\Mmod$, which belongs to $\Dth{\leq r}(\Dmod_B)$. We have a distinguished triangle
\[
	\tau_{\leq r-1}(\fp \Mmod) \to \fp \Mmod \to \shH^r(\fp \Mmod) \decal{-r} \to 
		\dotsb
\]
in $\Dbh(\Dmod_B)$, and hence also a distinguished triangle
\begin{equation} \label{eq:triangle}
	\fsi \tau_{\leq r-1}(\fp \Mmod) \to \fsi \fp \Mmod \to 
		\fsi \shH^r(\fp \Mmod) \decal{-r} \to \dotsb
\end{equation}
in $\Dbh(\Dmod_A)$. Since $f$ is smooth, $\fsi \shH^r(\fp \Mmod) \decal{-r} = \fu
\shH^r(\fp \Mmod)$ is a single holonomic $\Dmod_A$-module. Let $\alpha \colon \Mmod
\to \fsi \fp \Mmod$ be the adjunction morphism. 

Now we observe that the induced morphism $\Mmod \to \fu \shH^r(\fp \Mmod)$ must be 
nonzero. Indeed, suppose to the contrary that the morphism was zero.  Then $\alpha$
factors as
\[
	\Mmod \to \fsi \tau_{\leq r-1}(\fp \Mmod) \to \fsi \fp \Mmod.
\]
If we apply the Fourier-Mukai transform to this factorization, and use the properties
in \theoremref{thm:Laumon}, we obtain
\[
	F \to \derR \fsh_{\ast} \FM_B \bigl( \tau_{\leq r-1}(\fp \Mmod) \bigr) \to 
		\derR \fsh_{\ast} \derL(\fsh)^{\ast} F,
\]
which is a factorization of the adjunction morphism for the closed embedding $\fsh$.
In particular, we then have
\[
	\shH^r(F) \to \shH^r \Bigl( \derR \fsh_{\ast} \FM_B \bigl( \tau_{\leq r-1}(\fp
		\Mmod) \bigr) \Bigr) 
		\to \shH^r \bigl( \derR \fsh_{\ast} \derL(\fsh)^{\ast} F \bigr);
\]
but because the coherent sheaf in the middle is supported in a subset of $\im \fsh$ of
codimension at least two, this contradicts \lemmaref{lem:support}. Therefore, $\Mmod
\to \fu \shH^r(\fp \Mmod)$ is indeed nonzero. 

Being a holonomic $\Dmod_B$-module, $\shH^r(\fp \Mmod)$ admits a finite filtration
with simple quotients; consequently, we can find a simple holonomic $\Dmod_B$-module
$\Nmod$ and a nonzero morphism $\Mmod \to \fu \Nmod$. Since $\Mmod$ is simple, and
$\fu \Nmod$ is also simple by \lemmaref{lem:inverse-simple}, the morphism must be an
isomorphism, and so $\Mmod \simeq \fu \Nmod$.

To prove the final assertion, note that $\fu \Nmod = \fsi \Nmod \decal{-r}$; on 
account of \theoremref{thm:Laumon}, we therefore have
\[
	\FM_A(\Mmod) \simeq \FM_A(\fu \Nmod) \simeq 
		\derR \fsh_{\ast} \FM_B(\Nmod) \decal{-r} .
\]
Since $\im \fsh$ is an irreducible component of the support of $\shH^r \bigl(
\FM_A(\Mmod) \bigr)$, it follows that $\Supp \shH^0 \bigl( \FM_B(\Nmod) \bigr) =
\Bsh$, as claimed.
\end{proof}

The proof also gives the following surprising improvement of
\theoremref{thm:t-structure} for simple holonomic $\Dmod$-modules.

\begin{corollary} \label{cor:strong}
Let $\Mmod$ be a simple holonomic $\Dmod_A$-module with $\shH^0 \FM_A(\Mmod) \neq 0$.
Then for every $k > 0$, we have $\codim \Supp \shH^k \FM_A(\Mmod) \geq 2k+2$.
\end{corollary}

\begin{proof}
If we had $\codim \Supp \shH^k \FM_A(\Mmod) = 2k$ for some $k > 0$, then the same proof as
above would show that $\Mmod$ is pulled back from an abelian variety of dimension
$g-k$. This possibility is ruled out by our assumption that $\shH^0 \FM_A(\Mmod) \neq
0$.
\end{proof}

\subsection{Rigidity of the Fourier-Mukai transform}

\corollaryref{cor:strong} has a very interesting consequence, namely that the
Fourier-Mukai transform of a simple holonomic $\Dmod_A$-module is completely
determined by a single locally free sheaf: the restriction of the $0$-th cohomology
sheaf to the open subset where it is locally free.

\begin{proposition} \label{prop:IC}
Let $\Mmod$ be a simple holonomic $\Dmod_A$-module, and suppose that the support of
$F = \FM_A(\Mmod)$ is equal to $\Ash$. Then $\shH^0 F$ is a reflexive coherent sheaf,
locally free on the complement of a finite union of linear subvarieties of
codimension $\geq 4$; and $F$ is uniquely determined by the restriction of
$\shH^0 F$ to that open subset.
\end{proposition}

\begin{proof}
We shall use the symbols $\tau_{\leq n}$ and $\tau_{\geq n}$ for the truncation
functors with respect to the standard t-structure on $\Dbcoh(\OAsh)$. 

Let us first show that $\shH^0 F$ is reflexive. We have $F \in \mCoh(\OAsh)$, and
therefore $F \in \Dtcoh{\geq 0}(\OAsh)$, according to \theoremref{thm:t-structure}
and \lemmaref{lem:m-geq}. This means that we can write down a distinguished triangle
\[
	\shH^0 F \to F \to \tau_{\geq 1} F \to (\shH^0 F) \decal{1}, 
\]
and after dualizing, an exact sequence
\[
	R^i \shHom(F, \shO) \to R^i \shHom \bigl( \shH^0 F, \shO \bigr) 
		\to R^{i+1} \shHom \bigl( \tau_{\geq 1} F, \shO \bigr).
\]
Since $\derR \shHom(F, \shO)$ is isomorphic to the Fourier-Mukai transform of the
simple holonomic $\Dmod_A$-module $\langle -1_A \rangle^{\ast} \DA \Mmod$, we get $\codim
\Supp R^i \shHom(F, \shO) \geq 2i+2$ for all $i \geq 1$; on the other hand, we have
$\tau_{\geq 1} F \in \Dtcoh{\geq 1}(\OAsh)$, and \cite{Kashiwara}*{Proposition~4.3} shows that 
\[
	\codim \Supp R^{i+1} \shHom \bigl( \tau_{\geq 1} F, \shO \bigr) \geq i+2.
\]
Together, this proves that $\codim \Supp R^i \shHom \bigl( \shH^0 F, \shO \bigr) \geq
i+2$ for every $i \geq 1$; because $\Ash$ is nonsingular, these inequalities guarantee
that $\shH^0 F$ is reflexive.

Now let $j \colon U \into \Ash$ be the maximal open subset where $\shH^0 F$ is
locally free. Observe that the complement
\[
	\Ash \setminus U = \bigcup_{k \geq 1} \Supp R^k \shHom(F, \OAsh)
		= \bigcup_{k \geq 1} \langle -1_{\Ash} \rangle 
			\Supp \shH^k \FM_A \bigl( \DA \Mmod \bigr)
\]
is a finite union of linear subvarieties of $\Ash$, of codimension at least four.
Because $\shH^0 F$ is reflexive, it is uniquely determined by $\ju \shH^0 F$; in
fact, we have
\[
	\shH^0 F \simeq \jl \ju \shH^0 F.
\]

To prove the remaining assertion, it will be enough to show that, for each $n
\geq 0$, the truncation $\tau_{\leq n} F$ can be reconstructed starting from $\shH^0
F$. We shall argue by induction on $n \geq 0$, using that $\tau_{\leq 0} F = \shH^0
F$. Consider the distinguished triangle
\[
	\tau_{\leq n} F \to F \to \tau_{\geq n+1} F \to (\tau_{\leq n} F) \decal{1}.
\]
According to \corollaryref{cor:strong}, we have $\codim \Supp \shH^i(F) \geq 2i+2
\geq i+n+3$ for every $i \geq n+1$; in combination with \cite{Kashiwara}*{Proposition~4.3},
this implies that $\derR \shHom(\tau_{\geq n+1} F, \shO)$ belongs to $\Dtcoh{\geq
n+3}(\OAsh)$. After dualizing, we conclude that
\[
	\tau_{\leq n+1} \derR \shHom(F, \shO) 
		\simeq \tau_{\leq n+1} \derR \shHom \bigl( \tau_{\leq n} F, \shO \bigr).
\]
The same argument, applied to the dual complex $\derR \shHom(F, \shO)$, shows that
\[
	\tau_{\leq n+2} F \simeq \tau_{\leq n+2} 
		\derR \shHom \Bigl( \tau_{\leq n+1} \derR \shHom(F, \shO), \shO \Bigr).
\]
Together, this says that we always have
\begin{equation} \label{eq:reconstruction-n}
	\tau_{\leq n+2} F \simeq \tau_{\leq n+2} \derR \shHom \Bigl( \tau_{\leq n+1}
		\derR \shHom(\tau_{\leq n} F, \shO), \shO \Bigr),
\end{equation}
thereby concluding the proof.
\end{proof}

The procedure used during the proof leads to the following striking result.

\begin{corollary} \label{cor:IC}
Under the same assumptions as above, $\FM_A(\Mmod)$ can be reconstructed, up to
isomorphism, by applying the functor
\[
	\tau_{\leq \ell-1} \circ \derR \shHom(\argbl, \shO) \circ \dotsb \circ
		\tau_{\leq 2} \circ \derR \shHom(\argbl, \shO) \circ \tau_{\leq 1} \circ 
		\derR \shHom(\argbl, \shO) \circ \jl 
\]
to the locally free sheaf $\ju \shH^0 \FM_A(\Mmod)$; here $\ell$ is any odd
integer with $\ell \geq \dim A$.  
\end{corollary}

\begin{proof}
This follows \eqref{eq:reconstruction-n} and the fact that $F \in \Dtcoh{\leq
g-1}(\OAsh)$.
\end{proof}

\subsection{Intersection complexes}
\label{subsec:IC}

The formula in \corollaryref{cor:IC} for the Fourier-Mukai transform of a simple
holonomic $\Dmod_A$-module is similar to Deligne's formula for the intersection
complex of a local system. The purpose of this
section is to turn that analogy into a rigorous statement.%
\footnote{%
In \cite{AB}*{Section~4}, Arinkin and Bezrukavnikov also define a notion of
``coherent IC-sheaves''. Unfortunately, I do not know how their definition relates to
\definitionref{def:IC}.
}

We shall use the assumptions and notations of \subsecref{subsec:perverse}; in
particular, $X$ will be a smooth complex algebraic variety. We write
$\tau_{\leq n}$ and $\tau_{\geq n}$ for the truncation functors with respect to the
standard t-structure on $\Dbcoh(\OX)$. To simplify the notation, let us introduce the
symbol
\[
	\Delta = \derR \shHom(\argbl, \OX) \colon \Dbcoh(\OX) \to \Dbcoh(\OX)^{\opp}
\]
for the naive duality functor. Define
\[
	\ell(X) = 2 \left\lceil \frac{\dim X + 1}{4} \right\rceil - 1,
\]
which is the smallest odd integer such that $2 \ell + 1 \geq \dim X$.

\begin{definition} \label{def:IC}
Let $\shF$ be a reflexive coherent sheaf. The complex 
\[
	\IC(\shF) = 
	\bigl( \tau_{\leq \ell(X)-1} \circ \Delta \circ \tau_{\leq \ell(X)-2} \circ \Delta \circ 
		\dotsb \circ \tau_{\leq 2} \circ \Delta \circ 
		\tau_{\leq 1} \circ \Delta \bigr) \shF 
\]
will be called the \define{intersection complex} of $\shF$.
\end{definition}

In the definition, $\ell(X)$ may be replaced by any odd integer $\ell$ with the property
that $2 \ell + 1 \geq \dim X$; will see below that this does not change the resulting
complex of coherent sheaves (up to isomorphism). 

When the coherent sheaf $\shF$ is locally free, we of course have $\IC(\shF) \simeq \shF$.
The first result is that $\IC(\shF)$ is always an $m$-perverse coherent sheaf.

\begin{proposition} \label{prop:IC-mCoh}
The complex $\IC(\shF)$ is an $m$-perverse coherent sheaf. Its $0$-th cohomology
sheaf is isomorphic to $\shF$, and 
\[
	\codim \Supp \shH^i \IC(\shF) \geq 2i+1 \quad \text{for every $i \geq 1$.}
\]
The dual complex $\Delta \IC(\shF)$ has the same properties, but $\shH^0 \Delta
\IC(\shF) \simeq \shHom(\shF, \OX)$.
\end{proposition}

\begin{proof}
Our main task is to show that both $\IC(\shF)$ and $\Delta \IC(\shF)$ belong to the
subcategory $\mDtcoh{\leq 0}(\OX)$. We recursively define a sequence of
complexes by setting
\[
	F_n = \tau_{\leq n-1} \Delta F_{n-1},
\]
starting from $F_0 = \shHom(\shF, \OX)$. Observe that $F_1 \simeq \shF$, because
we are assuming that $\shF$ is reflexive. We shall first prove by induction on $n \geq 0$
that
\begin{equation} \label{eq:Fn}
	\text{$F_n \in \Dtcoh{\geq 0}(\OX)$, and 
		$\codim \Supp \shH^i F_n \geq 2i+1$ for every $i \geq 1$.}
\end{equation}
To get started, note that \eqref{eq:Fn} is obviously true for $n = 0$ and $n = 1$:
indeed, both $F_0 = \shHom(\shF, \OX)$ and $F_1 \simeq \shF$ are sheaves. 

Suppose now that we already know the result for all integers between $0$ and $n$.
From the definition of $F_n$, we obtain a distinguished triangle
\[
	F_n \to \Delta F_{n-1} \to \tau_{\geq n} \Delta F_{n-1} \to F_n \decal{1};
\]
after dualizing again, this becomes
\begin{equation} \label{eq:Fn-induction}
	\Delta \tau_{\geq n} \Delta F_{n-1} \to F_{n-1} \to \Delta F_n \to 
		\Bigl( \Delta \tau_{\geq n} \Delta F_{n-1} \Bigr) \decal{1}.
\end{equation}
From \eqref{eq:Fn}, we get $\Delta F_n \in \mhDtcoh{\geq 0}(\OX)$; together with
\lemmaref{lem:m-geq}, this implies that both $\Delta F_n$ and $F_{n+1} = \tau_{\leq
n} \Delta F_n$ are objects of $\Dtcoh{\geq 0}(\OX)$. 

It remains to show that $F_{n+1}$ also lies in the subcategory $\mDtcoh{\leq 0}(\OX)$.
According to \cite{Kashiwara}*{Proposition~4.3}, which we have already used several
times, 
\[
	\codim \Supp \Bigl( \shH^i \Delta \tau_{\geq n} \Delta F_{n-1} \Bigr)
		\geq i+n;
\]
from \eqref{eq:Fn}, we also know that $\codim \Supp \shH^i F_{n-1} \geq 2i+1$ for $i
\geq 1$. The distinguished triangle in \eqref{eq:Fn-induction} gives an exact sequence
\begin{equation} \label{eq:Fn-exact}
	\shH^i F_{n-1} \to \shH^i \Delta F_n 
		\to \shH^{i+1} \Delta \tau_{\geq n} \Delta F_{n-1},
\end{equation}
from which it follows that $\codim \Supp \shH^i \Delta F_n \geq \min(2i+1, i+n+1)
\geq 2i+1$ as long as $1 \leq i \leq n$. This means that $F_{n+1} = \tau_{\leq n}
\Delta F_n$ also satisfies \eqref{eq:Fn}. We have thus proved that \eqref{eq:Fn}
is true for all $n \geq 0$. A useful consequence is that
\[
	\shH^i F_n  = 0 
		\quad \text{once $i > b(X) = \left\lceil \frac{\dim X}{2} \right\rceil - 1$,}
\]
which means that $F_n \in \Dtcoh{\leq b(X)}(\OX)$ for every $n \geq 0$.
	
Next, let us show that the sequence of complexes $F_n$ eventually settles into the
pattern $F, \Delta F, F, \Delta F, \dotsc$. Here we need a bound on the amplitude of
$\Delta F_n$:
\begin{equation} \label{eq:DeltaFn}
	\shH^i \Delta F_n = 0 \quad \text{if $i > b(X)$ and $i > \dim X - n - 1$}
\end{equation}
The proof is again by induction on $n \geq 0$: the statement is obviously true for $n
= 0$ and $n = 1$; for the inductive step, one uses \eqref{eq:Fn-exact}.
Looking back at \eqref{eq:Fn-induction}, we can then say that $\tau_{\geq n}
\Delta F_{n-1} = 0$ as soon as $n > b(X)$ and $n > \dim X - (n-1) - 1$; this
translates into the condition that
\[
	2n-1 \geq \dim X.
\]
If that is the case, we get
\[
	F_{n-1} \simeq \Delta F_n \quad \text{and} \quad
		F_{n+1} = \tau_{\leq n} \Delta F_n \simeq F_{n-1},
\]
and so from that point on, the sequence of complexes alternates between
$F_{n-1}$ and $\Delta F_{n-1}$. Consequently, the intersection complex of $\shF$
satisfies
\[
	\IC(\shF) \simeq F_{\ell} \quad \text{and} \quad \Delta \IC(\shF) \simeq F_{\ell+1}
\]
where $\ell$ is any odd integer such that $2 \ell + 1 \geq \dim X$. Since it is obvious
from the definition that $\shH^0 \IC(\shF)$ is isomorphic to $\shF$, we have
now proved everything that was asserted.
\end{proof}

As in the case of Fourier-Mukai transforms, the fact that $\shF$ is
reflexive means that $\IC(\shF)$ is already determined by the restriction of $\shF$
to an open subset.

\begin{corollary}
Let $j \colon U \into X$ denote the maximal open subset where $\shF$ is locally free.
Then the intersection complex of $\shF$ satisfies
\[
	\IC(\shF) \simeq \IC(\jl \ju \shF),
\]
and is therefore uniquely determined by the locally free sheaf $\ju \shF$.
\end{corollary}

\begin{proof}
The codimension of $X \setminus U$ is $\geq 2$, and so $\jl \ju \shF$ is
isomorphic to $\shF$.
\end{proof}

We can therefore extend the definition of the intersection complex to locally free
sheaves that are defined on the complement of a subset of codimension at least two.

\begin{definition}
Let $j \colon U \into X$ be an open subset with $\codim(X \setminus U) \geq 2$, and
let $\shE$ be a locally free coherent sheaf on $U$. Then the complex
\[
	\IC_X(\shE) = \IC(\jl \shE)
\]
will be called the \define{intersection complex} of $\shE$ (with respect to $X$).
\end{definition}

Our next result is that $\IC_X(\shE)$ actually behaves like an intersection complex:
it does not have nontrivial subobjects or quotient objects whose support is properly
contained in $X$. For technical reasons, the statement is not symmetric.

\begin{proposition}
Let $\shF$ be a torsion-free coherent sheaf on $X$.
\begin{aenumerate}
\item \label{en:IC-a}
If a subobject of $\IC_X(\shE)$ in the abelian category $\mCoh(\OX)$ is supported on a
proper subset of $X$, then that subobject must be zero.
\item \label{en:IC-b}
If a quotient object of $\IC_X(\shE)$ in the abelian category $\mhCoh(\OX)$ is
supported on a proper subset of $X$, then that quotient object must be zero.
\end{aenumerate}
\end{proposition}

\begin{proof}
Since $\Delta$ interchanges the two abelian categories $\mCoh(\OX)$ and
$\mhCoh(\OX)$, it suffices to prove \ref{en:IC-a} for both $\IC_X(\shE)$ and $\Delta
\IC_X(\shE)$. Set $\shF = \jl \shE$. We shall only deal with $\IC(\shF)$; the
argument for the other one is exactly the same.

Suppose then that we have a subobject $E$ of $\IC(\shF)$, with $\Supp E \neq X$; we
need to show that $E = 0$. Because $\mCoh(\OX)$ is defined as the heart of a
t-structure, $E$ being a subobject means that we have a distinguished triangle
\[
	E \to \IC(\shF) \to F \to E \decal{1}
\]
in which $E$ and $F$ are objects of $\mCoh(\OX)$. After dualizing, this becomes
\[
	\Delta F \to \Delta \IC(\shF) \to \Delta E \to \Delta F \decal{1}.
\]
Now $\Delta F \in \mhCoh(\OX)$ has the property that $\codim \Supp
\shH^i \Delta F \geq 2i-1$; we also know from \propositionref{prop:IC-mCoh} that $\codim
\Supp \shH^i \Delta \IC(\shF) \geq 2i+1$ for $i \geq 1$. Looking at the second
distinguished triangle, we find that
\[
	\codim \Supp \shH^i \Delta E \geq 2i+1
\]
for $i \geq 1$; in fact, this also holds for $i = 0$ because $\Supp E \neq X$. This
clearly means that $\Delta E \in \mDtcoh{\geq 0}(\OX)$, too. Now apply
\propositionref{prop:surprise} to get $\Delta E = 0$.
\end{proof}

Note that the distinction between $\mCoh(\OX)$ and $\mhCoh(\OX)$ does not matter if
we only consider objects for which the support of every cohomology sheaf has even
dimension. This is the case for Fourier-Mukai transforms of holonomic complexes. 

\begin{proposition}
Suppose that a reflexive coherent sheaf $\shF$ on $\Ash$ is the Fourier-Mukai
transform of a holonomic complex. Then the same is true for $\IC(\shF)$.
\end{proposition}

\begin{proof}
A look at the formula for the intersection complex in \definitionref{def:IC} shows
that it is obtained by repeatedly dualizing and truncating. By \theoremref{thm:Laumon}
and \theoremref{thm:truncation}, both operations preserve the subcategory $\FM_A
\bigl( \Dbh(\Dmod_A) \bigr)$.
\end{proof}

As an object of the abelian category $\FM_A \bigl( \Dbh(\Dmod_A) \bigr)$, the
intersection complex of $\shF$ now has neither subobjects nor quotient objects that
are supported in a proper subset of $\Ash$ (except for the zero object). This shows
that the analogy with the intersection complex in \cite{BBD} is meaningful.

\section{Miscellaneous results}

This concluding chapter contains a small number of additional results about
Fourier-Mukai transforms of holonomic complexes.

\subsection{A criterion for holonomicity}

In this section, we give a necessary and sufficient condition for a complex of
coherent sheaves on $\Ash$ to be of the form $\FM_A(\Mmod)$ for a holonomic complex
$\Mmod$. Unfortunately, this condition is not very useful in practice: on the one
hand, it is hard to verify; on the other hand, it does not directly imply any of the
properties of Fourier-Mukai transforms of holonomic complexes that we already know.%
\footnote{In a sense, this answers the question raised in
\subsecref{subsec:intro-conjecture} -- but not in a way that is particularly useful.
The problem is that the condition is \propositionref{prop:criterion} is a
\emph{global} one, whereas we are really asking for a \define{local} characterization
of Fourier-Mukai transforms of holonomic complexes.}

\begin{proposition} \label{prop:criterion}
The following conditions on $F \in \Dbcoh(\OAsh)$ are equivalent:
\begin{aenumerate}
\item $F$ is the Fourier-Mukai transform of a holonomic complex.
\item For every $L \in \Pic^0(\Ah)$, the cohomology groups
\[
	\derH^k \bigl( \Ash, F \tensor \piu L \bigr)
\]
are finite-dimensional for every $k \in \ZZ$.
\end{aenumerate}
\end{proposition}

\begin{proof}
Since $\FM_A \colon \Dbcoh(\Dmod_A) \to \Dbcoh(\OAsh)$ is an equivalence of
categories, there is a (essentially unique) complex of $\Dmod$-modules $\Mmod \in
\Dbcoh(\Dmod_A)$ with $F \simeq \FM_A(\Mmod)$. According to
\cite{HTT}*{Theorem~3.3.1}, $\Mmod$ is a holonomic complex if and only if
\[
	\isi_a \Mmod \in \Dbcoh(\CC)
\]
for every point $a \in A$. Let $P_a \in \Pic^0(\Ah)$ denote the line bundle on $\Ah$
corresponding to the point $a \in A$. If $p \colon \Ash \to \pt$ is the morphism to a
point, we have
\[
	\isi_a \Mmod = \derR \pl \bigl( \FM_A(\Mmod) \tensor \piu P_a^{-1} \bigr)
\]
by \theoremref{thm:Laumon}; this immediately implies the asserted equivalence.
\end{proof}

\subsection{Chern characters}

The purpose of this section is to compute the algebraic Chern character of
$\FM_A(\Mmod)$, for $\Mmod$ a holonomic $\Dmod_A$-module. 

For a smooth algebraic variety $X$, we denote by $\CH(X)$ the algebraic Chow ring of
$X$. To begin with, observe that since $\pi \colon \Ash \to \Ah$ is an affine bundle
in the Zariski topology, the pullback map $\pi^{\ast} \colon \CH(\Ah) \to \CH(\Ash)$
is an isomorphism.

\begin{proposition} \label{prop:Chern}
Let $\Mmod$ be a holonomic $\Dmod_A$-module. Then the algebraic Chern character of
the Fourier-Mukai transform $\FM_A(\Mmod)$ lies in the subring of $\CH(\Ah)$
generated by $\CH_1^1(\Ah) = \Pic^0(\Ah)$.
\end{proposition}

\begin{proof}
Since $\pi \colon E(A) \to \Ah$ is an algebraic vector bundle containing $\Ash =
\lambda^{-1}(1)$, pullback of cycles induces isomorphisms
\[
	\CH(\Ah) \simeq \CH \bigl( E(A) \bigr) \simeq \CH(\Ash).
\]
As in the proof of \propositionref{prop:KW}, choose a good filtration $F_{\bullet}
\Mmod$ and consider the Fourier-Mukai transform $\FMt_A(R_F \Mmod)$ of the associated
Rees module. Its restriction to $\Ash$ is isomorphic to $\FM_A(\Mmod)$, and so it
suffices to show that the Chern character of $\FMt_A(R_F \Mmod)$ is contained in the
subring generated by $\Pic^0(\Ah)$. Since $\lambda^{-1}(0) = \Ah \times H^0(A,
\Omega_A^1)$, we only need to prove this after restricting to $\Ah \times \{\omega\} \subseteq
\lambda^{-1}(0)$, for any choice of $\omega \in H^0(A, \Omega_A^1)$.

By \propositionref{prop:compatible}, the restriction of $\FMt_A(R_F \Mmod)$ to
$\lambda^{-1}(0)$ is isomorphic to 
\begin{equation} \label{eq:FM-temp}
	\derR (p_{23})_{\ast} \Bigl( p_{12}^{\ast} P \tensor p_1^{\ast} \Omega_A^g
			\tensor p_{13}^{\ast} \gr^F \!\! \Mmod \Bigr).
\end{equation}
Since $\Mmod$ is holonomic, the support of $\gr^F \!\! \Mmod$ is of pure dimension $g$.
Now choose $\omega \in H^0(A, \Omega_A^1)$ general enough that the restriction of
$\gr^F \!\! \Mmod$ to $A \times \{\omega\}$ is a coherent sheaf with zero-dimensional
support. The restriction of \eqref{eq:FM-temp} to $\Ah \times \{\omega\}$ is then
the Fourier-Mukai transform of a coherent sheaf on $A$ with zero-dimensional support;
its algebraic Chern character must therefore be contained in the subring of
$\CH(\Ah)$ generated by $A = \Pic^0(\Ah)$.
\end{proof}

\begin{corollary} \label{cor:Chern}
Let $\Mmod$ be a holonomic $\Dmod_A$-module. Then all the Chern classes of
$\FM_A(\Mmod)$ are zero in the singular cohomology ring of $\Ash$.
\end{corollary}

\subsection{Ampleness results}

The following result says that, in most cases, the singular locus of a holonomic
$\Dmod$-module is an ample divisor.

\begin{theorem}
Let $\Mmod$ be a holonomic $\Dmod$-module with support equal to $A$. Let
$D(\Mmod)$ be the image in $A$ of the projectivized characteristic variety of
$\Mmod$. If $\chi(A, \Mmod) \neq 0$, then $D(\Mmod)$ is an ample divisor.
\end{theorem}

\begin{proof}
Since the Euler characteristic of any holonomic $\Dmod_A$-module is nonnegative, it
suffices to prove the assertion when $\Mmod$ is simple. If $D(\Mmod)$ is not an ample
divisor, then all of its codimension one components are fibered in a fixed abelian
subvariety, and so there is a morphism $f \colon A \to B$, with $\dim B = \dim A -
r$, such that no codimension one component of $D(\Mmod)$ dominates $B$. On $A
\setminus D(\Mmod)$, we have a vector bundle with integrable connection; its restriction to a
general fiber of $f$ extends to a vector bundle with integrable connection on the
entire fiber. We can then argue as in the proof of \theoremref{thm:simple-support} to show
that $\Mmod \tensor_{\OA} (L, \nabla) \simeq \fu \Nmod$ for a simple holonomic
$\Dmod_B$-module $\Nmod$. But this implies that $\chi(A, \Mmod) = 0$.
\end{proof}

The following result is inspired by the application of Hodge modules and their
Fourier-Mukai transforms to varieties of general type in \cite{PS-Kodaira}. 

\begin{theorem}
Let $\Mmod$ be a holonomic $\Dmod$-module on $A$, and suppose that for some ample
line bundle $L$, there is a nonzero morphism of $\OA$-modules $L \to \Mmod$. 
\begin{aenumerate}
\item If $\Supp \Mmod = A$, then the singular locus $D(\Mmod)$ is an ample divisor.
\item The projection $\Ch(\Mmod) \to H^0(A, \OmA^1)$ is surjective.
\end{aenumerate}
\end{theorem}

\begin{proof}
Both assertions will be proved if we manage to show that $\shH^0 \FM_A(\Mmod) \neq
0$. The given morphism induces a nonzero morphism of $\Dmod_A$-modules
\[
	\Dmod_A \tensor_{\OA} L \to \Mmod,
\]
and therefore a nonzero morphism
\[
	\FM_A \bigl( \Dmod_A \tensor_{\OA} L \bigr) \to \FM_A(\Mmod)
\]
between their Fourier-Mukai transforms. By \theoremref{thm:Laumon}, we have 
\[
	\FM_A \bigl( \Dmod_A \tensor L \bigr) = 
		\derL \piu \derR \Phi_P(L) = \piu \shE_L,
\]
where $\shE_L = (p_2)_{\ast}(P \tensor p_1^{\ast} L)$ is a locally free sheaf of rank
$H^0(A, L)$ on $\Ah$. Because $\FM_A(\Mmod) \in \Dtcoh{\geq 0}(\OAsh)$, the morphism
factors through $\shH^0 \FM_A(\Mmod)$, which is only possible if the
latter is not zero.
\end{proof}

\begingroup
\parindent 0pt
\theendnotes
\endgroup


\section*{References}

\begin{biblist}
\bib{Arapura}{article}{
	author={Arapura, Donu},
	title={Higgs line bundles, Green-Lazarsfeld sets, and maps of K\"ahler manifolds
		to curves},
	journal={Bull. Amer. Math. Soc. (N.S.)},
	volume={26},
	date={1992},	
	number={2},
	pages={310--314},
}
\bib{AB}{article}{
	author={Arinkin, Dima},
	author={Bezrukavnikov, Roman},
	title={Perverse coherent sheaves},
	journal={Mosc. Math. J.},
	volume={10},
	date={2010}, 
	number={1},
	pages={3--29},
}
\bib{BS}{article}{
   author={Bando, Shigetoshi},
   author={Siu, Yum-Tong},
   title={Stable sheaves and Einstein-Hermitian metrics},
   conference={
      title={Geometry and analysis on complex manifolds},
   },
   book={
      publisher={World Sci. Publ., River Edge, NJ},
   },
   date={1994},
   pages={39--50},
}
\bib{BBD}{article}{
   author={Be{\u\i}linson, A. A.},
   author={Bernstein, J.},
   author={Deligne, P.},
   title={Faisceaux pervers},
   language={French},
   conference={
      title={Analysis and topology on singular spaces, I},
      address={Luminy},
      date={1981},
   },
   book={
      series={Ast\'erisque},
      volume={100},
      publisher={Soc. Math. France},
      place={Paris},
   },
   date={1982},
   pages={5--171},
}
\bib{Bonsdorff}{article}{
   author={Bonsdorff, Juhani},
   title={Autodual connection in the Fourier transform of a Higgs bundle},
   journal={Asian J. Math.},
   volume={14},
   date={2010},
   number={2},
   pages={153--173},
}
\bib{Dimca}{book}{
   author={Dimca, Alexandru},
   title={Sheaves in topology},
   series={Universitext},
   publisher={Springer-Verlag},
   place={Berlin},
   date={2004},
   pages={xvi+236},
}
\bib{FK}{article}{
   author={Franecki, J.},
   author={Kapranov, M.},
   title={The Gauss map and a noncompact Riemann-Roch formula for
   constructible sheaves on semiabelian varieties},
   journal={Duke Math. J.},
   volume={104},
   date={2000},
   number={1},
   pages={171--180},
}
\bib{GL1}{article}{
   author={Green, Mark},
   author={Lazarsfeld, Robert},
   title={Deformation theory, generic vanishing theorems, and some
   conjectures of Enriques, Catanese and Beauville},
   journal={Invent. Math.},
   volume={90},
   date={1987},
   number={2},
   pages={389--407},
}
\bib{GL2}{article}{
   author={Green, Mark},
   author={Lazarsfeld, Robert},
   title={Higher obstructions to deforming cohomology groups of line bundles},
   journal={J. Amer. Math. Soc.},
   volume={1},
   date={1991},
   number={4},
   pages={87--103},
}
\bib{Hacon}{article}{
	author={Hacon, Christopher},
	title={A derived category approach to generic vanishing},
	journal={J. Reine Angew. Math.},
	volume={575},
	date={2004},	
	pages={173--187},
}
\bib{HTT}{book}{
   author={Hotta, Ryoshi},
   author={Takeuchi, Kiyoshi},
   author={Tanisaki, Toshiyuki},
   title={$D$-modules, perverse sheaves, and representation theory},
   series={Progress in Mathematics},
   volume={236},
   publisher={Birkh\"auser Boston Inc.},
   place={Boston, MA},
   date={2008},
   pages={xii+407},
}
\bib{Jardim}{article}{
   author={Jardim, Marcos},
   title={Nahm transform and spectral curves for doubly-periodic instantons},
   journal={Comm. Math. Phys.},
   volume={225},
   date={2002},
   number={3},
   pages={639--668},
}
\bib{Kashiwara}{article}{
   author={Kashiwara, Masaki},
	title={t-structures on the derived categories of holonomic $\scr D$-modules and
		coherent $\scr O$-modules},
   journal={Moscow Math. J.},
   volume={4},
   date={2004},
   number={4},
   pages={847--868},
}
\bib{KW}{article}{
	author={Kr\"amer, Thomas},
	author={Weissauer, Rainer},
	title={Vanishing theorems for constructible sheaves on abelian varieties},
	date={2011},
	eprint={arXiv:1111.4947v3},
}
\bib{Laumon}{article}{
   author={Laumon, G\'erard},
   title={Transformation de Fourier g\'en\'eralis\'ee},
   eprint={arXiv:alg-geom/9603004},
	 date={1996}
}
\bib{Malgrange}{article}{
   author={Malgrange, Bernard},
   title={On irregular holonomic $\scr D$-modules},
   conference={
      title={\'El\'ements de la th\'eorie des syst\`emes diff\'erentiels
      g\'eom\'etriques},
   },
   book={
      series={S\'emin. Congr.},
      volume={8},
      publisher={Soc. Math. France},
      place={Paris},
   },
   date={2004},
   pages={391--410},
}
\bib{MM}{book}{
	author={Mazur, Barry},
	author={Messing, William},
	title={Universal Extensions and One Dimensional Crystalline Cohomology},
	series={Lecture Notes in Math.},
   volume={370},
   publisher={Springer-Verlag},
   place={Berlin},
   date={1974},
}
\bib{Mochizuki-B}{article}{
	author={Mochizuki, Takuro},
	title={Holonomic D-module with Betti structure},
	eprint={arXiv:1001.2336v3},
	date={2010},
}
\bib{Mochizuki-wild}{article}{
   author={Mochizuki, Takuro},
   title={Wild harmonic bundles and wild pure twistor $D$-modules},
   journal={Ast\'erisque},
   number={340},
   date={2011},
}
\bib{Mochizuki-N}{article}{
	author={Mochizuki, Takuro},
	title={Asymptotic behaviour and the Nahm transform of doubly periodic instantons
		with square integrable curvature},
	eprint={arXiv:1303.2394},
	date={2013},
}
\bib{Mukai}{article}{
	author={Mukai, Shigeru},
	title={Duality between $D (X)$ and $D(\hat{X})$ with its application to Picard sheaves},
	journal={Nagoya Math. J.},
	volume={81},
	date={1981},
	pages={153--175},
}
\bib{PR}{article}{
   author={Polishchuk, A.},
   author={Rothstein, M.},
   title={Fourier transform for $D$-algebras. I},
   journal={Duke Math. J.},
   volume={109},
   date={2001},
   number={1},
   pages={123--146},
}
\bib{Popa}{article}{
	author={Popa, Mihnea},
	title={Generic vanishing filtrations and perverse objects in derived categories 
		of coherent sheaves},
   conference={
      title={Derived categories in algebraic geometry},
      address={Tokyo},
      date={2011},
   },
   book={
      series={EMS Ser. Congr. Rep.},
      volume={8},
      publisher={Eur. Math. Soc.},
      place={Z\"urich},
	},
	date={2012},	
	pages={251--278},
}
\bib{PS-Kodaira}{article}{
	author={Popa, Mihnea},
	author={Schnell, Christian},
	title={Kodaira dimension and zeros of holomorphic one-forms},
	eprint={arXiv:1212.5714},
	date={2012},
	status={submitted},
}
\bib{PS}{article}{
	author={Popa, Mihnea},
	author={Schnell, Christian},
	title={Generic vanishing theory via mixed Hodge modules},
	journal={Forum of Mathematics, Sigma},
	volume={1},	
	pages={e1},
	year={2013},
	doi={10.1017/fms.2013.1},
}
\bib{Rothstein}{article}{
   author={Rothstein, Mitchell},
   title={Sheaves with connection on abelian varieties},
   journal={Duke Math. J.},
   volume={84},
   date={1996},
   number={3},
   pages={565--598}
}
\bib{Sabbah}{article}{
   author={Sabbah, Claude},
   title={Th\'eorie de Hodge et correspondance de Kobayashi-Hitchin sauvages
		(d'apr\`es T.~Mochizuki)}
   note={S\'eminaire Bourbaki. Vol. 2011/2012.}, 
   date={2013},
   pages={Exp. No. 1050},
}
\bib{Saito-HM}{article}{
   author={Saito, Morihiko},
   title={Modules de Hodge polarisables},
	journal={Publ. Res. Inst. Math. Sci.},
   volume={24},
   date={1988},
   number={6},
   pages={849--995},
}
\bib{Saito-MHM}{article}{
   author={Saito, Morihiko},
   title={Mixed Hodge modules},
	journal={Publ. Res. Inst. Math. Sci.},
   volume={26},
   date={1990},
   number={2},
   pages={221--333},
}
\bib{Saito-MM}{article}{
   author={Saito, Morihiko},
   title={Hodge conjecture and mixed motives. I},
   conference={
      title={Complex geometry and Lie theory},
      address={Sundance, UT},
      date={1989},
   },
   book={
      series={Proc. Sympos. Pure Math.},
      volume={53},
      publisher={Amer. Math. Soc.},
      place={Providence, RI},
   },
   date={1991},
   pages={283--303},
}
\bib{Schnell}{article}{
	author={Schnell, Christian},
	title={A remark about Simpson's standard conjecture},
	date={2013},
	status={submitted},
}
\bib{Simpson}{article}{
   author={Simpson, Carlos},
   title={Subspaces of moduli spaces of rank one local systems},
   journal={Ann. Sci. ENS},
   volume={26},
   date={1993},
   pages={361--401},
}
\bib{Watts}{article}{
	author={Watts, Charles E.},
	title={Intrinsic characterizations of some additive functors}, 
	journal={Proc. Amer. Math. Soc.},
	volume={11},
	date={1960},
	pages={5--8},
}
\bib{Weissauer}{article}{
	author={Weissauer, Rainer},
	title={Degenerate Perverse Sheaves on Abelian Varieties},
	date={2012},
	eprint={arXiv:1204.2247},
}
\end{biblist}
\end{document}